\documentclass[12pt]{amsart}
\usepackage{fullpage}
\usepackage{graphicx}
\usepackage{todonotes}
\usepackage{tikz}
\usetikzlibrary{decorations.markings}

\newtheorem{theorem}{Theorem}[section]
\newtheorem{corollary}[theorem]{Corollary}
\newtheorem{lemma}[theorem]{Lemma}
\newtheorem{proposition}[theorem]{Proposition}

\numberwithin{equation}{section}

\newcommand\CC {{\mathbb C}}

\newcommand\EE {{\mathbb E}}

\newcommand\NN {{\mathbb N}}
\newcommand\PP {{\mathbb P}}
\newcommand\QQ {{\mathbb Q}}
\newcommand\RR {{\mathbb R}}
\newcommand\TT {{\mathbb T}}
\newcommand\ZZ {{\mathbb Z}}

\newcommand\gltwor{{\rm GL(2,\RR)}}
\newcommand\sltwor{{\rm SL(2,\RR)}}

\newcommand\sltwoz{{\rm SL(2,\ZZ)}}

\newcommand\slgroup{{\rm SL}}

\newcommand\hol {{\rm Hol}}

\newcommand\area{{\rm Area }}
\newcommand\leb{{\rm Leb }}

\newcommand\re{{\rm Re }}
\newcommand\im{{\rm Im }}
\newcommand\distance{{\rm Dist }}

\newcommand\id{{\rm Id}}

\newcommand\wrecsup{\overline{w}_{\mathrm{ret}}}
\newcommand\whitsup{\overline{w}_{\mathrm{hit}}}

\newcommand\wrecinf{\underline{w}_{\mathrm{ret}}}
\newcommand\whitinf{\underline{w}_{\mathrm{hit}}}

\newcommand\rt{R}

\newcommand\cA{{\mathcal{A}  }}
\newcommand\cB{{\mathcal{B}  }}

\newcommand\cE{{\mathcal{E}  }}
\newcommand\cF{{\mathcal{F}  }}
\newcommand\cG{{\mathcal{G}  }}
\newcommand\cH{{\mathcal{H}  }}

\newcommand\cM{{\mathcal{M}  }}
\newcommand\cN{{\mathcal{N}  }}
\newcommand\cO{{\mathcal{O}  }}
\newcommand\cP{{\mathcal{P}  }}
\newcommand\cQ{{\mathcal{Q}  }}
\newcommand\cR{{\mathcal{R}  }}

\newcommand\cV{{\mathcal{V}  }}

\begin{document}

\title[Long Hitting time]{Long Hitting time for translation flows and L-shaped billiards}

\author[D.H. Kim]{Dong Han Kim}
\address{Department of Mathematics Education, Dongguk University -- Seoul, 30 Pildong-ro 1-gil, Jung-gu, Seoul, 04620 Korea}

\email{kim2010@dongguk.edu}

\author[L. Marchese]{Luca Marchese}
\address{Universit\'e Paris 13, Sorbonne Paris Cit\'e,
LAGA, UMR 7539, 99 Avenue Jean-Baptiste Cl\'ement, 93430 Villetaneuse, France.}

\email{marchese@math.univ-paris13.fr}

\author[S. Marmi]{Stefano Marmi}
\address{Scuola Normale
Superiore and  C.N.R.S. UMI 3483 Laboratorio Fibonacci, Piazza dei Cavalieri 7, 56126 Pisa, Italy}

\email{stefano.marmi@sns.it}


\maketitle

\begin{abstract}
We consider the flow in direction $\theta$ on a translation surface and we study the asymptotic behavior for $r\to 0$ of the time needed by orbits to hit the $r$-neighborhood of a prescribed point, or more precisely the exponent of the corresponding power law, which is known as \emph{hitting time}. For flat tori the limsup of hitting time is equal to the diophantine type of the direction $\theta$. In higher genus, we consider a generalized geometric notion of diophantine type of a direction $\theta$ and we seek for relations with hitting time. For genus two surfaces with just one conical singularity we prove that the limsup of hitting time is always less or equal to the square of the diophantine type. For any square-tiled surface with the same topology the diophantine type itself is a lower bound, and any value between the two bounds can be realized, moreover this holds also for a larger class of origamis satisfying a specific topological assumption. Finally, for the so-called \emph{Eierlegende Wollmilchsau} origami, the equality between limsup of hitting time and diophantine type subsists. Our results apply to L-shaped billiards.
\end{abstract}

\maketitle


\section{Introduction}
\label{Intro}

Consider a minimal dynamical system 
$
\phi:\RR\times X\to X,\,(t,p)\mapsto \phi^t(p)
$ 
in continuous time $t\in\RR$ on a metric space $X$, whose balls are the sets 
$
B(p,r):=\{p' \in X \mid  \distance(p,p')<r\}
$ 
with $p\in X$ and $r>0$. 
For any pair of points $p,p'$ in $X$ and any $r>0$ small enough, the \emph{hitting time} of $p$ to the ball $B(p',r)$ of radius $r$ around $p'$ is
$$
\rt(\phi,p,p',r):=\inf\{t>r \mid \distance(\phi^t(p),p')<r\}
$$ 
We are interested to study the scaling law of $\rt(\phi,p,p',r)$ when $r\to 0$, that is to say we consider the two quantities
$$
\whitinf(\phi,p,p'):=
\liminf_{r \to 0^+} \frac{\log \rt(\phi,p,p',r)}{-\log r},
\quad
\whitsup(\phi,p,p'):=
\limsup_{r \to 0^+} \frac{\log \rt(\phi,p,p',r)}{-\log r}.
$$
One can consider also the \emph{return time} $\rt(\phi,p,r):=\rt(\phi,p,p,r)$ of a point $p$ to its $r$-ball $B(p,r)$ and 
and define analogously 
$$ 
\wrecinf(\phi,p):=
\whitinf(\phi,p,p) 
\quad 
\text{ and } 
\quad 
\wrecsup(\phi,p):= 
\whitsup(\phi,p,p). 
$$ 
The same quantities are also obviously defined for dynamical systems in integer time $t\in\ZZ$. 
In particular, in \cite{KimSeo} it is proved that for the irrational rotation by parameter $\alpha$, that is for the map
$
T_\alpha:[0,1)\to[0,1)
$, 
$
x\mapsto x+\alpha \pmod\ZZ
$, 
we have 
\begin{equation}
\label{EquationTheoremDongHanRotations}
\whitinf(T_\alpha,p,p')=1
\quad \text{ and } \quad 
\whitsup(T_\alpha,p,p')=w(\alpha) 
\quad 
\textrm{ for a.e. } \ p,p'\in [0,1),
\end{equation}
where $w(\alpha)\geq 1$ denotes the \emph{diophantine type} of the irrational number $\alpha$, that is the supremum of those $w\geq 1$ such that there exist infinitely many rational numbers $p/q$ such that
\begin{equation}
\label{EquationDiophantineTypeClassical}
\left| \alpha - \frac pq \right| < \frac{1}{q^{w+1}} 
\quad \text{ equivalently } \quad 
\left| q\alpha - p \right| < \frac{1}{q^w}.  
\end{equation}
Moreover it is easy to check (See \cite{ChoeSeo}) that 
\begin{equation}\label{eq:1.3}
\wrecinf(T_\alpha,p) =  \frac{1}{w(\alpha)} 
\quad
\text{ and }
\quad
\wrecsup(T_\alpha,p)= 1
\quad 
\textrm{ for any} 
\quad
\ p\in [0,1).
\end{equation}

In particular, since $w(\alpha)=1$ for almost any $\alpha$, then for any such $\alpha$ and for generic points $p,p'$ one has 
$$
\lim_{r\to 0_+} 
\frac{\log R(T_\alpha,p,r)}{-\log r} = 1
\quad 
\textrm{ and }
\quad
\lim_{r\to 0_+} 
\frac{\log R(T_\alpha,p,p',r)}{-\log r} = 1
$$
In \cite{KimMarmi}, the same result was shown replacing the generic rotation $T_\alpha$ by the generic \emph{interval exchange map} $T$ over any number of intervals. On the other hand, Equation \eqref{EquationTheoremDongHanRotations} also implies the existence of many parameters $\alpha$ with $\whitsup(T_\alpha,p,p')>1$ for generic $p,p'$, indeed according to Jarn\'ik Theorem (see \cite{Jarnik}), for any $\eta$ 
$$
\dim_H \big( \{\alpha\in\RR \mid w(\alpha)=\eta\} \big)=\frac{2}{1+\eta},
$$
where $\dim_H$ denotes the Hausdorff dimension. The result in Equation \eqref{EquationTheoremDongHanRotations} stands in the same form replacing $T_\alpha$ by the linear flow $\phi_\theta$ in direction $\theta$ with slope $\alpha=\tan\theta$ on the standard torus $\TT^2:=\RR^2/\ZZ^2$, that is the flow defined for any $t\in\RR$ by  
$$
\phi_\theta^t(x,y):=(x,y)+t(\sin\theta,\cos\theta)\pmod{\ZZ^2}.
$$
Moreover, via a well known \emph{unfolding procedure}, one can derive the same conclusion for the \emph{billiard flow} in the square $[0,1/2]^2$. The aim of this paper is to study the relation between hitting time and the diophantine exponent in translation flows in higher genus. 
Definitions and statements of main results are given in \S~\ref{SectionDefinitionsResults}. Here, as a motivational result, we state a consequence of our main Theorems for the billiard flow in L-shaped polygons.

\smallskip

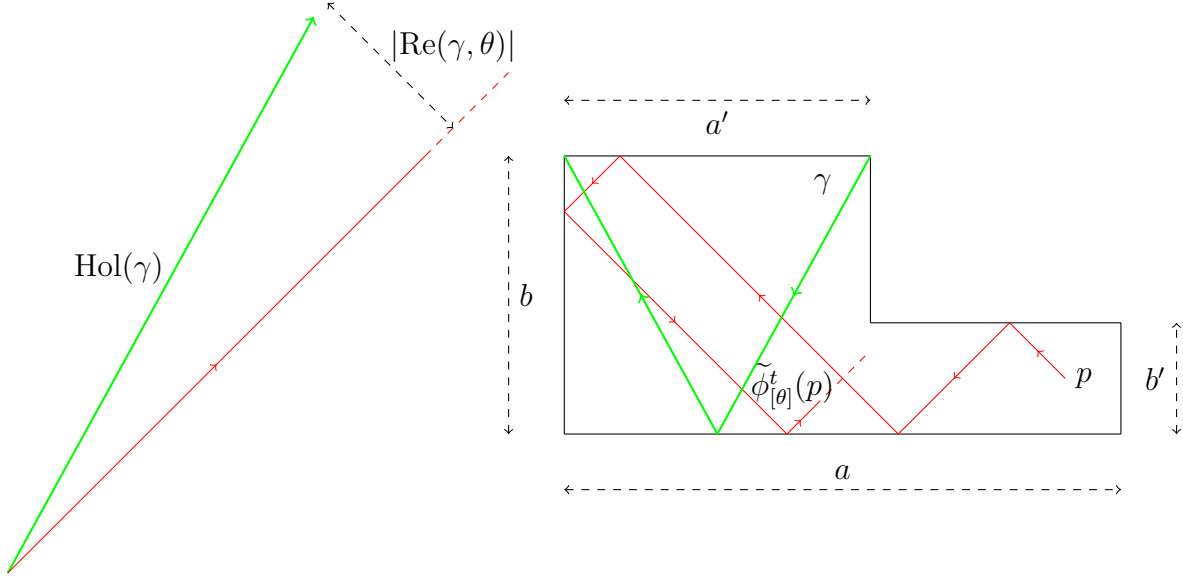
\begin{figure}[ht]
\begin{center} 
\begin{tikzpicture}[scale=0.37]

\tikzset
{->-/.style={decoration={markings,mark=at position .5 with {\arrow{>}}},postaction={decorate}}}


\draw[-] (0,0) -- (20,0);
\draw[-] (20,0) -- (20,4);
\draw[-] (20,4) -- (11,4);
\draw[-] (11,4) -- (11,10);
\draw[-] (11,10) -- (0,10);
\draw[-] (0,10) -- (0,0);


\node[] at (18,2) [right] {$p$};
\draw[->-,red] (18,2) -- (16,4);
\draw[->-,red] (16,4) -- (12,0);
\draw[->-,red] (12,0) -- (2,10);
\draw[->-,red] (2,10) -- (0,8);
\draw[->-,red] (0,8) -- (8,0);
\draw[->-,red] (8,0) -- (9,1);
\draw[-,red,dashed] (9,1) -- (11,3);
\node[] at (8.2,1.7) {$\widetilde{\phi}^t_{[\theta]}(p)$};


\draw[->-,thick,green] (11,10) -- (5.5,0);
\draw[->-,thick,green] (5.5,0) -- (0,10);
\node[] at (10,9) [left] {$\gamma$};

\draw[->,thick,green] (-20,-5) -- (-9,15);
\node[] at (-16,5) [left,above] {$\hol(\gamma)$};

\draw[<->,dashed] (-8.5,15.5) -- (-4,11);
\node[] at (-4,13) [right,above] {$|\re(\gamma,\theta)|$};

\draw[->-,red] (-20,-5) -- (-5,10);
\draw[-,red,dashed] (-5,10) -- (-2,13);


\draw[<->,thin,dashed] (0,-2) -- (20,-2);
\node[] at (10,-2) [above] {$a$};

\draw[<->,thin,dashed] (0,12) -- (11,12);
\node[] at (5.5,12) [below] {$a'$};

\draw[<->,thin,dashed] (-2,0) -- (-2,10);
\node[] at (-2,5) [right] {$b$};

\draw[<->,thin,dashed] (22,0) -- (22,4);
\node[] at (22,2) [left] {$b'$};

\end{tikzpicture}

\end{center}
\caption{In red, a billiard trajectory in a L-shaped billiard $L(a,b,a',b')$. The green path $\gamma$ is a generalized diagonal, the vector on the left of the figure its planar development $\hol(\gamma)$. The deviation of the direction $\theta$ from $\gamma$ is $|\re(\gamma,\theta)|$.}
\label{FigureLShapedBilliard}
\end{figure}

Fix four real numbers $0<a'<a$ and $0<b'<b$ and let $L=L(a,a',b,b')\subset\RR^2$ be the polygon whose vertices, listed in counterclockwise order, are $(0,0)$, $(a,0)$, $(a,b')$, $(a',b')$, $(a',b)$, $(0,b)$. 
An example of such polygon appears in Figure~\ref{FigureLShapedBilliard}. 
Let 
$
\widetilde{\phi}:L\times S^1\to L\times S^1
$ 
be the billiard flow. Reflections at sides of $L$ are affine maps with linear part given by the linear reflections $s_h:\RR^2\to\RR^2$ and $s_v:\RR^2\to\RR^2$ defined respectively by 
$$
s_h(x,y):=(-x,y)
\quad
\textrm{ and }
\quad
s_v(x,y):=(x,-y).
$$
Let $D \simeq (\ZZ/2\ZZ)^2$ be the group generated by $s_h$ and $s_v$ and consider its action on $S^1$. 
Any direction $\theta$ has an orbit $[\theta]$ of four elements, thus the phase space $L\times S^1$ decomposes into invariant subspaces $L\times[\theta]$ where the directional billiard flow 
$
\widetilde{\phi}_{[\theta]}
$ 
is defined, modulo the action of $D$ on the second factor. 
A \emph{generalized diagonal} $\gamma$ is a finite segment of billiard trajectory connecting two vertices of $L$ and without any other vertex in its interior. A direction $\theta$ is said rational if 
$
\widetilde{\phi}_{[\theta]}
$ 
admits a generalized diagonal, the set of rational directions being of course countable. We give a preliminary version of the notion of \emph{diophantine type} of a non rational direction $\theta$ in terms of the deviation of finite trajectories of 
$
\widetilde{\phi}_{[\theta]}
$ 
from generalized diagonals.

Let $\gamma$ be a generalized diagonal, parametrized as a continuous piecewise smooth path 
$
\gamma:[0,T]\to L
$ 
with unitary speed, that is $|d\gamma(t)/dt|=1$ when it is defined. Consider the sequence of instants  $0=t_0<t_1<\dots<t_N =T$ such that for any $i=1,\dots,N$ the restricted path 
$
\gamma:[t_{i-1},t_i]\to L
$ 
is a straight segment with 
$
\gamma(t_{i-1})\in \partial L
$, 
$
\gamma(t_{i})\in \partial L
$ 
and $\gamma(t)$ in the interior of $L$ for $t_{i-1}<t<t_i$. There exist an unitary vector 
$
v\in\RR_+\times\RR_+
$ 
in the first quadrant, depending only on $\gamma$, such that for any $i=1,\dots,N$ there exists an unique $s_i\in D$ with  
$$
s_i\left(\frac{d}{dt}\gamma(t)\right)=v
\quad
\textrm{ for any }
\quad
t_{i-1}<t<t_i.
$$
The \emph{planar development} of $\gamma$ is the vector 
$
\hol(\gamma):=T\cdot v\in\RR_+\times\RR_+
$. 
Consider a non rational direction $\theta$ on the L-shaped polygon $L$, and modulo replacing $\theta$ by an element in its orbit $[\theta]$, assume that $0\leq\theta\leq\pi/2$, where $\theta=0$ corresponds to the vertical direction. The diophantine type $w_L(\theta)$ of $\theta$ on the L-shaped polygon $L$ is the supremum of those $w\geq1$ such that there exist infinitely many generalized diagonal $\gamma$ with
$$
|\re(\theta,\gamma)|<\frac{1}{|\im(\gamma,\theta)|^{w}},
$$
where 
$
\re(\gamma,\theta):=\langle \hol(\gamma),(\cos\theta,-\sin\theta)\rangle
$ 
and 
$
\im(\gamma,\theta):=\langle \hol(\gamma),(\sin\theta,\cos\theta)\rangle
$.

\smallskip
  
The notion of generalized diagonal can be given also for the billiard in the square $[0,1/2]^2$, which has the same reflection group $D$ as any L-shaped polygon $L(a,a',b,b')$, thus the diophantine type of a direction $\theta$ can be defined in the same way. 
It is an exercise to check that such alternative notion of diophantine type in fact coincides with the quantity $w(\alpha)$ defined by Equation \eqref{EquationDiophantineTypeClassical}, modulo the change of variable $\alpha=\tan\theta$. 
Moreover the two notions also coincide for a class of L-shaped billiards $L=L(a,a',b,b')$. More precisely, according to Equation \eqref{EquationEquivalenceTypes} and to the discussion in \S~\ref{SectionBackToBilliards}, we have 
$
w_L(\theta)=w(\tan\theta)
$ 
whenever $a,a',b,b'$ are rationally dependent. The next Theorem is a direct consequence (derived in 
\S~\ref{SectionBackToBilliards}) of our main results, namely Theorem \ref{TheoremUpperBoundHittingTime} and Theorem \ref{TheoremHittingSpectrumOrigami}, stated in \S~\ref{SectionTheoremsHittingTime} below. It is practical to introduce the function 
$
f_\eta:[1,2]\to \RR_+
$ 
defined for any $s\in[1,2]$ by 
\begin{equation}
\label{EquationLowerBoundDimension}
f_\eta(s):=
\left\{
\begin{array}{ll}
\displaystyle{\frac{1}{\eta+1}}
&
\quad
\textrm{ if }
\quad
s=1
\\
\displaystyle{\frac{\eta^{s-1}-1}{\eta^s-1}}
&
\quad
\textrm{ if }
\quad
1<s\leq 2.
\end{array}
\right.
\end{equation}

\begin{theorem}
\label{TheoremLShapedBilliard}
Let $L=L(a,a',b,b')$ be the L-shaped billiard with parameter $a,a',b,b'$ and consider any non rational direction $\theta$ on $L$. 
\begin{enumerate}
\item
For any pair of points $p,p'$ in $L\times[\theta]$ we have 
$$
\whitsup(\widetilde{\phi}_{[\theta]},p,p')\leq w_L(\theta)^2.
$$
\end{enumerate}
Assume now that $a,a',b,b'$ are rationally dependent, so that 
$
w_L(\theta)=w(\tan\theta)
$ 
for any $\theta$. 
\begin{enumerate}
\setcounter{enumi}{1}
\item
For almost any $p,p'$ in $L\times[\theta]$ we have 
$$
\whitsup(\widetilde{\phi}_{[\theta]},p,p')\geq w(\tan\theta).
$$
\item
Finally, for any $\eta>1$ and any $s$ with $1\leq s\leq 2$ there exists a set of directions $\EE(L,\eta,s)$ with
$$
\dim_H\big(\EE(L,\eta,s)\big)\geq f_\eta(s)
$$
such that for any $\theta\in\EE(L,\eta,s)$ we have $w(\tan\theta)=\eta$ and moreover for almost any $p,p'$ in $L\times[\theta]$ we have
$$
\whitsup(\widetilde{\phi}_{[\theta]},p,p')=w(\tan\theta)^s=\eta^s.
$$
\end{enumerate}
\end{theorem}

When the Lebesgue measure is ergodic for $\widetilde{\phi}_{[\theta]}$, we also show that 
$
\whitinf(\widetilde{\phi}_{[\theta]},p,p')=1
$ 
for Lebesgue almost any $p,p'$ in $L\times[\theta]$. For recurrence times our results simply extend those obtained for rotations. More precisely, for any L-shaped billiard $L$ and for almost any $p\in L\times[\theta]$ we have $\wrecsup(\widetilde{\phi}_{[\theta]},p)=1$. Moreover, when $a,a',b,b'$ are rationally dependent we have 
$
\wrecinf(\widetilde{\phi}_{[\theta]},p)=1/w(\tan\theta)
$ 
for almost any $p\in L\times[\theta]$. For more details see \S~\ref{SectionBackToBilliards}.

\subsection*{Acknowledgements} 

The authors are grateful to J. Chaika, V. Delecroix, P. Hubert and S. Leli\`evre. 
This research has been supported by the following institutions: CNRS, FSMP, National Research Foundation of Korea(NRF-2015R1A2A2A01007090), Scuola Normale Superiore, UnicreditBank.

\section{Definitions and statement of main results}
\label{SectionDefinitionsResults}

\subsection{Translation surfaces}

Let $P\subset\RR^2$ be a polygon in the plane. The polygon $P$ is not necessarily connected, that is we allow $P$ to be the disjoint union of finitely many connected polygons $P_1,\dots,P_l$. We assume that $\partial P$ is union of $2d\geq 4$ segments which come in pairs and are denoted  
$
(\zeta_1,\zeta'_1),\dots,(\zeta_d,\zeta'_d)
$, 
and that there exist vectors $z_1,\dots,z_d$ in $\RR^2$ such that for any $i=1,\dots,d$ the boundary segments $\zeta_i$ and $\zeta'_i$ have the same direction and length of $z_i$, and the opposite orientation induced by the interior of $P$ (that is any $\zeta_i$ touches the interior of $P$ from the opposite side as $\zeta_i'$). A \emph{translation surface} is the quotient space $X=P/\sim$ obtained identifying for any $i=1,\dots,d$ the sides $\zeta_i$ et $\zeta'_i$ by a translation. We assume that the identification gives a connected quotient space $X$. If $P$ is a parallelogram then $X=P/\sim$ is a flat torus. In general a translation surface is a compact surface of genus $g\geq1$, with a metric which is flat outside of a finite set $\Sigma$ of points $p_1,\dots,p_r$ of $X$, where the metric has a conical singularity with angle $2(k_j+1)\pi$ and $k_j\in\NN$. Any $p_j$ corresponds to a subset of the vertices of $P$, all identified to the same point in $X$ by the equivalence relation $\sim$ on $\partial P$. We have 
$
k_1+\dots+k_r=2g-2
$.
For a general overview on the subject we recommend the surveys \cite{ForniMatheus} and \cite{Zorich}.

\subsubsection{Dynamics in moduli spaces}

A \emph{stratum} $\cH=\cH(k_1,\dots,k_r)$ is the set of all translation surfaces $X$ with the same order of conical singularities $k_1,\dots,k_r$. Any stratum is an affine \emph{orbifold}, the affine coordinates around some $X\in\cH$ being the vectors $z_1,\dots,z_d$ defined above, possibly modulo some linear equations with coefficients in $\QQ$. Consider any $G\in\gltwor$ and any translation surface $X=P/\sim$, represented via the polygon $P$. We define $G\cdot X$ as the quotient space $G\cdot P/\sim$, where $G\cdot P$ is the affine image of $P$ under the action of $G$ on $\RR^2$ and where the identifications in $\partial (G\cdot P)$ have the same combinatorics of those between the sides of $P$. This gives an action of $\gltwor$ on any stratum $\cH$. We will consider the subgroup action of rotations and diagonal elements, thus for $t\in\RR$ and $0\leq \theta<2\pi$ we set
$$
g_t:=
\begin{pmatrix}
e^t & 0 \\
0   & e^{-t}
\end{pmatrix}
\quad
\textrm{ and }
\quad
r_\theta:=
\begin{pmatrix}
\cos\theta & -\sin\theta \\
\sin\theta   & \cos\theta
\end{pmatrix}.
$$

According to the celebrated results of Eskin, Mirzakani \cite{EskinMirzakani} and Eskin, Mirzakaniand Mohammadi \cite{EskinMirzakaniMohammadi}, any subset $\cM$ on $\cH$ which is closed and invariant under $\gltwor$, is a sub-orbifold defined by linear equations in the coordinates $z_1,\dots,z_d$. Particularly simple closed invariant sets are closed orbits. It is known that the $\gltwor$-orbit of a surface $X\in\cH$ is closed in $\cH$ if and only if the stabiliser $\slgroup(X)$ of $X$ under the action of $\gltwor$, called the \emph{Veech group} of $X$, is a lattice in $\sltwor$ (see \S~5 of \cite{SmillieWeiss} for a proof). Such a surface is called \emph{Veech surface} and its orbit is an isometric image of $\gltwor/\slgroup(X)$ embedded in $\cH$, and is locally defined by a system of linear equations in the coordinates $z_1,\dots,z_d$, with real rank four. In general the Veech group of any translation surface $X$ is a discrete and non co-compact subgroup of $\sltwor$, and it is trivial for generic $X$.

\subsubsection{Saddle connections, cylinders and planar developments}

Let $|\cdot|$ denote the Euclidian norm on $\RR^2$, and recall that the flat metric of a translation surface $X$ is locally isometric to the Euclidian metric of $\RR^2$. A \emph{saddle connection} of a translation surface $X$ is a segment $\gamma$ of a geodesic for the flat metric of $X$ connecting two conical singularities $p_i$ and $p_j$ and not containing other conical singularities in its interior. We consider also \emph{closed geodesics} $\sigma$ of $X$. For any such $\sigma$ there exists a family of closed geodesics which are parallel to $\sigma$ with the same length and the same orientation. A \emph{cylinder} for $X$ is a maximal connected open set $C_\sigma$ foliated by such a family of parallel closed geodesics. By maximality, the boundary of a cylinder $C_\sigma$ around a closed geodesic $\sigma$ is union of saddle connections parallel to $\sigma$. The \emph{transversal width} $W(C_\sigma)$ of $C_\sigma$ is the length of a segment orthogonal to $\sigma$ which connects the two components of the boundary $\partial C_\sigma$, so that in particular 
$
\area(C_\sigma)=W(C_\sigma)\cdot |\sigma|
$.

Let $\gamma\subset X$ be a finite segment of a geodesic for the flat metric of $X$, for example either a saddle connection or a closed geodesic, and abusing the notation denote with the same symbol also a smooth parametrization of it $\gamma:[0,T]\to X$, $t\mapsto\gamma(t)$. There exists a vector $v\in\RR^2$ with $|v|=1$ such that $d\gamma(t)/dt=v$ for any $0<t<T$, and the segment $\gamma$ has a planar development in the plane, denoted by $\hol(\gamma,X)\in\RR^2$ and defined by 
$$
\hol(\gamma,X):=T\cdot v.
$$ 
Any such $\gamma$ is a geodesic segment also on the surface $G\cdot X$ for any $G\in\gltwor$, and we denote by $\hol(\gamma,G\cdot X)$ its holonomy with respect to the surface $G\cdot X$. We have
\begin{equation}
\label{EquationCovarianceHolonomy}
\hol(\gamma,G\cdot X)=G\big(\hol(\gamma,X)\big).
\end{equation}
The set $\hol(X)$ of \emph{relative periods} of $X$ is the set of vectors 
$
\hol(\gamma,X)\in\RR^2
$, 
where $\gamma$ is a saddle connection for $X$.

\subsubsection{Phase space dynamics}
\label{SectionPhaseSpaceDynamics}

Fix a translation surface $X$ and a direction $\theta$. One can define a constant vector field on $X\setminus\Sigma$, whose value at any point is equal to to the unitary vector 
$$
e_\theta:=(\sin\theta,\cos\theta),
$$
then consider the integral flow of such field, that is denoted by $\phi_\theta$. Orbits of $\phi_\theta$ are parallel lines in direction $\theta$ which wind on $X$. They are defined for any $t\in\RR$, outside the set of $2g-2+r$ leaves starting or ending at singular points $p_1,\dots,p_r$, which we call $(X,\theta)$-\emph{singular leaves}, or simply \emph{singular leaves}, when there is no ambiguity on the surface $X$ and the direction $\theta$. A direction $\theta$ on a translation surface $X$ is \emph{completely periodic} if every $(X,\theta)$-singular leaf extends to a saddle connection. In this case, the set of saddle connections in direction $\theta$ separate the surface into a finite number of cylinders, any of which is foliated by periodic orbits of the linear flow $\phi_\theta$. A direction $\theta$ is a \emph{Keane direction} for the translation surface $X$ if there is no saddle connection $\gamma$ in direction $\theta$. Obviously, all but countably many directions are Keane. Moreover, if $\theta$ is a Keane direction on $X$, then the flow $\phi_\theta$ is minimal, that is any infinite orbit is dense, both in the past and in the future (for a proof see Corollary 5.4 in \cite{Yoccoz}). According to Veech (see \cite{Veech}), for Veech surfaces there is a sharp dynamical dichotomy between Keane directions and directions of saddle connections. More precisely, on a Veech surface $X$ the flow $\phi_\theta$ is uniquely ergodic whenever $\theta$ is Keane, otherwise $\theta$ is a completely periodic direction, moreover there are two constants $a=a(X)>0$ and $C=C(X)>0$ depending only on $X$ such that for any saddle connection $\gamma$ and closed geodesic $\sigma$ in direction $\theta$ we have 
$
C^{-1}|\sigma|\leq |\gamma|\leq C|\sigma|
$ 
and $\area(C_\sigma)>a$, where $C_\sigma$ is the cylinder around $\sigma$.

\subsubsection{Origamis}
\label{SectionIntroductionOrigamis}

\emph{Origamis}, which are also know as \emph{square-tiled surfaces}, form a special class of translation surfaces. An origami is a translation surface $X$ tiled by copies of the square $[0,1]^2$. It is a direct consequence of definitions that $X$ is an origami if and only if $\hol(X)\subset\ZZ^2$ and the last condition is also equivalent to the existence of a ramified covering $\rho:X\to\TT^2$ of the standard torus such that the following conditions are satisfied:
\begin{enumerate}
\item
The covering is ramified only over the origin $[0]\in\TT^2$, where $[0]$ denotes the coset of $0$ in $\RR^2/\ZZ^2$. 
\item
Local inverses of $\rho$, that is maps $\varphi:U\to X$ defined over simply connected open sets 
$
U\subset\TT^2\setminus\{[0]\}
$ 
such that $\rho\circ\varphi=\id_U$, are all translations.
\end{enumerate}
A third equivalence says that $X$ is an origami if and only if its Veech group $\slgroup(X)$ shares a common subgroup of finite index with $\sltwoz$ (see \cite{GutkinJudge}). In particular, origamis are all Veech surfaces.

\subsection{Statement of main results}
\label{SectionTheoremsHittingTime}

It can be seen that the set $\hol(X)\subset\RR^2$ of relative periods of a translation surface $X$ is a discrete subset, whose projectivization is dense in $\PP\RR^2$, thus it is meaningful to consider diophantine approximations of a given direction $\theta$ by directions of vectors in $\hol(X)$. Given a saddle connection $\gamma$ and a direction $\theta$ on the surface $X$, the components of $\gamma$ along the direction $\theta$ are the two real numbers 
$
\re(\gamma,\theta)\in\RR
$ 
and 
$
\im(\gamma,\theta)\in\RR
$ 
such that 
$$
\hol(\gamma,r_{-\theta}\cdot X)=
\big(
\re(\gamma,\theta),\im(\gamma,\theta)
\big).
$$
The \emph{diophantine type} $w(X,\theta)\geq1$ of a Keane direction $\theta$ on the surface $X$ is the supremum of those $w\geq1$ such that there exists infinitely many saddle connections $\gamma$ for the surface $X$ with 
\begin{equation}
\label{EquationDiophantineTypeSaddleConnections}
|\re(\gamma,\theta)|<\frac{1}{|\im(\gamma,\theta)|^w}.
\end{equation}

As in the classical case, which corresponds to Equation \eqref{EquationDiophantineTypeClassical}, we have $w(X,\theta)\geq1$ for any $X$ and $\theta$, indeed for $w=1$ Equation \eqref{EquationDiophantineTypeSaddleConnections} has always infinitely many solutions. This corresponds to a version of Dirichlet's theorem for translation surfaces, which was known to many authors and a proof of which is given in Proposition 4.1 in \cite{MarcheseTrevinoWeil}. Moreover, the analogy with classical diophantine conditions extends also to Jarn\'ik Theorem, indeed in Theorem 6.1 \cite{MarcheseTrevinoWeil} it is proved that
$$
\dim_H 
\Big( 
\Big\{\theta \in \Big[-\frac{\pi}{2}, \frac{\pi}{2} \Big) \mid w^{cyl}(X,\theta)=\eta \Big\} 
\Big)=
\frac{2}{1+\eta}.
$$

Let $\theta$ be a Keane direction on the translation surface $X$, so that the flow $\phi_\theta$ is minimal and the functions 
$
\whitsup(\phi_\theta,\cdot,\cdot)
$, 
$
\whitinf(\phi_\theta,\cdot,\cdot)
$, 
$
\wrecsup(\phi_\theta,\cdot)
$ 
and 
$
\wrecinf(\phi_\theta,\cdot)
$ 
are defined outside of singular leaves. According to Lemma \ref{LemmaInvarianceRecurrenceAndHitting} of this paper, these functions are also invariant under $\phi_\theta$. A first relation with $w(X,\theta)$ can be derived by an easy geometric argument (see Lemma 7.2 in \cite{MarcheseTrevinoWeil}) which gives that for any $p$ not on any $(X,\theta)$-singular leaf we have
\begin{equation}
\label{EquationLowerBoundRecurrenceTime}
\wrecinf(\phi_\theta,p)\geq \frac{1}{w(X,\theta)}.
\end{equation}
On the other hand, according to Proposition \ref{PropositionRecurrenceKeaneDirection} of this paper, we prove that for Lebesgue almost any point $p\in X$ we have
$$
\wrecsup(\phi_\theta,p)=1.
$$
Under the extra assumption that the Lebesgue measure is ergodic for $\phi_\theta$, Equation \eqref{EquationUpperBoundRecurrenceCylinders} below establishes an uniform upper bound for $\wrecsup(\phi_\theta,p)$ for the generic $p$. Moreover according to Proposition \ref{PropositionLiminfHittingUniquelyErgodic} of this paper, for Lebesgue almost any $p$ and $p'$ we have also
$$
\whitinf(\phi_\theta,p,p')=1.
$$
In general, if $\mu$ is a non-ergodic invariant probability measure along a Keane direction $\theta$ on $X$, the function $\whitinf(\phi_\theta,\cdot,\cdot)$ can exhibit a different behavior for the $\mu\times\mu$ generic pair $(p,p')$, as it can be deduced by results in \cite{BoshernitzanChaika}. Our first main result establishes an upper bound for the hitting time on translation surfaces in $\cH(2)$.

\begin{theorem}
\label{TheoremUpperBoundHittingTime}
Let $X$ be a surface in $\cH(2)$ and let $\theta$ be a Keane direction on $X$. 
Then for any pair of points $p,p'$ not on any singular leaf we have
$$
\whitsup(\phi_\theta,p,p')\leq w(X,\theta)^2.
$$
\end{theorem}

The argument leading to Theorem \ref{TheoremUpperBoundHittingTime} can be considered as the first step in a more general procedure by induction on genus (see \S~\ref{SectionCombinatorialLemmasZipperedRectangles}, and in particular Corollary \ref{CorollaryNoHighTopSingularites}). In the general stratum the function $\whitsup(\phi_\theta,\cdot,\cdot)$ seems to be uniformly bounded by an expression which is a polynomial function of $w(X,\theta)$, whose degree is uniformly bounded for any genus. The situation in $\cH(2)$ is more interesting because in this stratum there are many surfaces and directions for which the bound is sharp, according to Theorem \ref{TheoremHittingSpectrumOrigami} below. 

\smallskip

Let $X$ be an origami and $\rho:X\to\TT^2$ be the corresponding ramified cover over the standard torus. 
Any saddle connection $\gamma$ on an origami $X$ gives rise to a closed geodesic 
$
t\mapsto \rho\circ\gamma(t)
$ 
on $\TT^2$, passing though the origin $[0]\in\TT^2$, whose direction $\theta_\gamma$ satisfies 
$
\tan(\theta_\gamma)=p/q\in\QQ
$. 
Moreover, if $m(\gamma)\in\NN^\ast$ denotes the degree of the map 
$
t\mapsto \rho\circ\gamma
$,  
then we have
$$
\hol(\gamma,X)=m(\gamma)\cdot(p,q)\in\ZZ^2.
$$
It follows that 
\begin{equation}
\label{EquationEquivalenceTypes}
|\re(\gamma,\theta)|<\frac{1}{|\im(\gamma,\theta)|^w}
\Leftrightarrow
\left|\tan\theta-\frac{p}{q}\right|
\leq
\frac{1}{m(\gamma)^{w+1}}\cdot\frac{1}{q^{1+w}}
\end{equation}
and observing that $1\leq m(\gamma)\leq N$ for any saddle connection $\gamma$, where $N$ is the number of squares of $X$, we get that for any $\theta$ we have 
$$
w(X,\theta)=w(\alpha)
\quad
\textrm{ where }
\quad
\alpha=\tan\theta.
$$

Since origamis are Veech surfaces, then the lower bound in Equation \eqref{EquationLimSupHittingVeech} below gives a simple counterpart to the upper bound in Theorem \ref{TheoremUpperBoundHittingTime}. In Theorem~\ref{TheoremHittingSpectrumOrigami} below we show that under specific topological assumption on $X$, both the upper and lower bound are sharp. An origami $X$ is said to be \emph{reduced} if 
$
\langle \hol(X)\rangle=\ZZ^2
$, 
that is the subgroup of $\RR^2$ generated by the set $\hol(X)$ of relative periods is the entire lattice $\ZZ^2$. Reduced origamis are relevant because they form a closed set for the subgroup action of $\sltwoz$ (see \S~\ref{SectionReducedOrigamisAndSL(2,Z)}), moreover the $g_t$-orbit of any origami $X$ contains a reduced origami $X_0$ (see \S~\ref{SectionReducedOrigamiAreEnough}). In \S~\ref{SectionVerticalSplittingPairs} we introduce a non-trivial topological property of a reduced origami $X_0$, which consists in the existence of a \emph{vertical splitting pair} $(\sigma_0,\gamma_0)$, where $\sigma_0$ and $\gamma_0$ are respectively a closed geodesic and a saddle connection in the vertical direction (satisfying a specific arithmetic condition), which splits the surface $X_0$ into a cylinder $C_0$, where $\phi_\theta$-orbits can be trapped for a long time, and a remaining non-empty open set. We say that an origami $X$ admits a \emph{splitting direction} $\theta_{split}$ if there exists some $G\in\gltwor$ such that $G\cdot X$ is a reduced origami with a vertical splitting pair. In \S~\ref{SectionOneCylinderVerticalDirections} we also consider a reduced origami $X_0$ whose vertical direction is a \emph{one cylinder direction} and we say that an origami $X$ admits an one cylinder direction $\theta_{one-cyl}$ if there exists some $G\in\gltwor$ such that $G\cdot X$ is a reduced origami whose vertical is a one cylinder direction. For any $\eta>1$ recall the function $f_\eta:[1,2]\to\RR_+$ defined by Equation~\eqref{EquationLowerBoundDimension}.

\begin{theorem}
\label{TheoremHittingSpectrumOrigami}
Let $X$ be any origami, so that in particular 
$
w(X,\theta)=w(\tan\theta)
$ 
for any direction $\theta$. Assume that $X$ admits both a splitting direction $\theta_{split}$ and a one cylinder direction $\theta_{one-cyl}$. Then for any $\eta>1$ and any $s$ with $1\leq s\leq 2$ there exists a set of directions $\EE(X,\eta,s)$ with 
$$
\dim_H\big(\EE(X,\eta,s)\big)
\geq
f_\eta(s)
$$
such that for any $\theta\in\EE(X,\eta,s)$ we have $w(\tan\theta)=\eta$, and moreover for almost any $p,p'$ in $X$ we have
$$
\whitsup(\phi_\theta,p,p')=w(\tan\theta)^s=\eta^s.
$$
\end{theorem}

Recall that, by Jarn\'ik Theorem, we have 
$
\dim_H\big(\EE(X,\eta,s)\big)\leq2/(1+\eta)
$ 
for any set as in Point (1) of Theorem \ref{TheoremHittingSpectrumOrigami} with respect to parameters $\eta\geq1$ and $s\in[1,2]$. Theorem~\ref{TheoremHittingSpectrumOrigami} is completed by the following Proposition.

\begin{proposition}
\label{PropositionHittingSpectrumOrigami}
Let $X$ be any origami.
\begin{enumerate}
\item
If $X$ admits a one cylinder direction $\theta_{one-cyl}$ then the result in Theorem~\ref{TheoremHittingSpectrumOrigami} holds for the parameter $s=1$. 
\item
If $X$ admits a splitting direction $\theta_{split}$, then for any $\eta>1$ there exists a set of directions $\EE(X,\eta,2)$ with 
$
\dim_H\big(\EE(X,\eta,2)\big)
\geq
f_\eta(2)=(1+\eta)^{-1}
$ 
such that for any $\theta\in\EE(X,\eta,2)$ we have $w(\tan\theta)=\eta$, and moreover for almost any $p,p'$ in $X$ we have
$$
\whitsup(\phi_\theta,p,p')\geq w(\tan\theta)^2=\eta^2.
$$
\end{enumerate}
\end{proposition}

\smallskip

By Corollary A.2 in \cite{McMullen}, any origami $X\in\cH(2)$ admits a one cylinder direction $\theta_{one-cyl}$. Moreover, according to Lemma \ref{LemmaExistenceSplittingPairs} of this paper, which is proved in Appendix \ref{SectionProofExistenceSplittingPairsH(2)}, any origami $X\in\cH(2)$ also admits a splitting direction $\theta_{split}$. Therefore Theorem~\ref{TheoremHittingSpectrumOrigami} can be applied to any origami $X\in\cH(2)$, and combining it with the bounds given by Theorem~\ref{TheoremUpperBoundHittingTime} and Equation~\eqref{EquationLimSupHittingVeech} below, the next Corollary follows.

\begin{corollary}
Let $X$ be any origami in $\cH(2)$. Then for any $\eta\geq1$ and any $s$ with $1\leq s\leq 2$ there exists a set $\EE(X,\eta,s)$ as in Theorem \ref{TheoremHittingSpectrumOrigami}. In particular, the spectrum of almost sure values of 
$
\whitsup(\phi_\theta,\cdot,\cdot)
$ 
over the set of those $\theta$ with $w(\tan\theta)=\eta$ equals the interval $[\eta,\eta^2]$.
\end{corollary}

Theorem~\ref{TheoremHittingSpectrumOrigami} above implies that in higher genus there is no functional relation between the diophantine type $w(X,\theta)$ and the almost sure value of the function 
$
\whitsup(\phi_\theta,\cdot,\cdot)
$ 
(where we recall that such almost sure value does not even exist if $\theta$ in a non ergodic direction on $X$). In genus $3$, the stratum $\cH(1,1,1,1)$ contains an origami $X_{EW}$ which is known by the German name \emph{Eierlegende Wollmilchsau}, whose definition is given in \S~\ref{SectionEierlegendeWollmilchsau}, and which exhibits several singular behaviors (see \S~7 and \S~8 in \cite{ForniMatheus}). In our case, such origami is special because it is a higher genus surface where the functional relation in Equation~\eqref{EquationTheoremDongHanRotations} subsists.

\begin{proposition}
\label{PropositionEierlegendeWollmilchsau}
Let $X_{EW}$ be the \emph{Eierlegende Wollmilchsau} origami. Then for any $\theta$ and almost any $p,p'$ in $X$ we have 
$$
\whitsup(\phi_\theta,p,p')=w(\tan\theta).
$$
\end{proposition}

\subsection{Other notions of diophantine type}
\label{SectionOtherNotionsDiophantineType}

One can consider also \emph{absolute periods} $\hol(\sigma)$ of closed geodesics $\sigma$ inside cylinders $C_\sigma$. For a Keane direction $\theta$ define
$
w^{cyl}(X,\theta)\geq1
$ 
as the supremum of those $w\geq1$ such that there exist infinitely many cylinders $C_\sigma$ with area 
$
\area(C_\sigma)\geq a
$, 
where $a=a(X)>0$ is a positive constant depending only on $X$, such that 
\begin{equation}
\label{EquationDiophantineTypeCylinders}
|\re(\sigma,\theta)|<\frac{1}{|\im(\sigma,\theta)|^w}.
\end{equation}

According to \S~\ref{SectionPhaseSpaceDynamics}, if $X$ is a Veech surface we have 
$
w^{cyl}(X,\theta)=w(X,\theta)
$ 
for any Keane direction $\theta$. When $X$ is a general translation surface, for any Keane direction $\theta$ on $X$ we always have 
$$
w^{cyl}(X,\theta)\leq w(X,\theta),
$$
indeed the boundary $\partial C_\sigma$ of any cylinder $C_\sigma$ around a closed geodesic $\sigma$ in $X$ is made of saddle connections $\gamma$ parallel to $\sigma$. Moreover when $w=1$, Equation \eqref{EquationDiophantineTypeCylinders} has always infinitely many solutions for any Keane direction $\theta$ on $X$ (replacing $1$ in the numerator by a bigger constant, see Proposition 4.1 in \cite{MarcheseTrevinoWeil}, which is derived from results in \cite{Vorobets}). Finally, according to Theorem 6.1 \cite{MarcheseTrevinoWeil} we also have
$$
\dim_H 
\Big( 
\Big\{\theta \in \Big[-\frac{\pi}{2}, \frac{\pi}{2} \Big) \mid w^{cyl}(X,\theta)=\eta \Big\} 
\Big)=
\frac{2}{1+\eta},
$$
and for these reasons $w^{cyl}(X,\theta)$ can also be considered as a natural notion of diophantine type. According to Lemma 7.3 in \cite{MarcheseTrevinoWeil}, which was developed after joint discussions related to this paper, if the Lebesgue measure is ergodic under the flow $\phi_\theta$ in direction $\theta$ on the surface $X$, then for almost any $p\in X$ we have 
\begin{equation}
\label{EquationUpperBoundRecurrenceCylinders}
\wrecinf(\phi_\theta,p)
\leq 
\frac{1}{w^{cyl}(X,\theta)}.
\end{equation}
Moreover, according to Proposition \ref{PropositionLimsupHittingUniquelyErgodic(cylinders)} is this paper, if $\theta$ in an ergodic direction on the surface $X$, then for almost any $p,p'$ in $X$ we have
$$
\whitsup(\phi_\theta,p,p')
\geq
w^{cyl}(X,\theta).
$$
In particular, let $\theta$ be a Keane direction on a Veech surface $X$, so that 
$
w^{cyl}(X,\theta)=w(X,\theta)
$. 
According to Equation \eqref{EquationLowerBoundRecurrenceTime} and Equation \eqref{EquationUpperBoundRecurrenceCylinders}, for almost any $p$ we have
\begin{equation}
\label{EquationLimInfRecurrenceVeech}
\wrecinf(\phi_\theta,p)
=
\frac{1}{w(X,\theta)}.
\end{equation}
Moreover, according to Proposition \ref{PropositionLimsupHittingUniquelyErgodic(cylinders)}, for almost any $p,p'$ we have
\begin{equation}
\label{EquationLimSupHittingVeech}
\whitsup(\phi_\theta,p,p')\geq w(X,\theta).
\end{equation}

The notion of diophantine type considered in this paper is also related to \emph{geodesic excursions} in moduli space. This is known to many authors, we refer to \S~6.3 in \cite{MarcheseTrevinoWeil} for sake of completeness. There are other definitions of Diophantine type given by the size of continued fraction matrices in the Rauzy-Vecch induction algorithm \cite{Kim}, \cite{KimMarmi2}. 
For example Roth type Diophantine condition form a full measure set \cite{MMY} and can be used for obtaining H\"older estimates for the solution of the cohomological equation \cite{MY}. 
See also \cite{HubertMarcheseUlcigrai} and \cite{KimMarmi} for more discussion on the size of the continued fraction matrices.

\subsection{Back to billiards}
\label{SectionBackToBilliards}

Consider the \emph{billiard flow} $\widetilde{\phi}$ in a \emph{rational polygon}, that is a polygon $Q\subset\RR^2$ whose angles are all rational multiples of $\pi$. Let $D$ be the finite group of linear isometries of $\RR^2$, which are the linear part of affine isometries generated by reflections at sides of $Q$. The group $D$ acts on directions $\theta\in S^1$, so that any $\theta$ has a finite orbit $[\theta]$ and all orbits have the same cardinality. Therefore the phase space $Q\times S^1$ is foliated into invariant surfaces of the form $Q\times [\theta]$, which are all mutually isometric, and the billiard flow acts as a linear flow on each of them. If $[\theta]$ is the $D$-orbit of $\theta$, then denoting by 
$
\widetilde{\phi}_{[\theta]}
$ 
the restriction to $Q\times[\theta]$ of the billiard flow $\widetilde{\phi}$ on $Q\times S^1$, we have the commutative diagram
\begin{equation}
\label{EquationBilliardTranslationFlows}
\begin{array}{ccc}
X & \quad\substack{\phi_\theta\\ \longrightarrow }\quad & X \\
\downarrow &                                                     & \downarrow \\
Q\times[\theta] & \quad\substack{\widetilde{\phi}_{[\theta]}\\ \longrightarrow }\quad & Q\times[\theta].
\end{array}
\end{equation}

Let $L=L(a,b,a',b')$ be the L-shaped polygon defined by real numbers $0<a'<a$ and $0<b'<b$ and let $D\sim(\ZZ/2\ZZ)^2$ be its reflection group. The surface $X$ related to $L$ in Equation \eqref{EquationBilliardTranslationFlows} is obtained pasting the opposite sides of the cross-like polygon 
$
P=\bigcup_{s\in D}s(L)
$, 
thus $X\in\cH(2)$. In particular, if $a,b,a',b'$ are rationally dependent, then modulo an homothety one can assume that $a,b,a',b'$ are all integers. In this case also the polygon $P$ has vertices with integer coordinates, thus $X$ is a square-tiled surface in $\cH(2)$. Theorem \ref{TheoremLShapedBilliard} follows from Theorem \ref{TheoremUpperBoundHittingTime} and Theorem \ref{TheoremHittingSpectrumOrigami}, via Equation \eqref{EquationBilliardTranslationFlows}. Similarly, Proposition \ref{PropositionRecurrenceKeaneDirection} implies 
$
\wrecsup(\widetilde{\phi}_{[\theta]},p)=1
$ 
for almost any $p\in L\times[\theta]$, and Proposition \ref{PropositionLiminfHittingUniquelyErgodic} implies 
$
\whitinf(\widetilde{\phi}_{[\theta]},p,p)=1
$ 
for Lebesgue almost any $p,p'$, when the Lebesgue measure over $L$ is ergodic for 
$
\widetilde{\phi}_{[\theta]}
$. 
Finally, since origamis are Veech surfaces, when $a,a',b,b'$ are rationally dependent Equation \eqref{EquationLimInfRecurrenceVeech} implies 
$
\wrecinf(\widetilde{\phi}_{[\theta]},p)=1/w(\tan\theta)
$ 
for almost any $p\in L\times[\theta]$.

\subsection{Summary of the contents}

The rest of the paper is organized as follows.

In \S~\ref{SectionZipperedRectanglesAndRauzyVeech} we recall the representation of a translation surface $X$ by \emph{zippered rectangles} pasted together along a Keane direction $\theta$. Zippered rectangles (defined in \S~\ref{SectionRepresentationByZipperedRectangles}) are used in \S~\ref{SectionUpperBoundHittingTime} to prove Theorem \ref{TheoremUpperBoundHittingTime}. In 
\S~\ref{SectionRauzyVeechOnZipperedRectangles} we give a qualitative description of the \emph{Rauzy-Veech induction} on zippered rectangles, which plays a role in \S~\ref{SectionBoundsErgodicDirections} in the proof of Proposition \ref{PropositionRecurrenceKeaneDirection} and of Proposition \ref{PropositionLiminfHittingUniquelyErgodic}.

In \S~\ref{SectionBoundsErgodicDirections} we first establish some general properties of the functions 
$
\wrecinf
$, 
$
\wrecsup
$, 
$
\whitinf
$, 
$
\whitsup
$. 
Then we prove Proposition \ref{PropositionRecurrenceKeaneDirection} and of Proposition \ref{PropositionLiminfHittingUniquelyErgodic}. We also prove Proposition \ref{PropositionLimsupHittingUniquelyErgodic(cylinders)}, by a geometric argument only based on the flat geometry of cylinders.

In \S~\ref{SectionUpperBoundHittingTime} we prove Theorem \ref{TheoremUpperBoundHittingTime}. In \S~\ref{SectionCombinatorialLemmasZipperedRectangles} we establish some general combinatorial Lemmas on zippered rectangles. In \S~\ref{SectionEndProofUpperBoundHitting} the general Lemmas are applied to the specific case of $\cH(2)$, completing the proof of Theorem \ref{TheoremUpperBoundHittingTime}. 

In \S~\ref{SectionGeometricConstructionsOrigamis} we explain the geometric constructions needed to prove  Theorem \ref{TheoremHittingSpectrumOrigami}. In \S~\ref{SectionReducedOrigamisAndSL(2,Z)} we recall the graph structure of orbits of reduced origamis under the action of $\slgroup(2,\ZZ)$. In \S~\ref{SectionOneCylinderVerticalDirections} we consider origamis whose vertical is a one cylinder direction, which are used to obtain an upper bound for hitting time. In \S~\ref{SectionVerticalSplittingPairs} we consider \emph{vertical splitting pairs} $(\sigma_0,\gamma_0)$ for reduced origamis, which are used to obtain a lower bound for hitting time. 

In \S~\ref{SectionProofTheoremOrigami} we prove Theorem \ref{TheoremHittingSpectrumOrigami} and Proposition~\ref{PropositionHittingSpectrumOrigami}. In \S~\ref{SectionConstructionCantorSlopes} we define a set of directions and in \S~\ref{SectionDimensionEstimate} we give a lower bound for its dimension. The proof of Theorem \ref{TheoremHittingSpectrumOrigami} is completed in \S~\ref{SectionProofTheoremHittingSpectrumOrigami}. The proof of Proposition~\ref{PropositionHittingSpectrumOrigami} is completed in \S~\ref{SectionProofPropositionHittingSpectrumOrigami}. 
 
In \S~\ref{SectionEierlegendeWollmilchsau} we prove Proposition~\ref{PropositionEierlegendeWollmilchsau}. The argument is an easy modification of the construction in \S~\ref{SectionOneCylinderVerticalDirections}.

In \S~\ref{SectionProofExistenceSplittingPairsH(2)} we prove Lemma \ref{LemmaExistenceSplittingPairs}, which ensures that any origami $X$ in $\cH(2)$ admits a splitting direction.

\section{Zippered rectangles and Rauzy-Veech induction}
\label{SectionZipperedRectanglesAndRauzyVeech}

Recall that a translation surface $X$ is defined as quotient space of the disjoint union of polygons $P_1,\dots,P_d$ in the plane under some identifications in their boundary. Here, for a Keane direction $\theta$ on $X$, we describe a representation of the surface where all the polygons are rectangles $R_1,\dots,R_d$, with sides alignes along the direction $\theta$. The construction is originally due to Veech, here we follow the presentation in \S~4.3 of \cite{Yoccoz}.

\smallskip

An alphabet is a finite set $\cA$ with $d\geq2$ letters. A \emph{combinatorial datum} $\pi$ is a pair of bijections $\pi_t,\pi_b:\cA\to\{1,\dots,d\}$. For us combinatorial data $\pi$ are assumed to be \emph{admissible} that is 
$
(\pi_t)^{-1}\{1,\dots,k\}\not=(\pi_b)^{-1}\{1,\dots,k\}
$ 
for any $1\leq k\leq d-1$. A \emph{length datum} is any vector $\lambda\in\RR_+^\cA$ with all entries positive. For any such pair of data $(\pi,\lambda)$ consider the interval 
$
I:=[0,\sum_{\chi\in\cA}\lambda_\chi)
$ 
and its two partitions 
$
I=\bigsqcup_{\alpha\in\cA}I^t_\alpha
$ 
and 
$
I=\bigsqcup_{\beta\in\cA}I^t_\beta
$, 
where for any $\alpha$ and $\beta$ we set
$$
I^t_\alpha:=
\Big[
\sum_{\pi_t(\chi)\leq\pi_t(\alpha)-1}\lambda_\chi,
\sum_{\pi_t(\chi)\leq\pi_t(\alpha)}\lambda_\chi
\Big)
\quad
\textrm{ and }
\quad
I^b_\beta:=
\Big[
\sum_{\pi_b(\chi)\leq\pi_b(\beta)-1}\lambda_\chi,
\sum_{\pi_b(\chi)\leq\pi_b(\beta)}\lambda_\chi
\Big).
$$ 
An IET, or extensively \emph{interval exchange transformation}, is the map $T:I\to I$ uniquely determined by the data $(\pi,\lambda)$ as the map that for any $\alpha$ sends $I^t_\alpha$ onto $I^b_\alpha$ via a translation.

\subsection{Veech's zippered rectangles construction}
\label{SectionRepresentationByZipperedRectangles}

A \emph{suspension datum} for the combinatorial datum $\pi$ is a vector $\tau\in\RR^\cA$ such that for any $\alpha$ and $\beta$ we have 
$$
\sum_{\pi_t(\chi)<\pi_t(\alpha)}\tau_\chi>0
\quad
\textrm{ and }
\quad
\sum_{\pi_b(\chi')<\pi_b(\alpha)}\tau_{\chi'}<0.
$$
Fix combinatorial-length-suspension data $(\pi,\lambda,\tau)$. The procedure below defines a translation surface 
$
X=X(\pi,\lambda,\tau)
$. 
Details on the construction can be find in \S~4.3 of \cite{Yoccoz}. A picture of a surface obtained by this construction can be seen in Figure \ref{FigureRemoveCilinder} of this paper. Let us first set 
$$
\tau_\ast:=\sum_{\chi\in\cA}\tau_\chi
$$
then let $h=h(\pi,\tau)\in\RR_+^\cA$ be the vector whose coordinates, for any $\alpha\in\cA$, are defined by
$$
h_\alpha:=
\sum_{\pi_t(\chi)<\pi_t(\alpha)}\tau_\chi
-
\sum_{\pi_b(\chi')<\pi_b(\alpha)}\tau_{\chi'}.
$$

For any $\alpha\in\cA$ define the rectangle 
$
R_\alpha^t:=I_\alpha^t\times[0,h_\alpha]\subset\RR^2
$. 
These rectangles are the polygons to be pasted in order to define the translation surface $X$. The identifications between their horizontal sides of the rectangles $R_\alpha^t$ for $\alpha\in\cA$ are given by 
$$
(x,h_\alpha)\sim (Tx,0)
\quad
\textrm{ iff }
\quad
x\in I^t_\alpha.
$$
In order to describe the identification between the vertical sides of the rectangles above, it is convenient to introduce a copy of them, setting 
$
R_\alpha^b:=I_\alpha^b\times[-h_\alpha,0]
$ 
for any $\alpha$. We stress that this is just a convenient way to describe identifications, while in the quotient surface one has $R_\alpha^t=R_\alpha^b$. For letters $\alpha$ and $\beta$ with respectively 
$
\pi_t(\alpha)\geq2
$ 
and 
$
\pi_b(\beta)\geq2
$ 
define the vertical segments 
$$
S_\alpha^t:=
\{\inf I_\alpha^t\}\times\big[0,\sum_{\pi_t(\chi)<\pi_t(\alpha)}\tau_\chi\big)
\quad
\textrm{ and }
\quad
S_\beta^b:=
\{\inf I_\beta^b\}\times\big(\sum_{\pi_b(\chi)<\pi_b(\alpha)}\tau_\chi,0\big].
$$
Consider also the vertical segment $S^\ast$ whose endpoints are 
$
\big(\sum_{\chi\in\cA}\lambda_\chi,0\big)
$ 
and 
$
\big(\sum_{\chi\in\cA}\lambda_\chi,\tau_\ast\big)
$. 
The identifications between vertical sides are given by the procedure below.
\begin{enumerate}
\item
For any $\alpha$ with $\pi_t(\alpha)\geq2$ let $\alpha'$ be the letter with 
$
\pi_t(\alpha')=\pi_t(\alpha)+1
$, 
then identify the common segment of $\partial R_{\alpha'}^t$ and $\partial R_\alpha^t$ which coincide with the segment $S_\alpha^t$. 
\item
For any $\beta$ with $\pi_b(\beta)\geq2$ let $\beta'$ be the letter with 
$
\pi_b(\beta')=\pi_b(\beta)+1
$, 
then identify the common segment of $\partial R_{\beta'}^b$ and $\partial R_\beta^b$ which coincide with the segment $S_\beta^b$. Since any $R_\alpha^b$ is identified with $R_\alpha^t$, this induces identifications between segments in the vertical sides of the rectangles $R_\beta^t$ and $R_{\beta'}^t$ not considered in the previous Point (1).
\item
The last segment to be identified to some parallel one is $S^\ast$. If $\tau_\ast\geq0$ then $S_\ast$ is identified with 
$
S_\alpha^t\setminus\partial R_{\beta_0}
$, 
where $\beta_0$ and $\alpha$ are the letters with respectively $\pi_b(\beta_0)=d$ and 
$
\pi_t(\alpha)=\pi_t(\beta_0)+1
$. 
If $\tau_\ast<0$ then $S_\ast$ is identified with 
$
S_\beta^b\setminus\partial R_{\alpha_0}
$, 
where $\alpha_0$ and $\beta$ are the letters with respectively $\pi_t(\alpha_0)=d$ and 
$
\pi_b(\beta)=\pi_b(\alpha_0)+1
$.  
\end{enumerate}

Let $X=X(\pi,\lambda,\tau)$ be the translation surface obtained by the above construction. According to Proposition 5.7 in \cite{Yoccoz}, if $X$ is any translation surface and if $\theta$ is a Keane direction on $X$, then there exists data $(\pi,\lambda,\tau)$ such that 
\begin{equation}
\label{EquationRepresentationByZipperedRectangles}
r_{-\theta}\cdot X=X(\pi,\lambda,\tau).
\end{equation}
Moreover, modulo identifying 
$
I:=\big[0,\sum_{\chi\in\cA}\lambda_\chi\big)
$ 
with an horizontal segment $I\subset X$, the IET $T:I\to I$ corresponding to data $(\pi,\lambda)$ is the first return of $\phi_\theta$ to $I$.

\subsection{On the position of conical points in the flat representation}
\label{SectionPositionConicalPointsFlatRepresentation}

This subsection contains some remarks that will be used in \S~\ref{SectionCombinatorialLemmasZipperedRectangles}. It can be skipped at the first reading. According to a standard notation, identify $\RR$ with $\CC$. Fix combinatorial-length-suspension data 
$
(\pi,\lambda,\tau)
$, 
then for any $\alpha\in\cA$ consider the complex number 
$
\zeta_\alpha:=\lambda_\alpha+i\tau_\alpha
$ 
and set
$$
\xi^t_\alpha:=
\sum_{\pi_t(\chi)<\pi_t(\alpha)}\zeta_\chi
\quad
\textrm{ and }
\quad
\xi^b_\beta:=
\sum_{\pi_b(\chi')<\pi_b(\alpha)}\zeta_{\chi'},
$$
which give the coordinates of conical singularity of the surface $X$ in the planar development corresponding to the data $(\pi,\lambda,\tau)$. We have $\xi^t_\alpha=0$ if $\pi_t(\alpha)=1$ and 
$
\im(\xi^t_\alpha)>0
$ 
whenever $2\leq \pi_t(\alpha)\leq d$. Similarly $\xi^b_\beta=0$ if $\pi_b(\beta)=1$ and 
$
\im(\xi^b_\beta)<0
$ 
whenever $2\leq \pi_b(\beta)\leq d$. Moreover for any $\alpha$ the hight $h_\alpha$ of the corresponding rectangle satisfies 
$
h_\alpha=\im(\xi^t_\alpha)-\im(\xi^b_\alpha)>0
$. 
It follows that points $\xi^{t/b}_\alpha$ are always contained in the left vertical boundary side of the corresponding rectangle $R^{t/b}_\alpha$, that is for any $\alpha\in\cA$ we have
$$
0\leq \im(\xi^t_\alpha)\leq h_\alpha
\quad
\textrm{ and }
\quad
0\geq \im(\xi^b_\alpha)\geq -h_\alpha.
$$
For the right boundary side this is always true with one exception. Consider letters $\alpha$ and $\beta$ with $\pi_t(\alpha)\geq 2$ and $\pi_b(\beta)\geq 2$, so that there exists letters $\alpha'$ and 
$
\beta'
$ 
such that respectively  
$
\pi_t(\alpha)=\pi_t(\alpha')+1
$ 
and 
$
\pi_b(\beta)=\pi_b(\beta')+1
$. 
Assume that $\tau_\ast<0$. 
\begin{enumerate}
\item
We have always 
$
\sum_{\pi_b(\chi)\leq\pi_b(\alpha')}\tau_\chi<0
$, 
either by suspension condition, or by assumption $\tau_\ast<0$, thus
$$
0\leq
\im(\xi^t_\alpha)=
\im(\xi^t_{\alpha'})+\tau_{\alpha'}=
h_{\alpha'}+\sum_{\pi_b(\chi)\leq\pi_b(\alpha')}\tau_\chi\leq h_{\alpha'}.
$$ 
\item
If $\pi_t(\beta')<d-1$ we have
$
\sum_{\pi_t(\chi)\leq\pi_t(\beta')}\tau_\chi>0
$ 
by suspension condition, hence 
$$
0\geq
\im(\xi^b_\beta)=
\im(\xi^b_{\beta'})+\tau_{\beta'}=
-h_{\beta'}+\sum_{\pi_t(\chi)\leq\pi_t(\beta')}\tau_\chi
\geq -h_{\beta'}.
$$ 
\item
Finally, the exceptional case occurs when $\pi_t(\beta')=d$, indeed we have
$$
0\geq
\im(\xi^b_\beta)=
\im(\xi^b_{\beta'})+\tau_{\beta'}=
-h_{\beta'}+\sum_{\pi_t(\chi)\leq\pi_t(\beta')}\tau_\chi=
-h_{\beta'}+\tau_\ast< -h_{\beta'}.
$$ 
\end{enumerate}

\subsection{Rauzy Veech induction}
\label{SectionRauzyVeechOnZipperedRectangles}

Let $X$ be any translation surface and $\theta$ a Keane direction on $X$, then consider combinatorial-length-suspension data $(\pi,\lambda,\tau)$ representing the surface $X$ along the direction $\theta$, that is data such that Equation \eqref{EquationRepresentationByZipperedRectangles} is satisfied.  The so-called \emph{Rauzy-Veech induction algorithm} defines inductively a sequence of data 
$
(\pi^{(n)},\lambda^{(n)},\tau^{(n)})
$ 
representing the same surface $X$ along the direction $\theta$. In this subsection we just point out a qualitative property of the algorithm used in \S~\ref{SectionBoundsErgodicDirections}, for more details see \S~7 of \cite{Yoccoz}.

\medskip

Admissible combinatorial data $\pi$ over the alphabet $\cA$ are organized in \emph{Rauzy classes}, which are always finite, since the number of admissible combinatorial data over a finite alphabet is finite. In any Rauzy class $\cR$ are defined two bijections $Q^t:\cR\to\cR$ and $Q^b:\cR\to\cR$, where the symbols $t$ and $b$ stands respectively for \emph{top} and $\emph{bottom}$. For any $\pi\in\cR$ and any value of $\epsilon\in\{t,b\}$, it is defined an integer matrix 
$
B(\pi,\epsilon)\in\slgroup(d,\ZZ)
$ 
(see \S~7.5 in \cite{Yoccoz}). For a pair of data $(\pi,\lambda)$ with
\begin{equation}
\label{EquationDomainDefinitionRauzyVeech}
\sum_{\pi_t(\chi)<d}\lambda_\chi
\not=
\sum_{\pi_b(\chi')<d}\lambda_{\chi'}
\end{equation}
it is defined a \emph{type} $\epsilon=\epsilon(\pi,\lambda)\in\{t,b\}$ as follows. Let $\alpha_t$ and $\alpha_b$ be the letters with 
$
\pi_t(\alpha_t)=\pi_b(\alpha_b)=d
$, 
then set $\epsilon(\pi,\lambda)=t$ if 
$
\lambda_{\alpha_t}>\lambda_{\alpha_b}
$ 
whereas $\epsilon(\pi,\lambda)=b$ if 
$
\lambda_{\alpha_t}<\lambda_{\alpha_b}
$ 
(see \S~7.2 in \cite{Yoccoz}). Consider $n\geq1$ and assume the sequence of data 
$
(\pi^{(k)},\lambda^{(k)},\tau^{(k)})
$ 
is defined for $k=0,\dots,n-1$, with 
$
(\pi^{(0)},\lambda^{(0)},\tau^{(0)})=(\pi,\lambda,\tau)
$. 
Then define inductively 
\begin{eqnarray*}
&&
\epsilon_n:=\epsilon(\pi^{(n-1)},\lambda^{(n-1)})
\\
&&
\pi^{(n)}:=Q^{\epsilon_n}(\pi^{(n-1)})
\\
&&
\lambda^{(n)}:=B(\pi^{(n-1)},\epsilon_n)\lambda^{(n-1)}
\\
&&
\tau^{(n)}:=B(\pi^{(n-1)},\epsilon_n)\tau^{(n-1)},
\end{eqnarray*} 
where the definition above is possible for any $n\geq1$ because if $\theta$ is a Keane direction on $X$ then any data $(\pi^{(n-1)},\lambda^{(n-1)})$ satisfies Condition \eqref{EquationDomainDefinitionRauzyVeech} (see \S~5 in \cite{Yoccoz}), thus the type 
$
\epsilon_n=\epsilon(\pi^{(n-1)},\lambda^{(n-1)})
$ 
is always defined. For any $n$, let $T^{(n)}:I^{(n)}\to I^{(n)}$ be the IET determined by the data 
$
(\pi^{(n)},\lambda^{(n)})
$, 
which acts on the interval 
$
I^{(n)}:=\big[ 0, \sum_{\chi\in\cA} \lambda^{(n)}_\chi \big)
$. 
The following holds
\begin{enumerate}
\item
We have $I^{(n)}\subset I^{(n-1)}$ and $T^{(n)}:I^{(n)}\to I^{(n)}$ is the first return map of $T^{(n-1)}$ to $I^{(n)}$ (\S~7.2 in \cite{Yoccoz}).
\item
The data $(\pi^{(n)},\lambda^{(n)},\tau^{(n)})$ are combinatorial-length-suspension data satisfying Condition \eqref{EquationRepresentationByZipperedRectangles}. The interval $I^{(n)}$ is identified with an horizontal segment in $X$ and $T^{(n)}$ is the first return map of $\phi_\theta$ to $I^{(n)}$ (\S~7.4 in \cite{Yoccoz}).
\item
Define  
$
h^{(n)}:=h(\pi^{(n)},\tau^{(n)})\in\RR_+^\cA
$ 
as in the beginning of \S~\ref{SectionRepresentationByZipperedRectangles}. For any $\alpha\in\cA$ we have (\S~7.7 in \cite{Yoccoz})
$$
\lambda^{(n)}_\alpha\to0
\quad
\textrm{ and }
\quad
h^{(n)}_\alpha\to+\infty
\quad
\textrm{ for }
\quad
n\to\infty.
$$
\item
Let $\mu$ be any ergodic invariant measure for $\phi_\theta$. According to \S~8 in \cite{Yoccoz}, the rectangles defined from the data $(\pi^{(n)},\lambda^{(n)},\tau^{(n)})$ as in \S~\ref{SectionRepresentationByZipperedRectangles} give a $\mod\mu$ partition of $X$ by open rectangles
\begin{equation}
\label{EquationRenormalizedZipperedRectangles}
R^{(n)}_\alpha:=
\big(
\sum_{\pi^{(n)}_t(\chi)\leq\pi^{(n)}_t(\alpha)-1}\lambda^{(n)}_\chi,
\sum_{\pi^{(n)}_t(\chi)\leq\pi^{(n)}_t(\alpha)}\lambda^{(n)}_\chi,
\big)
\times
(0,h^{(n)}_\alpha)
\quad
\textrm{ , }
\quad
\alpha\in\cA.
\end{equation}
\end{enumerate}

\section{Generic bounds for Keane and ergodic directions}
\label{SectionBoundsErgodicDirections}

\subsection{General measure preserving systems}
\label{SectionGeneralMeasurePreservingSystems}

In general measure preserving dynamical systems, the upper limit of the recurrence asymptote is bounded by the dimension. 
For example, in \cite{BarreiraSassuol} it was shown that if $T : X \to X$ is a Borel measurable transformation on a measurable set $X \subset {\mathbb R}^d$ for some $d \in \mathbb N$ and $\mu$ is a $T$-invariant probability measure on $X$, then for $\mu$ almost every $p$ we have 
$$ 
\wrecsup(T,p) \le \limsup_{r \to 0^+} \frac{\log \mu (B_r(x))}{\log r}. 
$$
 
For the hitting time, it's also well known that the limit inferior of the hitting time asymptote is bounded by below. In \cite{Galatolo} it was shown that, if $(X,T)$ is a discrete time dynamical system  where $X$ is a separable metric space equipped with a Borel locally finite measure $\mu$ and $T : X \to X$ is a measurable map for each fixed $p'\in X$ and for $\mu$ almost every $p$, then we have
$$
\whitinf (T,p,p') \ge \liminf_{r \to 0^+} \frac{\log \mu (B_r(p'))}{\log r}. 
$$

\subsection{General Keane directions on a translation surface}

Let $X$ be a translation surface of genus $g\geq1$ with $\area(X)=1$. For any Keane direction $\theta$ on $X$ the set $\cP(X,\theta)$ of Borel probability measures invariant under $\phi_\theta$ is a finite dimensional simplicial cone of dimension at most $g$, whose extremal points are the ergodic measures $\mu_1,\dots,\mu_g$ for $\phi_\theta$ (see \S~8.2 in \cite{Yoccoz}). In particular, we have positive real numbers with $a_1+\dots+a_g=1$ such that  
\begin{equation}
\label{EquationErgodicDecompositionLeb}
\leb=a_1\mu_1+\dots+a_g\mu_g.
\end{equation}

\begin{lemma}
\label{LemmaMeasureTheoreticBounds}
Let $\theta$ be a Keane direction on $X$ and let $\mu$ be an invariant Borel probability measure for 
$
\phi_\theta
$. 
Fix $p'\in X$. Then for $\mu$-a.e. $p\in X$ we have
$$
\wrecsup(\phi_\theta,p) \leq 1, \qquad \whitinf (\phi_\theta,p,p') \geq 1.
$$
\end{lemma}

\begin{proof}
Let $(\pi,\lambda,\tau)$ satisfying Equation \eqref{EquationRepresentationByZipperedRectangles}, so that the surface $X$ is obtained by zippered rectangles in direction $\theta$, pasted along a segment $I$ in direction orthogonal to $\theta$, and the first return of $\phi_\theta$ to $I$ is a Keane IET $T:I\to I$. Continuous time for 
$
\phi_\theta
$ 
and discrete time for $T$ are comparable, the ratio between the two being bounded by the ratio between the tallest and the shortest rectangle. Moreover the $\phi_\theta$-invariant measure $\mu$ corresponds to an invariant Borel probability measure for $T$. Then the statement follows because, according to the results of  \cite{BarreiraSassuol} and \cite{Galatolo} reported in \S~\ref{SectionGeneralMeasurePreservingSystems}, it holds for $T:T\to I$.  
\end{proof}

\begin{lemma}
\label{LemmaInvarianceRecurrenceAndHitting}
Let $\theta$ be a Keane direction on a translation surface $X$. Then for any pair of points $p,p'$ in $X$ not on any singular leaf, and for any $s,t$ in $\RR$ we have 
\begin{align*}
\wrecsup\big(\phi_\theta,\phi^t_\theta(p)\big)
&=
\wrecsup\big(\phi_\theta,p\big)
\\
\wrecinf\big(\phi_\theta,\phi^t_\theta(p)\big)
&=
\wrecinf\big(\phi_\theta,p\big)
\\
\whitsup\big(\phi_\theta,\phi^t_\theta(p),p'\big)
&=
\whitsup\big(\phi_\theta,p,\phi^s_\theta(p')\big)
=
\whitsup\big(\phi_\theta,p,p'\big)
\\
\whitinf\big(\phi_\theta,\phi^t_\theta(p),p'\big)
&=
\whitinf\big(\phi_\theta,p,\phi^s_\theta(p')\big)
=
\whitinf\big(\phi_\theta,p,p'\big)
\end{align*}
\end{lemma}

\begin{proof}
Just observe that since locally the flow $\phi_\theta$ is a translation, when $r$ is small enough we have
\begin{eqnarray*}
&&
\rt\big(\phi_\theta,\phi^t_\theta(p),r \big)
=
\rt\big(\phi_\theta,p,r\big)
\\
&&
| \rt\big(\phi_\theta,\phi^t_\theta(p),p',r\big)
-
\rt\big(\phi_\theta,p,p',r\big)|
\leq |t|
\\
&&
| \rt\big(\phi_\theta,p,\phi^s_\theta(p'),r \big)
-
\rt\big(\phi_\theta,p,p',r\big)|
\leq |s|.
\end{eqnarray*}
\end{proof}

\begin{lemma}
\label{LemmaRecurrenceHittingConstant}
Fix any direction $\theta$ on a translation surface $X$ and let $\mu$ be an ergodic measure for $\phi_\theta$. 
Then there exist constants 
$$
0\leq \underline{w}_1\leq \overline{w}_1\leq 1\
\quad
\textrm{ and }
\quad
1\leq \underline{w}_2 \leq \overline{w}_2 \leq +\infty
$$
depending on $\mu$ such that 
\begin{align*}
&\wrecinf(\phi_\theta,p) =\underline{w}_1 &
\textrm{ and } \quad
&\wrecsup(\phi_\theta,p) =\overline{w}_1
\qquad
&\textrm{ for $\mu$ a.e. }
p.\\
&\whitinf(\phi_\theta,p,p') =\underline{w}_2  &
\textrm{ and } \quad
&\wrecsup(\phi_\theta,p,p') =\overline{w}_2
\quad
&\textrm{ for $\mu\times\mu$ a.e. }
p,p'.
\end{align*}
\end{lemma}

\begin{proof}
It is well-known that if $\varphi:X\to[0,+\infty]$ is a Borel function which is invariant under $\phi_\theta$, and $\mu$ is ergodic under $\phi_\theta$, then there exists a constant $c\in[0,+\infty]$ such that $\varphi(x)=c$ for $\mu$-almost any $p\in X$ (e.g. \cite{Walters}). The first part of the statement follows trivially since 
$
\wrecinf(\phi_\theta,\cdot)\leq\wrecsup(\phi_\theta,\cdot)
$ 
and $\wrecsup(\phi_\theta,\cdot)\in L^1(\mu)$ according to Lemma \ref{LemmaMeasureTheoreticBounds}. 
We finish the proof for $\whitsup(\phi_\theta,\cdot,\cdot)$, the argument for $\whitinf(\phi_\theta,\cdot,\cdot)$ being the same. The generalized Birkhoff Theorem recalled above implies that for any $p\in X$ there exists 
$
\overline{w}_2(p)\in [1,+\infty]
$ 
such that 
$$
\whitsup(\phi_\theta,p,p')=\overline{w}_2(p)
\quad
\textrm{ for $\mu$-a.e. }p'\in X.
$$
Similarly for any $p'\in X$ there exists 
$
\overline{w}_2(p')\in [1,+\infty]
$ 
such that 
$$
\whitsup(\phi_\theta,p,p')=\overline{w}_2(p')
\quad
\textrm{ for $\mu$-a.e. }p\in X.
$$
Then a standard Fubini argument gives the proof.
\end{proof}

The simple Lemma \ref{LemmaMeasureLimsupSet} below will be used in several arguments. For a countable family of measurable sets $(A_n)_{n\in\NN}$ in a probability space $(X,\mu)$ we set 
$$
\limsup A_n:=
\bigcap_{n=1}^\infty \big( \bigcup_{k \ge n} A_k \big).
$$

\begin{lemma}
\label{LemmaMeasureLimsupSet}
For any countable family of sets $(A_n)_{n\in\NN}$ in a probability space $(X,\mu)$ we have
$$
\mu ( \limsup A_n) \ge \limsup \mu(A_n).
$$
\end{lemma}

\begin{proof}
Just observe that since $\cup_{k \ge n} A_k$ is a decreasing sequence of sets, we have 
\begin{equation*}
\mu(\limsup A_n)=
\lim_{n \to \infty} \mu( \cup_{k \ge n} A_k)=
\limsup_{n \to \infty} \mu( \cup_{k \ge n} A_k)\ge
\limsup_{n \to \infty} \mu( A_n ).    \qedhere
\end{equation*}
\end{proof}

\begin{proposition}
\label{PropositionRecurrenceKeaneDirection}
Let $\theta$ be any Keane direction on the translation surface $X$. Then for Lebesgue almost any $p\in X$ we have
$$
\wrecsup(\phi_\theta,p)=1.
$$
\end{proposition}

\begin{proof}
According to the first part of the statement of Lemma~\ref{LemmaMeasureTheoreticBounds}, it is enough to prove that 
$
\wrecsup(\phi_\theta,p)\geq1 
$ 
for Lebesgue almost any $p$. Recall that 
$ 
\leb=a_1\mu_1+\dots+a_g\mu_g,
$, 
according to Equation \eqref{EquationErgodicDecompositionLeb}, where $\mu_1,\dots,\mu_g$ are the ergodic measures for $\phi_\theta$, and assume without loss of generality that $a_i>0$ strictly for any $i=1,\dots,g$ (otherwise just consider those $i$ with $a_i>0$). Fix any $i$ with $i=1\leq i\leq g$ and $\mu:=\mu_i$, with $a:=a_i>0$, so that $\mu(E)\leq a^{-1}\cdot \leb(E)$ for any Borel set $E\subset X$. It is enough to prove that 
$
\wrecsup(\phi_\theta,p)\geq1 
$ 
for $\mu$ almost any $p$. Let $(\pi,\lambda,\tau)$ be data satisfying Condition \eqref{EquationRepresentationByZipperedRectangles}, that is 
$
r_{-\theta}\cdot X=X(\pi,\lambda,\tau)
$. 
Following \S~\ref{SectionRauzyVeechOnZipperedRectangles}, for any $n\in\NN$ let
$
(\pi^{(n)},\lambda^{(n)},\tau^{(n)})
$ 
be the data obtained by Rauzy-Veech induction from $(\pi,\lambda,\tau)$, so that 
$
r_{-\theta}\cdot X=X(\pi^{(n)},\lambda^{(n)},\tau^{(n)})
$. 
In particular we have a $\mod\mu$ partition of $X$ into $d=\sharp\cA$ embedded and mutually disjoint open rectangles $R^{(n)}_\alpha$, $\alpha\in\cA$ as in Equation \eqref{EquationRenormalizedZipperedRectangles}. Let $\alpha\in\cA$ be a letter such that $\mu(R^{(n)}_\alpha)\geq 1/d$ for infinitely many $n$. For simplicity, assume without loss of generality that $\mu(R^{(n)}_\alpha)\geq 1/d$ for any $n\in\NN$. Thus for any $n$ set 
$$
r_n:=\frac{a}{4d}\cdot\lambda_\alpha^{(n)}
$$
and recall that $h^{(n)}_\alpha\to+\infty$ for $n\to\infty$ and also $r_n\to 0$ for $n\to\infty$, since $\lambda_\alpha^{(n)}\to0$ for $n\to\infty$. In the coordinate system of $R^{(n)}_\alpha$, consider the sub-rectangle
$$
D_n:= 
(r_n,\lambda_\alpha^{(n)}-r_n)\times (0,h^{(n)}_\alpha)\subset R^{(n)}_\alpha.
$$ 
Observe that for any $n$ we have 
$
\mu(D_n)=\mu(R^{(n)}_\alpha)-\mu(R^{(n)}_\alpha\setminus D_n)\geq 1/(2d)
$
because $\mu(R^{(n)}_\alpha)\geq1/d$ and 
$$
\mu(R^{(n)}_\alpha\setminus D_n)
\leq
\frac{\leb(R_n\setminus D_n)}{a}
=
\frac{2r_nh^{(n)}_\alpha}{a}
=
\frac{\lambda_\alpha^{(n)}h^{(n)}_\alpha}{2d}
\leq
\frac{1}{2d},
$$
thus Lemma~\ref{LemmaMeasureLimsupSet} implies $\mu(\limsup D_n)\ge (2d)^{-1}$. On the other hand for any $p\in D_n$ we have 
$
\rt(\phi_\theta,p,r_n) > h^{(n)}_\alpha- r_n
$. 
It follows that for any $p \in \limsup(D_n)$ we have
\begin{align*}
\wrecsup(p,\theta)
& 
\ge 
\limsup_{n\to\infty} 
\frac
{\log \rt(\phi_\theta,p,r_n)}
{-\log r_n} 
\ge 
\limsup_{n\to\infty} 
\frac
{\log (h^{(n)}_\alpha- r_n)}
{-\log r_n} 
=
\limsup_{n\to\infty} 
\frac
{\log h^{(n)}_\alpha}
{-\log r_n} 
\\
&
=
\limsup_{n\to\infty} 
\frac
{\log h^{(n)}_\alpha}
{-(a\lambda_\alpha^{(n)}/4)} 
\ge 
\limsup_{n\to\infty} 
\frac
{-\log\lambda_\alpha^{(n)} + \log(a/d)}
{-\log \lambda_\alpha^{(n)}- \log(a/4)} 
= 1,
\end{align*}
where the last equality in the first line holds since $r_n\to0$ and $h^{(n)}_\alpha\to+\infty$ for 
$
n\to\infty
$ 
and the last inequality follows observing that any rectangle $R^{(n)}_\alpha$ satisfies 
$$
\frac{1}{d}< 
\mu(R^{(n)}_\alpha)\leq
\frac{\leb (R^{(n)}_\alpha)}{a}= 
\frac{\lambda_\alpha^{(n)} h^{(n)}_\alpha}{a}.
$$
Ergodicity of $\mu$ and Lemma \ref{LemmaRecurrenceHittingConstant} imply 
$
\wrecsup(\phi_\theta,p)\geq 1
$ 
for $\mu$ a.e. $p\in X$. The Proposition is proved.
\end{proof}

\subsection{Ergodic directions}

In this subsection we assume that the Lebesque measure $\leb$ is ergodic for the flow $\phi_\theta$ in direction $\theta$ on the surface $X$.

\begin{proposition}
\label{PropositionLiminfHittingUniquelyErgodic}
Let $\theta$ be an ergodic direction on the translation surface $X$. Then for Lebesgue almost any $p,p'$ in $X$ we have
$$
\whitinf(\phi_\theta,p,p')=1.
$$
\end{proposition}

\begin{proof}
Let $(\pi,\lambda,\tau)$ be data satisfying Condition \eqref{EquationRepresentationByZipperedRectangles}. Following \S~\ref{SectionRauzyVeechOnZipperedRectangles}, for any $n\in\NN$ let
$
(\pi^{(n)},\lambda^{(n)},\tau^{(n)})
$ 
be the data obtained by Rauzy-Veech induction from $(\pi,\lambda,\tau)$, so that 
$
r_{-\theta}\cdot X=X(\pi^{(n)},\lambda^{(n)},\tau^{(n)})
$. 
Consider the $d=\sharp\cA$ open rectangles $R^{(n)}_\alpha$, $\alpha\in\cA$ given by Equation \eqref{EquationRenormalizedZipperedRectangles}. Let $\alpha\in\cA$ be a letter such that $\leb(R^{(n)}_\alpha)\geq 1/d$ for infinitely many $n$. For simplicity, assume without loss of generality that $\leb(R^{(n)}_\alpha)\geq 1/d$ for any $n\in\NN$. Since the rectangles $R^{(n)}_\alpha$ are embedded in $X$ then 
$
\leb(R^{(n)}_\alpha)=\lambda_\alpha^{(n)} h^{(n)}_\alpha<1
$, 
thus for any $n$ we have finally
$
(1/d)\leq\lambda_\alpha^{(n)} h^{(n)}_\alpha<1
$. 
For any $n$ set $r_n:=\lambda^{(n)}_\alpha$. Moreover, in the coordinate system of $R^{(n)}_\alpha$, set 
$$
E_n:=(0,\lambda_\alpha^{(n)}) \times (0,h^{(n)}_\alpha/2)
\quad
\textrm{ and }
\quad
F_n:=(0,\lambda_\alpha^{(n)}) \times (h^{(n)}_\alpha/2, h^{(n)}_\alpha),
$$
then note that $\leb(E_n)=\leb(F_n)\geq1/2d$, so that 
$
\leb( \limsup(E_n \times F_n) ) \ge 1/(4d^2)
$, 
according to Lemma~\ref{LemmaMeasureLimsupSet}. Observe that 
$
\rt (\phi_\theta,p,p',r_n) < h^{(n)}_\alpha
$ 
for any $p \in E_n$ and $p' \in F_n$, therefore for 
$
(p,p') \in \limsup(E_n \times F_n)
$ 
we have
\begin{align*}
\whitinf(p,p',\theta) 
&
\leq\liminf_{n \to \infty} 
\frac{\log \rt(\phi_\theta,p,p',r_n)}{-\log r_n} 
=
\liminf_{n \to \infty} 
\frac{\log \rt(\phi_\theta,p,p',\lambda_\alpha^{(n)})}{-\log \lambda_\alpha^{(n)}}
\\
&
\le\liminf_{n \to \infty} 
\frac{ \log h^{(n)}_\alpha}{-\log \lambda_\alpha^{(n)}} 
\le 
\liminf_{n \to \infty} 
\frac{ -\log \lambda_\alpha^{(n)} }{-\log \lambda_\alpha^{(n)}} = 1.
\end{align*}
Since $\leb$ is ergodic, then Lemma \ref{LemmaRecurrenceHittingConstant} implies 
$
\whitinf(\phi_\theta,p,p')\leq1
$ 
for $\leb\times\leb$ almost any $(p,p')$. The Proposition follows from the second part of Lemma \ref{LemmaMeasureTheoreticBounds}.
\end{proof}

\begin{proposition}
\label{PropositionLimsupHittingUniquelyErgodic(cylinders)}
Let $\theta$ be an ergodic direction on the translation surface $X$. Then for Lebesgue almost any $p,p'$ in $X$ we have
$$
\whitsup(\phi_\theta,p,p')\geq w^{cyl}(X,\theta).
$$
\end{proposition}

\begin{proof}
Fix any $w$ with $1\leq w\leq w^{cyl}(X,\theta)$. Since $w$ is arbitrary, it is enough to prove that for $\leb\times\leb$ almost any $(p,p')$ we have
$
\whitsup(\phi_\theta,p,p')\geq w
$. 
Moreover, since $\leb$ is ergodic under $\phi_\theta$, according to Lemma \ref{LemmaRecurrenceHittingConstant} it is enough to prove the last inequality for any pair of points $(p,p')$ is some subset of $X\times X$ with positive measure with respect to $\leb\times\leb$.

According to the definition of $w^{cyl}(X,\theta)$, there exists infinitely many closed geodesic 
$
\sigma_n
$ 
whose corresponding cylinder $C_n:=C_{\sigma_n}$ satisfies $\area(C_n)>a$ and such that
$$
|\re(\sigma_n,\theta)|<|\im(\sigma_n,\theta)|^{-w}.
$$
Let $\theta_n$ be the direction of $C_n$. Let also $H(C_n,\theta)$ be the \emph{orthogonal width} of $C_n$ with respect to $\theta$, that is the length of a segment in direction $\theta^\perp$ contained in the interior of $C_n$ and with both endpoints on $\partial C_n$, so that in particular we have
$
\area(C_n)=H(C_n,\theta)\cdot|\im(\sigma_n,\theta)|
$. 
Set 
$$
r_n:=H(C_n,\theta)/5.
$$
Consider $p_0\in\partial C_n$ and assume that $\phi_\theta^t(p_0)\in C_n$ for small $t>0$. The time $T_n>0$ needed to such $p_0$ to come back to $\partial C_n$ is 
$$
T_n:=
\frac{H(C_n,\theta)}{|\tan(\theta-\theta_n)|}=
H(C_n,\theta)\cdot\frac{|\im(\sigma_n,\theta)|}{|\re(\sigma_n,\theta)|}
\geq
H(C_n,\theta)\cdot|\im(\sigma_n,\theta)|^{w+1}.
$$ 
Let $E_n,F_n\subset C_\sigma$ be the subset defined by
\begin{eqnarray*}
&&
E_n:=
\{p\in C_n\textrm{ ; }\phi_\theta^t(p)\in C_n\textrm{ for }(1/5)T_n<t<(2/5)T_n\}
\\
&&
F_n:=
\{p'\in C_n\textrm{ ; }\phi_\theta^t(p')\in C_n\textrm{ for }(3/5)T_n<t<(4/5)T_n\},
\end{eqnarray*}
then observe that $\leb(E_n)=\leb(F_n)=\leb(C_n)/5\geq a/5$, so that Lemma \ref{LemmaMeasureLimsupSet} implies 
$
\leb\times\leb( \limsup E_n\times F_n) \ge (a/5)^2
$. 
Observe that for any $n$ and any $p\in E_n$ and $p'\in F_n$ we have 
$
R(\phi_\theta,p,p',r_n)\geq T_n/5
$, 
thus the Proposition follows because for any $(p,p')\in\limsup E_n\times F_n$ we have
\begin{align*}
\whitsup(p,p',\theta) 
&
\geq\limsup_{n \to \infty} 
\frac{\log \rt(\phi_\theta,p,p',r_n)}{-\log r_n} 
\geq
\limsup_{n \to \infty} 
\frac{\log T_n-\log 5}{-\log H(C_n,\theta)+\log 5 }
\\
&
\geq\limsup_{n \to \infty} 
\frac
{\log(H(C_n,\theta)\cdot|\im(\sigma_n,\theta)|^{w+1})}
{-\log H(C_n,\theta)} 
=w,
\end{align*}
where the last equality follows because for any $n$ we have
$$
a\leq \area(C_n)=H(C_n,\theta)\cdot|\im(\sigma_n,\theta)|\leq1.
$$

\end{proof}

\section{Upper bound for hitting time: proof of Theorem \ref{TheoremUpperBoundHittingTime}}
\label{SectionUpperBoundHittingTime}

In this section we prove Theorem \ref{TheoremUpperBoundHittingTime}. 
Fix a surface $X\in\cH(2)$ and a Keane direction $\theta$ on $X$. The Theorem follows directly from Proposition \ref{PropositionLimsupHittingTime} below. 
Indeed, considering a positive sequence $r_n\to0$ and applying the Proposition for any such $r_n$, one gets a sequence of saddle connections $\gamma_n$ on $X$ such that $|\re(\gamma_n,\theta)|\to0$, thus the saddle connection $(\gamma_n)_{n\in\NN}$ must form an infinite family, because $\theta$ is a Keane direction. 
Theorem~\ref{TheoremUpperBoundHittingTime} of course holds also for translation surfaces with $\area(X)\not=1$ because an homothety of the surface does not change $\whitsup(\phi_\theta,\cdot,\cdot)$.

\begin{proposition}
\label{PropositionLimsupHittingTime}
There exists a constant $C>1$, specific of $\cH(2)$, such that the following holds. 
Fix $\omega>1$ and let $X$ be a surface in $\cH(2)$ with $\area(X)=1$ and $\theta$ a Keane direction on $X$. Consider $r>0$ small enough and points $p,p'$ in $X$ with 
$$
\rt(\phi_\theta,p,p',r)\geq\left(\frac{1}{r}\right)^{-\omega}.
$$
Then there exists a saddle connection $\gamma$ on $X$ with $|\gamma|<r^{-\omega}$ such that  
$$
|\re(\gamma,\theta)|
\leq
\frac{C}{|\im(\gamma,\theta)|^{\sqrt{\omega}}}.
$$
\end{proposition}

\subsection{Long hitting time implies tall rectangles}

Let $\cH$ be a stratum of translation surfaces and $s_h,s_v:\cH\to\cH$ be involutions induced by the elements $s_h,s_v\in\gltwor$ acting by $s_h(x,y):=(-x,y)$ and $s_v(x,y):=(x,-y)$ on vectors $(x,y)\in\RR^2$. Let $D\simeq (\ZZ/2\ZZ)^2$ be the group generated by $s_h$ an $s_v$. If $\theta$ is a Keane direction on the surface $X\in\cH$, in order to compute $w(X,\theta)$ one can replace $X$ and $\theta$ by their images under any $s\in D$, indeed for any saddle connection $\gamma$ in $X$ the action of $D$ leaves invariant the quantities 
$
|\re(\gamma,\theta)|
$ 
and
$
|\im(\gamma,\theta)|
$. 
We say that the pair $(X,\theta)$ is represented by data $(\pi,\lambda,\tau)$ \emph{in the standard sense} if there exists data $(\pi,\lambda,\tau)$  such that Equation~\eqref{EquationRepresentationByZipperedRectangles} is satisfied, that is 
$
r_{-\theta}\cdot X=X(\pi,\lambda,\tau)
$. 
We say that the pair $(X,\theta)$ is represented by data $(\pi,\lambda,\tau)$ in the \emph{general sense} Equation~\eqref{EquationRepresentationByZipperedRectangles} is satisfied by some surface 
$
s\cdot(r_{-\theta}\cdot X)
$ 
with $s\in D$. Recall that we set 
$
\tau_\ast:=\sum_{\chi\in\cA}\tau_\chi
$. 
Moreover set also 
$
\|\lambda\|_1:=\sum_{\chi\in\cA}\lambda_\chi>0
$ 
and 
$
\|h\|_\infty:=\max_{\chi\in\cA}h_\chi>0
$.

\begin{lemma}
\label{LemmaGettingZipperedRectangles}
Fix $\omega>1$. Consider $r>0$ and points $p,p'$ in $X$ with 
$$
\rt(\phi_\theta,p,p',r)\geq r^{-\omega}.
$$
Then there are data $(\pi,\lambda,\tau)$ with $\tau_\ast<0$ representing the pair $(X,\theta)$ in the general sense such that
$$
\frac{r}{2}\leq\|\lambda\|_1\leq 2r
\quad
\textrm{ and }
\quad
\|h\|_\infty\geq r^{-\omega}.
$$ 
\end{lemma}

\begin{proof}
In order to simplify the notation, assume that $\theta=0$, which amounts to replace the surface $X$ by $r_{-\theta}X$. The flow $\phi_\theta$ thus corresponds to the flow in the vertical direction $\theta=0$, and is simply denoted $\phi$. In particular write 
$
\rt(p,p',r): = \rt(\phi_{\theta=0},p,p',r)
$ 
for points $p,p'$ in $X$ and for $r>0$. Replace $X$ by $s_v(X)$, that is invert the time of the vertical flow, and let $I$ be an horizontal interval in $X$ centered at $p'$ and with length $r$. 
Let $T:T\to I$ be the first return to $I$ of the vertical flow of $s_v(X)$, which a priori is an IET on $d+2$ intervals. Condition  
$
\rt(p,p',r)\geq r^{-\omega}
$ 
implies that $p$ must belong to a rectangle with hight at least $r^{-\omega}$. 

\emph{Step (1)}. Let $t\geq 0$ be minimal such that $\phi^t(I)\cap\Sigma\not=\emptyset$, then set $I':=\phi^t(I)$ and let $T':I'\to I'$ be the first return. We claim that $T'=T$, moreover the return time function has the same (constant) values on corresponding subintervals of $T$ and $T'$. This is because the vertical flow is trivial inside a \emph{flow box}, that is an open set of the form $\bigcup_{a<t<b}\phi^t(I)$ which does not intersect $\Sigma$. Then get $I'$ just by extending a flow box to its maximum. Therefore a rectangle with hight at least $r^{-\omega}$ persist for $T'$.

\emph{Step (2)} Let $I''$ be the biggest connected component of $I'\setminus\Sigma$. We have $|I''|\geq r$ (if $r$ is small enough than balls of radius $r$ around conical singularities are mutually disjoint, hence just 2 connected components). Let $T'':I''\to I''$ be the first return to $I''$. Since $T''$ is a first return of $T'$ onto a subset $I''\subset I''$, then it also has a rectangle with hight at least $r^{-\omega}$. On the other hand, $T''$ is an IET on $d+1$ intervals. Finally, modulo replacing $X$ by $s_h(X)$ one can assume that the left endpoint of $I''$ belongs to $\Sigma$.

\emph{Step (3)} Let $p_\ast$ be the right endpoint of $I''$, which a priori is not an element of $\Sigma$. Thus in general $T''$ has singularities $v_1,\dots,v_{d-1}$ whose positive $\phi$-orbit ends at points of $\Sigma$, plus one extra singularity $v_\ast$ whose $\phi$-orbit ends in $p_\ast$. The singularities of $(T'')^{-1}$ are the images $T''(v_1),\dots,T''(v_{d-1})$, corresponding to first intersection with $I''$ of positive $\phi$-trajectories starting at points of $\Sigma$, plus $T''(v_\ast)$, corresponding to the first intersection with $I''$ of the positive $\phi$-orbit of $p_\ast$. Let $I'''\subset I''$ be the subinterval with same left endpoint as $I''$, and whose right endpoint is the rightmost of the points $v_1,\dots,v_{d-1}$ and $T''(v_1),\dots,T''(v_{d-1})$, then let $T''':I'''\to I'''$ be the first return of $\phi$ to $I'''$. Let $(\pi,\lambda,\tau)$ be the data representing $X$ with $T'''=(\pi,\lambda)$. A rectangle with hight at least $r^{-\omega}$ persist for $T'''$, since the latter is a first return of $T''$. Modulo replacing $X$ by $s_v(X)$ one can also assume $\tau_\ast<0$. Finally, to see the estimate on $\|\lambda\|=|I'''|$, it is enough to observe that instead of shortening $I''$, one can extend it by a horizontal segment $I''''$ with $|I''''|>|I''|$ whose endpoint belongs to the $\phi$-orbit of $\Sigma$, ans then recovering $I'''$ from $I''''$ by a finite number of Rauzy steps, as the first renormalized interval with length shorter than $|I''|$. We get 
\begin{equation*}
2r\geq |I''|>|I'''|\geq |I''|/2\geq r/2.   \qedhere
\end{equation*}
\end{proof}

\subsection{Combinatorial Lemmas on zippered rectangles}
\label{SectionCombinatorialLemmasZipperedRectangles}

In this subsection we state some properties of the zippered rectangles construction introduced in \S~\ref{SectionRepresentationByZipperedRectangles}, that we will use in the next subsection. We consider the direction $\theta=0$ on a surface $X$, and we assume that $X$ is represented by data $(\pi,\lambda,\tau)$ in the standard sense. No normalization is required on the total area, except for Corollary \ref{CorollaryNoHighTopSingularites}, where $\sum_\chi\lambda_\chi h_\chi=1$. The first is an easy Lemma, whose proof is left to the reader.

\begin{lemma}
\label{LemmaClosedCylinderCondigurationZZXXX}
Let $X$ be a surface represented by data $(\pi,\lambda,\tau)$. Let $\alpha$ be the letter with 
$
\pi_t(\alpha)=d
$ 
and assume that 
$$
\sum_{\pi_b(\chi)\geq \pi_b(\alpha)+1}\lambda_\chi<\lambda_\alpha.
$$ 
Then the straight segment $\sigma$ connecting the center of the rectangles $R^b_D$ and $R^t_D$ corresponds to a closed geodesic on the surface $X$ with 
$$
|\re(\sigma)|=\sum_{\pi_b(\chi)\geq \pi_b(\alpha)+1}\lambda_\chi
\quad
\textrm{ and }
\quad
|\im(\sigma)|=h_\alpha.
$$
\end{lemma}

The next two Lemmas concern data $(\pi,\lambda,\tau)$ with $\tau_\ast<0$ where all singularities touch the rectangles in the top line very close to their lower horizontal side. We recall that when 
$
\tau_\ast<0
$ 
the discussion in \S~\ref{SectionPositionConicalPointsFlatRepresentation} applies for the position of points $\xi^{t/b}\alpha$ relatively to the rectangles $R^{t/b}_\alpha$.

Quantitative relations are stated in term of a parameter $M=M(d)$, depending only on the number of intervals, with $M(d)>>d$. For surfaces in $\cH(2)$, this situation leads to cylinder bounded by two homologous saddle connection, as in Figure \ref{FigureRemoveCilinder}. In general, Corollary \ref{CorollaryNoHighTopSingularites} holds, which represents a starting point for a proof by induction on genus of a general result extending Proposition \ref{PropositionLimsupHittingTime} to any stratum. Set 
$$
H=H(\pi,\tau):=\max_{\chi\in\cA}h_\chi=\|h\|_\infty.
$$

\begin{lemma}
\label{LemmaNoHighTopSingularites(1)}
Let $(\pi,\lambda,\tau)$ be a triple of data with $\tau_\ast<0$. Fix an integer $M\geq d+1$ and assume that 
$$
\max_{\chi\in\cA} \left \{ \im(\xi^t_\chi)\right \}
<
\frac{H}{M}.
$$
Let $A$ be an integer with $1\leq A\leq M-d$ and $\alpha\in\cA$ be a letter with $\pi^b(\alpha)\geq 2$ such that 
$$
\im(\xi^b_\alpha)<-\frac{M-A}{M}\cdot H.
$$
Then for any letter $\chi$ with 
$
\pi_b(\alpha)\leq\pi_b(\chi)\leq d
$ 
we have
$$
\im(\xi^b_\chi)<-\frac{M+2-A-d}{M}\cdot H.
$$
Moreover we also have
$$
\tau_\ast\leq -\frac{M+1-A-d}{M}\cdot H.
$$
\end{lemma}

\begin{proof}
We first prove the first part of the statement. If $\pi_b(\alpha)=d$ then the required property holds trivially. Otherwise, let $\alpha'$ be the letter with 
$
\pi_b(\alpha')=\pi_b(\alpha)+1
$ 
and observe that we have
$$
\im(\xi^b_{\alpha'})
=
\sum_{\pi_b(\chi)<\pi_b(\alpha')}\tau_\chi
=
\sum_{\pi_b(\chi)\leq\pi_b(\alpha)}\tau_\chi
=
-h_\alpha+\sum_{\pi_t(\chi)\leq\pi_t(\alpha)}\tau_\chi.
$$
Consider two cases.
\begin{enumerate}
\item
If $\pi_t(\alpha)=d$ then the assumption $\tau_\ast<0$ implies
$$
\im(\xi^b_{\alpha'})=
-h_\alpha+\tau_\ast
<
-h_\alpha
<
\im(\xi^b_\alpha)
\leq
-\frac{M-A}{M}H.
$$
\item
Otherwise, calling $\alpha''$ the letter such that $\pi_t(\alpha'')=\pi_t(\alpha)+1$, we have
$$
\im(\xi^b_{\alpha'})
=
-h_\alpha+\im(\xi^t_{\alpha''})
\leq 
-h_\alpha+\frac{H}{M}
\leq
\im(\xi^b_\alpha)+\frac{H}{M}
\leq 
-\frac{M-A-1}{M}H.
$$
\end{enumerate}
Replacing $\alpha$ by $\alpha'$ and repeating the argument recursively on $i$ with $1\leq i\leq d-2$, we get that if $\alpha_i$ is the letter with 
$
\pi_b(\alpha_i)=\pi_b(\alpha)+i
$ 
then we have
$$
\im(\xi^b_{\alpha_i})
\leq 
-\frac{M-A-i}{M}H.
$$
The first part of the statement is proved. In order to prove the second part, observe that if $\beta$ is the letter with $\pi_b(\beta)=d$ then we have $\pi_t(\beta)\leq d-1$ and thus
$
\im(\xi^t_\beta)+\tau_\beta<M^{-1}\cdot H
$, 
hence
\begin{eqnarray*}
&&
\tau_\ast=
\im(\xi^b_\beta)+\tau_\beta=
\im(\xi^t_\beta)-h_\beta+\tau_\beta=
-h_\beta+\frac{H}{M}\leq
\\
&&
\im(\xi^b_\beta)+\frac{H}{M}\leq
-\frac{M+2-A-d}{M}H+\frac{H}{M}\leq
\frac{M+1-A-d}{M}H.
\end{eqnarray*}
\end{proof}

\begin{lemma}
\label{LemmaNoHighTopSingularites(2)}
Let $(\pi,\lambda,\tau)$ be a triple of data with $\tau_\ast<0$. Fix an integer $M\geq 4d$ and assume that 
$$
\max_{\chi\in\cA}\left\{ \im(\xi^t_\chi) \right \}
<
\frac{H}{M}.
$$
Then we have
$$
\tau_\ast\leq -\frac{M-d}{M} H.
$$
Moreover, if $\alpha\in\cA$ is the letter with $\pi_t(\alpha)=d$ the following holds
\begin{eqnarray*}
&&
\im(\xi^b_\chi)<-\frac{M+2-2d}{M} H
\quad
\textrm{ for all letters with }
\pi_b(\alpha)+1\leq\pi_b(\chi)\leq d.
\\
&&
0\geq \im(\xi^b_\chi)\geq -\frac{2d-1}{M} H
\quad
\textrm{ for all letters with }
1\leq \pi_b(\chi)\leq \pi_b(\alpha).
\end{eqnarray*}
\end{lemma}

\begin{proof}
Since 
$
h_\chi=\im(\xi^t_\chi)-\im(\xi^b_\chi)
$ 
for any $\chi$, then the assumption implies that there exists $\alpha'$ with 
$
2\leq \pi_b(\alpha')\leq d
$ 
and
$$
\im(\xi^b_{\alpha'})\leq -\frac{M-1}{M}H.
$$  
The second part of Lemma \ref{LemmaNoHighTopSingularites(1)} with $A:=1$ implies 
$$
\tau_\ast<-\frac{M-d}{M} H.
$$
Now let $\alpha''$ be the letter with $\pi_b(\alpha'')=\pi_b(\alpha)+1$. We have
$$
\im(\xi^b_{\alpha''})=
\im(\xi^b_{\alpha})+\tau_\alpha=
-h_\alpha+\tau_\alpha+\sum_{\pi_t(\chi)<\pi_t(\alpha)}\tau_\chi=
-h_\alpha+\tau_\ast\leq
-\frac{M-d}{M} H.
$$
The first part of Lemma \ref{LemmaNoHighTopSingularites(1)} applied to the letter $\alpha''$ with $A:=d$ implies that for all letters $\chi$ with 
$
\pi_b(\alpha)+1\leq\pi_b(\chi)\leq d
$ 
we have
$$
\im(\xi^b_\chi)<-\frac{M+2-2d}{M}H.
$$
Finally, assume by absurd that there exists some $\beta$ with 
$
1\leq\pi_b(\beta)\leq\pi_b(\alpha)
$ 
and 
$$
\im(\xi^b_\beta)\leq -\frac{2d-1}{M}H.
$$ 
The first part of Lemma \ref{LemmaNoHighTopSingularites(1)} applied to the letter $\beta$ with $A:=M-2d+1$ implies
$$
\im(\xi^b_\alpha)\leq 
-\frac{M+2-(M-2d+1)-d}{M}H=
-\frac{d+1}{M}H.
$$
Then for the letter $\alpha''$ with $\pi_b(\alpha'')=\pi_b(\alpha)+1$ it follows
$$
-h_{\alpha''}\leq
\im(\xi^b_{\alpha''})=
\im(\xi^b_{\alpha})+\tau_\ast\leq
-\frac{M+1}{M}H
$$
which is absurd. The Lemma is proved.
\end{proof}

\begin{corollary}
\label{CorollaryNoHighTopSingularites}
Let $(\pi,\lambda,\tau)$ be a triple of data with $\tau_\ast<0$ and $\sum_\chi\lambda_\chi h_\chi=1$. Fix an integer $M\geq 4d$ and assume that 
$$
\max_{\chi\in\cA}\left \{ \im(\xi^t_\chi) \right \} 
<
\frac{H}{M}.
$$
Then on the surface $X$ corresponding to data $(\pi,\lambda,\tau)$ there exists a closed geodesic $\sigma$ with
$$
|\re(\sigma)|<\frac{(d-1)M}{M+2-2d}\frac{1}{H}
\quad
\textrm{ and }
\quad
|\im(\sigma)|<\frac{2d-1}{M}H.
$$
\end{corollary}

\emph{Note:} see the picture in Figure \ref{FigureRemoveCilinder} for the special case $\pi^b(\alpha)=d-1$.

\begin{proof}
Let $\alpha$ be the letter with $\pi_t(\alpha)=d$ and consider those $\chi$ with 
$
\pi_b(\alpha)+1\leq \pi_b(\chi)\leq d
$. 
For any such letter $\chi$ we have 
$
h_\chi\geq -\im(\xi^b_\chi)\geq 0
$, 
since the total area is one, and $\lambda_\chi<h_\chi^{-1}$. The Corollary follows directly from Lemma \ref{LemmaNoHighTopSingularites(2)}.
\end{proof}

\subsection{End of the proof of Proposition \ref{PropositionLimsupHittingTime}}
\label{SectionEndProofUpperBoundHitting}

In order to simplify the notation, assume that $\theta=0$, as in the proof of Lemma \ref{LemmaGettingZipperedRectangles}. This amounts to replace the surface $X$ by $r_{-\theta}X$. The flow $\phi_\theta$ thus corresponds to the flow in the vertical direction $\theta=0$. In particular, relative (or absolute) periods 
$
\hol(\gamma,\theta)=\big(\re(\gamma,\theta),\im(\gamma,\theta)\big)
$ 
take the form 
$
\hol  (\gamma,\theta=0)=\big(\re(\gamma),\im(\gamma)\big)
$. 
Moreover, for points $p,p'$ in $X$ and for $r>0$ write 
$
\rt(p,p',r):=\rt(\phi_{\theta=0},p,p',r)
$. 

\medskip

Consider $r>0$ small enough and points $p,p'$ in $X$ with $\rt(p,p',r)\geq r^{-\omega}$.
Let $(\pi,\lambda,\tau)$ be data representing the surface $X$ in the general sense as in Lemma~\ref{LemmaGettingZipperedRectangles}. 
Assume without loss of generality that $(\pi,\lambda,\tau)$ represent $X$ in the standards sense. Moreover recall that we have
$$
\sum_\chi\lambda_\chi h_\chi=\area(X)=1.
$$ 

Let $\cA=\{A,B,C,D\}$ be the alphabet for the Rauzy class $\cR$ of $\pi$. Let also fix the names of the letters in the first row of $\pi$, that is set 
$$
\pi_t=(A,B,C,D).
$$
The proof of Proposition \ref{PropositionLimsupHittingTime} follows by separate analysis of the cases listed below. 

\medskip

\emph{Case (1)} Suppose that there is a subset $\cA^{(1)}\subset\{A,B,C,D\}$ with $\sharp\cA^{(1)}=3$ such that $\lambda_\chi<r^{\sqrt{\omega}}$ for any $\chi\in\cA^{(1)}$. In this case the Proposition follows from an straightforward generalization of Lemma \ref{LemmaClosedCylinderCondigurationZZXXX}, which gives a closed geodesic $\sigma$ with 
$
|\re(\sigma)|<3r^{\sqrt{\omega}}
$ 
and 
$
|\im(\sigma)|<(r/2-3r^{\sqrt{\omega}})^{-1}<3r^{-1}
$.

\medskip

\emph{Case (2)} Suppose that there is a subset $\cA^{(2)}\subset\{A,B,C,D\}$ with $\sharp\cA^{(2)}=2$ such that 
$
\im(\xi^t_\chi)\geq r^{-\sqrt{\omega}}
$ 
for any $\chi\in\cA^{(2)}$. Then condition $\tau_\ast<0$ implies that there exists 
$
\cA^{(1)}\subset\{A,B,C,D\}
$ 
with $\sharp\cA^{(1)}=3$ such that for any $\chi\in\cA^{(1)}$ we have 
$
h_\chi\geq r^{-\sqrt{\omega}}
$ 
and thus 
$
\lambda_\chi<r^{\sqrt{\omega}}
$. 
Thus the proof follows as in Case (1).

\medskip

\emph{Case (3)} Suppose that there exists an unique $\alpha\in\{A,B,C,D\}$ with 
$
\im(\xi^t_\alpha)\geq r^{-\sqrt{\omega}}
$. 
Moreover assume also that such letter $\alpha$ satisfies
$$
\im(\xi^t_\alpha)\geq \frac{1}{100 \cdot r^{\omega}}.
$$

If $\alpha=B$ then condition $\tau_\ast<0$ implies 
$
h_\chi\geq (100 \cdot r^\omega)^{-1}
$ 
and thus $\lambda_\chi<100\cdot r^\omega$ for $\chi=A,B$. Moreover 
$
0<\im(\xi^t_C)<r^{-\sqrt{\omega}}<<\im(\xi^t_B)
$, 
therefore 
$
\gamma:=\xi^t_C-\xi^t_A=\zeta_A+\zeta_B
$ 
corresponds to a saddle connection on the surface $X$. Moreover such $\gamma$ satisfies 
$$
|\re(\gamma)|\leq 200 r^\omega
\quad
\textrm{ and }
\quad
|\im(\gamma)|\leq r^{-\sqrt{\omega}},
$$
thus the Proposition is proved. The case $\alpha=C$ is similar and $\gamma:=\xi^t_D-\xi^t_B$ corresponds to a saddle connection with the required properties. The last sub-case to consider is $\alpha=D$. In this case, observe first that arguing as above one has 
$
\lambda_\chi<100\cdot r^\omega
$ 
for $\chi=C,D$. Then let $\beta$ be the letter with $\pi_b(\beta)=\pi_b(D)+1$ and observe since 
$
\lambda_\beta\cdot h_\beta<1
$ 
then we also have 
$
\lambda_\beta\leq 100\cdot r^\omega
$, 
indeed condition $\tau_\ast<0$ implies
$$
-h_\beta<\im(\xi^b_\beta)=-h_D+\tau_\ast<-h_D<-(100\cdot r^\omega)^{-1}.
$$
If $\beta\not=C$ then the Proposition follows as in Case (1). Otherwise $\beta=C$, thus there is a letter $\beta'\not=C,D$ with $\pi_b(\beta')=4$. If 
$
\lambda_{\beta'}<r^{\sqrt{\omega}}
$ 
again the Proposition follows as in Case (1). Otherwise condition 
$
\lambda_{\beta'}\geq r^{\sqrt{\omega}}>>200 r^\omega>\lambda_C+\lambda_D
$ 
implies that 
$
\gamma:=\xi_\ast-\xi^t_C=\zeta_C+\zeta_D
$ 
is a saddle connection for the surface $X$, and moreover we have
$
\tau_\ast>-h_{\beta'}>r^{-\sqrt{\omega}}
$ 
thus such $\gamma$ satisfies 
$$
|\re(\gamma)|\leq 200 r^\omega
\quad
\textrm{ and }
\quad
|\im(\gamma)|\leq \im(\xi^t_C)+|\tau_\ast|<2\cdot r^{-\sqrt{\omega}}.
$$

\medskip

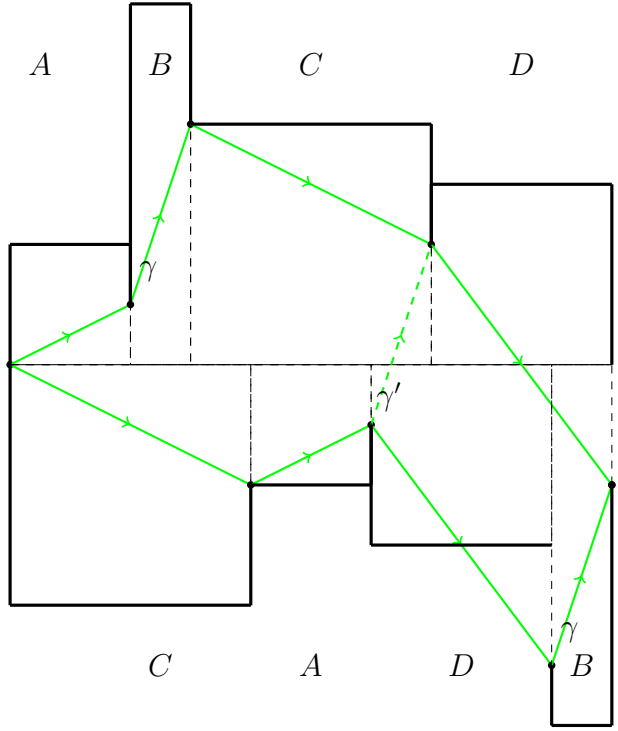
\begin{figure}
\begin{center} 
\begin{tikzpicture}[scale=0.8]

\tikzset
{->-/.style={decoration={markings,mark=at position .5 with {\arrow{>}}},postaction={decorate}}}



\draw[-] (0,0) -- (10,0);
\node [circle,fill,inner sep=1pt] at (0,0) {};


\draw[->-,thick,green] (0,0) -- (2,1);
\node [circle,fill,inner sep=1pt] at (2,1) {};

\draw[->-,thick,green] (2,1) -- (3,4) node[pos=0.3,below,black] {$\gamma$};
\node [circle,fill,inner sep=1pt] at (3,4) {};

\draw[->-,thick,green] (3,4) -- (7,2);
\node [circle,fill,inner sep=1pt] at (7,2) {};

\draw[->-,thick,green] (7,2) -- (10,-2);
\node [circle,fill,inner sep=1pt] at (10,-2) {};


\draw[->-,thick,green] (0,0) -- (4,-2);
\node [circle,fill,inner sep=1pt] at (4,-2) {};

\draw[->-,thick,green] (4,-2) -- (6,-1);
\node [circle,fill,inner sep=1pt] at (6,-1) {};

\draw[->-,thick,green] (6,-1) -- (9,-5);
\node [circle,fill,inner sep=1pt] at (9,-5) {};

\draw[->-,thick,green] (9,-5) -- (10,-2) node[pos=0.3,below,black] {$\gamma$};
\node [circle,fill,inner sep=1pt] at (10,-2) {};


\draw[->-,thick,green,dashed] (6,-1) -- (7,2) node[pos=0.3,below,black] {$\gamma'$};


\draw[-,very thin,dashed] (0,0) rectangle (2,2);
\node[] at (0.5,5) {$A$};
\draw[-,very thick] (0,0) -- (0,2);
\draw[-,very thick] (0,2) -- (2,2);
\draw[-,very thick] (2,1) -- (2,2);

\draw[-,very thin,dashed] (2,0) rectangle (3,6);
\node[] at (2.5,5) {$B$};
\draw[-,very thick] (2,1) -- (2,6);
\draw[-,very thick] (2,6) -- (3,6);
\draw[-,very thick] (3,4) -- (3,6);

\draw[-,very thin,dashed] (3,0) rectangle (7,4);
\node[] at (5,5) {$C$};
\draw[-,very thick] (3,4) -- (7,4);
\draw[-,very thick] (7,2) -- (7,4);

\draw[-,very thin,dashed] (7,0) rectangle (10,3);
\node[] at (8.5,5) {$D$};
\draw[-,very thick] (7,2) -- (7,3);
\draw[-,very thick] (7,3) -- (10,3);
\draw[-,very thick] (10,0) -- (10,3);


\draw[-,very thin,dashed] (0,-4) rectangle (4,0);
\node[] at (2.5,-5) {$C$};
\draw[-,very thick] (0,-4) -- (0,0);
\draw[-,very thick] (0,-4) -- (4,-4);
\draw[-,very thick] (4,-4) -- (4,-2);

\draw[-,very thin,dashed] (4,-2) rectangle (6,0);
\node[] at (5,-5) {$A$};
\draw[-,very thick] (4,-2) -- (6,-2);
\draw[-,very thick] (6,-2) -- (6,-1);

\draw[-,very thin,dashed] (6,-3) rectangle (9,0);
\node[] at (7.5,-5) {$D$};
\draw[-,very thick] (6,-3) -- (6,-1);
\draw[-,very thick] (6,-3) -- (9,-3);

\draw[-,very thin,dashed] (9,-6) rectangle (10,0);
\node[] at (9.5,-5) {$B$};
\draw[-,very thick] (9,-6) -- (9,-5);
\draw[-,very thick] (9,-6) -- (10,-6);
\draw[-,very thick] (10,-6) -- (10,-2);
\end{tikzpicture}

\end{center}
\caption{The homologous saddle connection $\gamma$ and $\gamma'$, appearing in Case (4) in the proof of Proposition \ref{PropositionLimsupHittingTime}, when $\pi^b(D)=3$.}
\label{FigureRemoveCilinder}
\end{figure}

\emph{Case (4)} The last case to consider is when for any $\alpha$ with $\pi^t(\alpha)\geq 2$ we have
$$
\im(\xi^t_\alpha)\leq \frac{1}{100\cdot r^\omega}.
$$
According to Lemma \ref{LemmaNoHighTopSingularites(2)}, where $M=100$ and $H=r^{-\omega}$ we have 
$
\lambda_\chi<100\cdot r^\omega
$ 
for any $\chi$ with $\pi_b(\chi)\geq \pi_b(D)+1$. If $\pi_b(D)=1$ then the Proposition follows immediately as in Case (1). If $\pi_b(D)=2$ and $h_D>r^{-\sqrt{\omega}}$ then 
$
\lambda_D<r^{\sqrt{\omega}}
$, 
thus again we have 3 short intervals and we conclude as in Case (1). Finally, if $\pi_b(D)=2$ and 
$
h_D\leq r^{-\sqrt{\omega}}
$ 
then the Proposition follows because Lemma \ref{LemmaClosedCylinderCondigurationZZXXX} provides a closed geodesic $\sigma$ on the surface $X$ with 
$$
|\re(\sigma)|=\sum_{\pi_b(\chi)\geq \pi_b(D)+1}\lambda_\chi<200\cdot r^\omega
\quad
\textrm{ and }
\quad
|\im(\sigma)|=h_D\leq r^{-\sqrt{\omega}}.
$$
The last sub-case to consider is $\pi_b(D)=3$. In this case, let $\beta$ be the letter with 
$
\pi_b(\beta)=\pi_b(D)+1=4
$ 
and observe that
$
\zeta_\beta=\xi^t_D-\xi^b_D
$, 
so that $\im(\zeta_\beta)=h_D>0$ and 
$$
\im(\xi^b_\beta)=-h_D+\tau_\ast<-h_D,
$$
whereas the remaining singularities $\xi^b_\chi$ with $2\leq \pi_b(\chi)\leq 3$ lie on the vertical boundary of both rectangles at their left and at their right. Moreover $1\leq\pi_t(\beta)\leq 2$, since $\pi$ is admissible, thus let $\alpha\not=D,\beta$ be the letter with $\pi_t(\alpha)=3$. Let also 
$
\alpha'\not=D,\beta
$ 
be the letter with $\pi_b(\alpha')=2$, where we may have $\alpha=\alpha'$. Finally, observe also that $\gamma:=\zeta_\beta$ and $\gamma':=\xi^t_D-\xi^b_D$ corresponds to two homologous saddle connections on the surface $X$, since they bound a cylinder (see Figure \ref{FigureRemoveCilinder}). We have 
$
|\re(\gamma)|=|\re(\gamma')|=\lambda_\beta<100\cdot r^\omega
$, 
according to Lemma \ref{LemmaNoHighTopSingularites(2)}. If 
$
|\im(\zeta_\beta)|<r^{-\sqrt{\omega}}
$ 
then the Proposition follows since the saddle connection $\gamma$ (and $\gamma'$ too) satisfies
$$
|\re(\gamma)|<100\cdot r^\omega
\quad
\textrm{ and }
\quad
|\im(\gamma)|<r^{-\sqrt{\omega}}.
$$
Otherwise, since 
$
\xi^t_D=\xi^b_D+\zeta_\beta
$ 
we have either 
$
\im(\xi^t_D)>(1/2)r^{-\sqrt{\omega}}
$ 
and thus 
$$
h_\alpha,h_D\geq \im(\xi^t_D)>(1/2)r^{-\sqrt{\omega}},
\quad
\textrm{ so that }
\quad
\lambda_\alpha,\lambda_D<2r^{\sqrt{\omega}},
$$
or
$
\im(\xi^b_D)<-(1/2)r^{-\sqrt{\omega}}
$ 
and thus 
$$
h_{\alpha'},h_D\geq -\im(\xi^b_D)>(1/2)r^{-\sqrt{\omega}},
\quad
\textrm{ so that }
\quad
\lambda_{\alpha'},\lambda_D<2r^{\sqrt{\omega}}.
$$
Since $\lambda_\beta<100\cdot r^\omega$ then in both cases the Proposition follows as in Case (1). The analysis of cases is complete and Proposition \ref{PropositionLimsupHittingTime} is proved. $\qed$

\section{Geometric constructions on reduced origamis}
\label{SectionGeometricConstructionsOrigamis}

It is convenient to consider the slope $\alpha=\tan\theta$ of a direction $\theta$ on an origami $X$, rather than the direction $\theta$ itself. That is for any $\alpha\in\RR\cup\{\pm\infty\}$ we consider the linear flow $\phi_\alpha:X\to X$, as the integral flow of the constant unitary vector field $e_\alpha$ on $X\setminus\Sigma$ defined by 
\begin{equation}
\label{EquationArctan}
e_\alpha:=(\sin\theta,\cos\theta)
\quad
\textrm{ where }
\quad
\theta:=\arctan{\alpha}\in(-\pi/2,\pi/2].
\end{equation}
Modulo the same change of variable, we will also refer to $(X,\alpha)$-singular leaves as trajectories of $\phi_\alpha$ starting or ending at singular points. In the following we will establish relations between flows $\phi_\alpha:X\to X$ and $\phi_{\alpha'}:X'\to X'$ in different slopes and on different surfaces. In order to avoid ambiguities, when considering $\phi_\alpha:X\to X$, for $r>0$ and for points $p,p'$ in $X$ we introduce the extended notation
$$
R(X,\alpha,p,p',r):=R(\phi_\alpha,p,p',r).
$$

\subsection{Reduced origamis and their orbits}
\label{SectionReducedOrigamisAndSL(2,Z)}

Recall from \S~\ref{SectionIntroductionOrigamis} that a translation surface $X$ is an origami if and only if $\hol(X)\subset\ZZ^2$, where $\hol(X)$ is the set of relative periods of $X$, and this is equivalent to say that the Veech group $\slgroup(X)$ of $X$ and $\sltwor$ share a finite index subgroup. We say that an origami $X$ is \emph{reduced} if 
$
\langle\hol(X)\rangle=\ZZ^2
$, 
that is the subgroup of $\RR^2$ generated by the set $\hol(X)$ is the entire lattice $\ZZ^2$. In this case Equation~\eqref{EquationCovarianceHolonomy} implies $A(\ZZ^2)\subset\ZZ^2$ for any $A\in\slgroup(X)$, that is $\slgroup(X)$ is a subgroup of $\sltwoz$ with finite index $[\sltwoz:\slgroup(X)]$. According to Equation~\eqref{EquationCovarianceHolonomy}, the action of $\sltwor$ on translation surfaces induces an action of $\sltwoz$ on origamis. Moreover, it is also clear that the action of $\sltwoz$ preserves the set of reduced origamis (see also Lemma 2.4 in \cite{HubertLelievre}). If $X$ is a reduced origami, denote by $\cO(X)$ its orbit under $\sltwoz$, that is
$$
\cO(X):=\{Y=A\cdot X \mid A\in\sltwoz\}.
$$
There is a natural identification 
$
\cO(X)=\sltwoz/\slgroup(X)
$, 
thus $\cO(X)$ is a finite set with cardinality 
$
N:=[\sltwoz:\slgroup(X)]<+\infty
$. 
The action of $\sltwoz$ passes to the quotient $\cO(X)$ and can be represented by a finite oriented graph, whose vertices are the elements $Y\in\cO(X)$ and whose oriented edges correspond to the operations $Y\mapsto T\cdot Y$ and $Y\mapsto V\cdot Y$ for $Y\in \cO(X)$, where we introduce the two generators 
$$
T:=
\begin{pmatrix}
1 & 1\\
0 & 1
\end{pmatrix}
\quad
\textrm{ and }
\quad
V:=
\begin{pmatrix}
1 & 0\\
1 & 1
\end{pmatrix}
$$
of $\sltwoz$. Fix $n\in\NN$ and consider positive integers $a_1,\dots,a_n$. Define the element $g(a_1,\dots,a_n)$ of $\sltwoz$ by
\begin{equation}
\label{EquationContinuedFractionSL(2,Z)}
g(a_1,\dots,a_n):=
\left\{
\begin{array}{c}
V^{a_1}\circ\dots\circ
V^{a_{n-1}}\circ T^{a_n}
\quad
\textrm{ for even }
n;
\\
V^{a_1}\circ\dots\circ
T^{a_{n-1}}\circ V^{a_n}
\quad
\textrm{ for odd }
n.
\end{array}
\right.
\end{equation}

\begin{lemma}
\label{LemmaConnectionCusps}
Let $\cO(X)$ be an orbit with $N\geq2$ elements. For any two elements $Y_2$ and $Y_1$ in $\cO(X)$ there exists a word $(a_1,\dots,a_{2m})$ with even length $2m$ with $0\leq 2m\leq 2N-2$ and $a_i\leq N$ for any $i=0,\dots,2m$, such that 
$$
Y_2=g(a_1,\dots,a_{2m})\cdot Y_1.
$$
\end{lemma}

\begin{proof}
Rational slopes $p/q$ are partitioned into \emph{cusps}, the latter being identified also with $T$-orbits over $\cO(X)$, see for example \S~2.6 in \cite{HubertLelievre}. For any $Y\in\cO(X)$ there exist two positive integers $h=h(Y)$ and $v=v(Y)$ with $h,v\leq N$, called \emph{width of the cusp} respectively of the horizontal $p/q=\infty$ and vertical $p/q=0$ slope of $Y$, such that $T^h\cdot Y=Y$ and $T^v\cdot Y=Y$. Since $\cO(X)$ is connected and has $N$ elements, there is a path in the letters $T$ and $V$ with length at most $N-1$ connecting $Y_1$ to $Y_2$. Contract subpaths which are the product of $a$ terms 
$
T\cdot\ldots\cdot T
$ 
into letters of the form $T^a$, and do the same for the generator $V$. This produces a word $(a_1,\dots,a_p)$ with length $0\leq p\leq N-1$ such that 
$
Y_2=g(a_1,\dots,a_{p})\cdot Y_1
$, 
moreover 
$
a_i\leq \max_{Y\in\cO(X)}\max\{h(Y),v(Y)\}\leq N
$ 
for any $i=0,\dots,p$. If $p$ is not even, recalling that $T^h\cdot Y_1=Y_1$ for $h=h(Y_1)$, get a word of even length 
$
2m=p+1
$ 
of the form $(a_1,\dots,a_{p},h(Y_1))$, where $2m\leq N$. Finally observe that $2N-2\geq N$ for $N\geq 2$. 
\end{proof}

\subsection{Continued fraction}

For any real number $\alpha$ let 
$
[\alpha]:=\max\{k\in\ZZ,k\leq \alpha\}
$ 
be its \emph{integer part} and $\{\alpha\}:=\alpha-[\alpha]$ be its \emph{fractional part}, where 
$
0\leq\{\alpha\}<1
$ 
by construction. The so-called \emph{Gauss map} is the application $G:(0,1)\to(0,1)$ defined for $\alpha\in(0,1)$ by 
$$
G(\alpha):=\{\alpha^{-1}\}.
$$

Fix $n$ positive integers $a_1,\dots,a_n$ and consider the rational number $[a_1,\dots,a_n]\in(0,1)$ defined by 
$$
[a_1,\dots,a_n]:=
\cfrac{1}{a_1+
\cfrac{1}{a_2+\dots+\cfrac{1}{a_n}.}}
$$
Any 
$
\alpha\in(0,1)\setminus\QQ
$ 
admits an unique \emph{continued fraction expansion} 
$
\alpha=[a_1,a_2,\dots]
$, 
where for any $n\geq1$ the entry $a_n=a_n(\alpha)\in\NN^\ast$ is known as the $n$-th \emph{partial quotient} of $\alpha$. 
Setting $\alpha_n:=G^n(\alpha_{n-1})$ for any $n \ge 1$ and $\alpha_0 = \alpha$, the entries $a_n$ are given by 
$$
a_n:=\left[\frac{1}{\alpha_{n-1}}\right]
\quad
\textrm{ that is }
\quad
\frac{1}{\alpha_{n-1}}=a_n+\alpha_n.
$$

For a finite word $(a_1,\dots,a_n)$ of positive integers, let 
$
I(a_1,\dots,a_n)\subset(0,1)
$ 
be the interval of those $\alpha\in(0,1)$ such that $G^{n}(\alpha)$ is defined with $a_i(\alpha)=a_i$ for any $i=1,\dots,n$. Set also 
$
p_n/q_n:=[a_1,\dots,a_n]
$, 
where $p_n=p_n(a_1,\dots,a_n)$ and $q_n=q_n(a_1,\dots,a_k)$. It is well know (see for example \S~12 in \cite{Khinchin}) that for any word $(a_1,\dots,a_n)$ we have
\begin{equation}
\label{EquationSizeCylinderContinuedFraction}
|I(a_1,\dots,a_n)|=\frac{1}{q_n(q_n+q_{n-1})},
\end{equation}
and moreover, for any pair of positive integers $1\leq A\leq B$ we have
\begin{equation}
\label{EquationBoundedDistortion}
\frac{1}{3}
\left(\frac{1}{A}-\frac{1}{B+1}\right)
\leq
\frac{\big|
\bigcup_{A\leq a_{n+1}\leq B}
I(a_1,\dots,a_n,a_{n+1})
\big|}
{|I(a_1,\dots,a_n)|}
\leq
2\left(\frac{1}{A}-\frac{1}{B+1}\right).
\end{equation}

Finally, let $g(a_1,\dots,a_n)$ be the element in $\sltwoz$ defined in Equation~\eqref{EquationContinuedFractionSL(2,Z)}. The group $\sltwoz$ acts on $\RR$ by homographies, that is the maps acting on $\alpha\in\RR$ by
$$
\alpha\mapsto
\begin{pmatrix}
a & b\\
c & d
\end{pmatrix}
\cdot
\alpha:=
\frac{a\alpha+b}{c\alpha+d}.
$$

\begin{lemma}
\label{LemmaActionSL(2,Z)Slopes}
The sequence of approximations $p_n/q_n$ of $\alpha=[a_1,a_2,\dots]\in(0,1)$ is given by
$$
p_n/q_n=
\left\{
\begin{array}{c}
g(a_1,\dots,a_n)\cdot 0
\quad
\textrm{ for even }
n\\
g(a_1,\dots,a_n)\cdot \infty
\quad
\textrm{ for odd }
n.
\end{array}
\right.
$$
Moreover, for any $k\in\NN$ we have
\begin{equation}
\label{EquationActionSL(2,Z)SlopesIrrational}
\alpha=
g(a_1,\dots,a_{2k})\cdot \alpha_{2k}=
g(a_1,\dots,a_{2k},a_{2k+1})\cdot \frac{1}{\alpha_{2k+1}}.
\end{equation}
\end{lemma}

\begin{proof}
Just recall that setting $(p_{-1},q_{-1})=(1,0)$ and $(p_{0},q_{0})=(0,1)$, the numerator and the denominator in $p_n/q_n$ satisfy
$
p_n=a_np_{n-1}+p_{n-2}
$ 
and 
$
q_n=a_nq_{n-1}+q_{n-2}
$ 
for any $n\geq1$. For $n=2k-1$ and $n=2k$ the same recursive relations take the form 
$$
\begin{pmatrix} p_{2k-1} & p_{2k}\\
q_{2k-1} & q_{2k}
\end{pmatrix}=
\begin{pmatrix}
p_{2k-1} & p_{2k-2}\\
q_{2k-1} & q_{2k-2}
\end{pmatrix}
\circ T^{a_{2k}}
\textrm{ and }
\begin{pmatrix}
p_{2k+1} & p_{2k}\\
q_{2k+1} & q_{2k}
\end{pmatrix}=
\begin{pmatrix}
p_{2k-1} & p_{2k}\\
q_{2k-1} & q_{2k}
\end{pmatrix}
\circ V^{a_{2k+1}}.
$$
The first part of the statement is proved. In order to prove second part consider three consecutive iterates
$
\alpha_n
$, 
$
\alpha_{n+1}
$ 
and
$
\alpha_{n+2}
$ 
of the Gauss map $G$ and observe that
\begin{align*}
\alpha_n &=
\frac{1}{a_{n+1}+\alpha_{n+1}}=
\begin{pmatrix}
0 & 1 \\
1 & a_{n+1}
\end{pmatrix}
\cdot
\alpha_{n+1}=
\begin{pmatrix}
0 & 1 \\
1 & a_{n+1}
\end{pmatrix}
\begin{pmatrix}
0 & 1 \\
1 & a_{n+2}
\end{pmatrix}
\cdot
\alpha_{n+2} \\
&=
\begin{pmatrix}
0 & 1 \\
1 & a_{n+1}
\end{pmatrix}
\begin{pmatrix}
0 & 1 \\
1 & 0
\end{pmatrix}
\begin{pmatrix}
0 & 1 \\
1 & 0
\end{pmatrix}
\begin{pmatrix}
0 & 1 \\
1 & a_{n+2}
\end{pmatrix}
\cdot
\alpha_{n+2}
=
\begin{pmatrix}
1       & 0 \\
a_{n+1} & 1
\end{pmatrix}
\begin{pmatrix}
1 & a_{n+2} \\
0 & 1
\end{pmatrix}
\cdot
\alpha_{n+2}.
\end{align*}
The relation between $\alpha$ and $\alpha_{2k}$ follows applying $k$ times the identity above to $\alpha_0=\alpha$. Then the relation with $\alpha_{2k+1}$ follows because
$
\alpha_{2k}=(a_{2k+1}+\alpha_{2k+1})^{-1}=V^{a_{2k+1}}\cdot \alpha_{2k+1}^{-1}
$. 
\end{proof}

\subsection{Flow segments in a cylinder}
\label{SectionFlowSegmentsVerticalCylinder}

Fix a reduced origami $X_0$ and let $\sigma_0$ be a closed geodesic with slope $p/q=0$, corresponding to the vertical direction. Let $C_0$ be the vertical cylinder corresponding to $\sigma_0$ and let $W_0$ be the transversal width of $C_0$, so that 
$
\area(C_0)=|\sigma_0|\cdot W_0
$. 
We have an isometric embedding  
$
(0,W_0)\times \RR/(|\sigma_0|\cdot\ZZ)\to C_0\subset X_0
$, 
and in order to establish coordinates $(x,y)$ with $x=x(p)\in(0,W_0)$ and 
$
y=y(p)\in\RR/(|\sigma_0|\cdot\ZZ)
$ 
for points $p\in C_0$ we consider the inverse
\begin{equation}
\label{EquationCoordinatesCylinder}
C_0\to (0,W_0)\times \RR/(|\sigma_0|\cdot\ZZ),
\qquad
p\mapsto\big(x(p),y(p)\big).
\end{equation}
Let $\partial C_0$ be the boundary of $C_0$ with respect to its intrinsic metric, and extend $x(p)$ and $y(p)$ continuously to points $p\in\partial C_0$. We have a disjoint union 
$
\partial C_0=\partial^{(L)}C_0\sqcup\partial^{(R)}C_0
$, 
where $\partial^{(L)}C_0$ and $\partial^{(R)}C_0$ are the components of $\partial C_0$ of those points with $x(p)=0$ and $x(p)=W_0$ respectively (on the other hand, the boundary of $C_0$ with respect to the metric of $X_0$ may have just one connected component). For points $p,p'$ in $C_0\cup\partial C_0$ set 
$
\delta_H(p,p'):=x(p)-x(p')
$ 
and 
$
\delta_V(p,p'):=y(p)-y(p')
$. 
Fix a slope $\alpha$ with 
\begin{equation}
\label{EquationSlopeVerticalEnough}
0<\alpha<|\sigma_0|^{-1},
\end{equation}
then set 
$$
T_0:=|\sigma_0|\cdot\sqrt{1+\alpha^2}
\quad
\textrm{ and }
\quad
T_1:=W_0\cdot\frac{\sqrt{1+\alpha^2}}{\alpha}.
$$ 
Inside $C_0$, trajectories of $\phi_\alpha$ travel in direction close to the vertical, and Equation \eqref{EquationSlopeVerticalEnough} implies that there are points $p$ such that $\phi_\alpha(p)\in C_0$ for any 
$
0\leq t\leq T_0
$. 
For $t=T_0$ the orbit $\phi_\alpha^t(p)$ comes back to the horizontal segment passing through $p$ and we have 
\begin{equation}
\label{EquationHorizontalTranslationVerticalCylinder}
\delta_H\big(p,\phi_\alpha^{T_0}(p)\big)=|\sigma_0|\cdot\alpha
\quad
\textrm{ and }
\quad
\delta_V\big(p,\phi_\alpha^{T_0}(p)\big)=0.
\end{equation}
For any $p_0\in\partial^{(L)}C_0$ we have
\begin{equation}
\label{EquationTimeInsideVerticalCylinder}
\phi_\alpha^t(p_0)\in C_0
\quad
\textrm{ for }
\quad
0<t<T_1.
\end{equation}
Moreover, for $t=T_1$, the orbit $\phi_\alpha^t(p_0)$ comes back to $\partial C_0$ with horizontal and vertical translation given by
\begin{equation}
\label{EquationVerticalTranslationVerticalCylinder}
\delta_H\big(p,\phi_\alpha^{T_1}(p)\big)=W_0
\quad
\textrm{ and }
\quad
\delta_V\big(\phi_\alpha^{T_1}(p_0),p_0\big)=
\frac{W_0}{\alpha}\mod |\sigma_0|\cdot\ZZ.
\end{equation}

Now fix a reduced origami $X$ and assume that $\cO(X)=\cO(X_0)$, that is $X$ and $X_0$ belongs to the same $\sltwoz$-orbit. Consider a slope 
$
\alpha=[a_1,a_2,\dots]\in(0,1)
$ 
on $X$ and for $n\in\NN$ let $\alpha_n=G^n(\alpha)$ be its $n$-th image under the Gauss map. Fix $n\in\NN$ and consider the first $2n$ entries $a_1,\dots,a_{2n}$ of the continued fraction of $\alpha$. Set 
$
A:=g(a_1,\dots,a_{2n})\in\sltwoz
$, 
which is defined by Equation~\eqref{EquationContinuedFractionSL(2,Z)}, and assume that 
$$
g(a_1,\dots,a_{2n})\cdot X_0=X.
$$ 
Assume also that the renormalized slope $\alpha_{2n}=G^{2n}(\alpha)$ satisfies Equation~\eqref{EquationSlopeVerticalEnough}, that is 
$$
\alpha_{2n}<|\sigma_0|^{-1}.
$$
The action $X_0\mapsto A\cdot X_0=X$ of $A\in\sltwoz$ induces an affine diffeomorphism 
\footnote{Actually, such $f_A$ is not unique, indeed there exist reduced origamis $X_0$ admitting non-trivial automorphisms, that is affine diffeomorphisms $f:X_0\to X_0$ with $Df=\id$. Nevertheless, non-unicity of $f_A$ does not affects the arguments in this paper.} $f_A:X_0\to X$, that is a diffeomorphism between the surfaces $X_0$ and $X$ whose derivative is constant with value  
$
Df_A=g(a_1,\dots,a_{2n})
$. 
Under such affine diffeomorphism, the flow $\phi_{\alpha_{2n}}$ with slope $\alpha_{2n}$ on the surface $X_0$ corresponds to the flow $\phi_\alpha$ with slope $\alpha$ on the surface $X$, where the two slopes are related under the homographic action of $A$ by Equation~\eqref{EquationActionSL(2,Z)SlopesIrrational}.

Observe that if $u,v\in\RR^2$ are linearly independent and 
$
\cV:=\{\lambda u+\mu u;\lambda,\mu\geq0\}\subset\RR^2
$ 
is such that $A\cV\subset \cV$ then for any $t$ with $0\leq t\leq 1$ we have
$$
\min
\left\{\frac{\|Au\|}{\|u\|},\frac{\|Av\|}{\|v\|}\right\}
\leq
\frac{\|A(u+t(v-u))\|}{\|u+t(v-u)\|}
\leq
\|A\|.
$$

Let $e_{\alpha}$ and $e_{\alpha_{2n}}$ be the unitary vectors with slopes $\alpha$ and $\alpha_{2n}$, which are defined in Equation~\eqref{EquationArctan}, so that in particular $Df_A(e_{\alpha_{2n}})$ is parallel to $e_\alpha$. Let also $e_0$ and $e_1$ be the unitary vectors with slope $\alpha=0$ and $\alpha=1$ respectively. Since for $0<\alpha<1$ we have $0<p_k/q_k<1$ for any $k$, then the estimate above can be applied to the cone $\cV$ spanned by $e_0$ and $e_1$. Observe that 
$
\|A\|\leq p_{2n}+q_{2n}+p_{2n-1}+q_{2n-1}<4q_{2n}
$ 
and that we have 
$$
Df_A(e_0)=(p_{2n},q_{2n})
\quad
\textrm{ and }
\quad
Df_A(e_1)=(\sqrt{2})^{-1}\cdot(p_{2n}+p_{2n-1},q_{2n}+q_{2n-1}).
$$
Therefore, if $I$ is a segment of trajectory of $\phi_{\alpha_{2n}}$, its image under $f_A$ has length $|f_A(I)|$ such that 
\begin{equation}
\label{EquationMaximalDilatation}
\frac{q_{2n}}{\sqrt{2}}\cdot |I|
\leq 
|f_A(I)|
\leq 
4\cdot q_{2n}\cdot |I|.
\end{equation}

\subsection{One cylinder directions}
\label{SectionOneCylinderVerticalDirections}

Fix a reduced origami $X_0$ and let $C_0\subset X_0$ be a cylinder in the vertical slope $p/q=0$. Let $S$ be a segment of straight line contained in the interior of $C_0$ and abusing the notation let $S:(0,1)\to C_0$ be a parametrization of it with constant speed $dS(t)/dt=(u_1,u_2)\in\RR^2$. The slope of such segment is $\alpha(S):=u_1/u_2$. We say that such segment $S$ is \emph{transversal to} $C_0$ if $S(0)\in\partial C_0$, 
$
S(1)\in\partial C_0
$ 
and the slope is negative, that is
$$
-\infty\leq \alpha(S)<-1.
$$
Observe that if $S$ is horizontal then $\alpha(S)=+\infty=-\infty$, thus it is simply transversal to $C_0$. Recall that  in the notation of \S~\ref{SectionFlowSegmentsVerticalCylinder} we set 
$
T_0:=|\sigma_0|\cdot\sqrt{1+\alpha_{2n}^2}
$.

\begin{lemma}
\label{LemmaVerticalDirection}
Fix $C_0$ and $\sigma_0$ as above and a slope  
$
\alpha=[a_1,a_2,\dots]\in(0,1)
$ 
satisfying Equation \eqref{EquationSlopeVerticalEnough}. Let $p_0\in\partial C_0$ be a point not on any $(X_0,\alpha)$-singular leaf and such that 
$
\phi_\alpha^t(p_0)\in C_0
$ 
for $0\leq t\leq T_0$. Then for any segment $S$ transversal to $C_0$ there exists $t\in\RR$ such that   
$$
\phi_{\alpha}^t(p_0)\in S
\quad
\textrm{ and }
\quad
0\leq t\leq T_0.
$$
\end{lemma}

\begin{proof}
It is enough to prove the Lemma when $S$ is horizontal, and in this case the statement follows directly applying Equation~\eqref{EquationHorizontalTranslationVerticalCylinder} to the point $p_0$
\end{proof}

Assume now that the vertical is a \emph{one cylinder direction} on the reduced origami $X_0$, that is there is only one cylinder $C_0$ in slope $p/q=0$. Let $\sigma_0$ be the corresponding vertical closed geodesic and observe that since $X_0$ is reduced, then the cylinder $C_0$ has transversal width $W_0=1$. The boundary $\partial C_0$ (with respect to the intrinsic metric of $C_0$) is composed by saddle connections parallel to $\sigma_0$, which appear in pairs, and the identification between paired saddle connection gives the surface $X_0$. Now let $X$ be a reduced origami and assume that $\sltwoz\cdot X$ contains $X_0$ as above.

\begin{proposition}
\label{PropositionOneCylinderDirection}
Fix a slope  
$
\alpha=[a_1,a_2,\dots]
$ 
and assume that for $n\in\NN$ we have 
\begin{eqnarray*}
&&
a_{2n+1}\geq|\sigma_0|
\\
&&
g(a_1,\dots,a_{2n})\cdot X_0=X.
\end{eqnarray*}
Then for any pair of points $p,p'$ in $X$, where $p$ does not belong to any $(X,\alpha)$-singular leaf, we have
$$
R(X,\phi_\alpha,p,p',r_n)\leq 16\cdot|\sigma_0|\cdot q_{2n}
\quad
\textrm{ where }
\quad
r_n:=\frac{2}{q_{2n}}
$$
\end{proposition}

\begin{proof} 
Following \S~\ref{SectionFlowSegmentsVerticalCylinder}, set 
$
A:=g(a_1,\dots,a_{2n})
$ 
and consider the corresponding affine diffeomorphism $f_A:X_0\to X$. Under $f_A$ the vertical cylinder $C_0\subset X_0$ corresponds to a cylinder $C_n$ with slope $p_{2n}/q_{2n}=A\cdot 0$ in the surface $X$. Let $\alpha_{2n}$ be the slope related to $\alpha$ by Equation~\eqref{EquationActionSL(2,Z)SlopesIrrational}, that is 
$
\alpha=A\cdot \alpha_{2n}
$. 
Since 
$
\alpha_{2n}<(a_{2n+1})^{-1}\leq|\sigma_0|^{-1}
$, 
then Equation \eqref{EquationSlopeVerticalEnough} is satisfied by the slope $\alpha_{2n}$ on the surface $X_0$. Let $p$ be a point as in the statement and observe that $p'_0:=f_A^{-1}(p)$ does not belong to any $(X_0,\alpha_{2n})$-singular leaf. Either there exists some $t(0)$ with $0\leq t(0)\leq T_0$ such that 
$
\phi_{\alpha_{2n}}^{t(0)}(p'_0)\in\partial C_0
$, 
and in this case we set 
$
p_0:=\phi_{\alpha_{2n}}^{t(0)}(p'_0)\in\partial C_0
$, 
or 
$
\phi_{\alpha_{2n}}^{t}(p'_0)
$ 
belongs to the interior of $C_0$ for $0\leq t\leq T_0$, and in this case we set $t(0):=0$ and $p_0:=p'_0$. In both cases, since $C_0$ is the only vertical cylinder in $X_0$, then the point $p_0$ satisfies the assumption of Lemma~\ref{LemmaVerticalDirection}, that is 
$
\phi_{\alpha_{2n}}^{t}(p_0)\in C_0
$ 
for any $t$ with $0\leq t\leq T_0$. Finally let $S^\perp\subset C_n$ be the segment in the surface $X$ passing through $p'$ with slope 
$
\alpha(S^\perp)=-q_{2n}/p_{2n}
$, 
that is orthogonal to the direction of the cylinder $C_n$, and with both endpoints on $\partial C_n$. The segment $S^\perp$ has length 
$
|S^\perp|=1/\sqrt{q_{2n}^2+p_{2n}^2}
$, 
indeed we have
$$
|\sigma_0|=\area(C_0)=|\sigma_0|\cdot\sqrt{q_{2n}^2+p_{2n}^2}\cdot|S^\perp|.
$$
The segment $S:=f_A^{-1}(S^\perp)$ in the surface $X_0$ has slope 
\begin{equation}
\label{EquationPropositionOneCylinderDirection}
\alpha(S)=
A^{-1}\cdot\frac{-q_{2n}}{p_{2n}}=
\begin{pmatrix}
q_{2n} & -p_{2n} \\
-q_{2n-1} & p_{2n-1}
\end{pmatrix}
\cdot\frac{-q_{2n}}{p_{2n}}=
\frac{-(q_{2n}^2+p_{2n}^2)}{q_{2n}q_{2n-1}+p_{2n}p_{2n-1}}
<-a_{2n}<-1,
\end{equation}
thus is transversal to $C_0$. According to Lemma~\ref{LemmaVerticalDirection} there exists some $t(1)\in\RR$ with $0\leq t(1)\leq T_0$ and such that  
$
\phi_{\alpha_{2n}}^{t(1)}(p_0)\in S
$, 
that is 
$$
\phi_{\alpha_{2n}}^{t(0)+t(1)}\big(f_A^{-1}(p)\big)\in S.
$$
Consider $T>0$ such that 
$
f_A\circ \phi_{\alpha_{2n}}^{t(0)+t(1)}=\phi_\alpha^T\circ f_A
$ 
and observe that 
$$
\phi_\alpha^T(p)=
\phi_\alpha^T\circ f_A \big(f_A^{-1}(p)\big)=
f_A\circ \phi_{\alpha_{2n}}^{t(0)+t(1)}\big(f_A^{-1}(p)\big)
\in f_A(S)=S^\perp.
$$
The Proposition follows because Equation~\eqref{EquationMaximalDilatation} implies 
$$
0\leq T\leq 4\cdot q_{2n}\cdot\big(t(0)+t(1)\big)
\leq
8\cdot q_{2n}\cdot |\sigma_0|\cdot\sqrt{1+\alpha_{2n}^2}
\leq 
16\cdot |\sigma_0|\cdot q_{2n},
$$
and on the other hand, since both $p'$ and $\phi_\alpha^T(p)$ belong to $S^\perp$, then 
$$
\left|
\phi_{\alpha}^{T}(p)-p'
\right|
\leq
|S^\perp|
\leq
\frac{1}{\sqrt{q_{2n}^2+p_{2n}^2}}
\leq
\frac{2}{q_{2n}}=r_n.
$$
\end{proof}

\subsection{Vertical splitting pairs}
\label{SectionVerticalSplittingPairs}

Let $X_0$ be a reduced origami. The vertical slope $p/q=0$ is completely periodic and $X_0$ is decomposed into vertical cylinders pasted together along their boundary. Let $\sigma_0$ and $\gamma_0$ be respectively a vertical closed geodesic and saddle connection on the surface $X_0$. Let $C_0$ be the vertical cylinder corresponding to $\sigma_0$ and $W_0$ be its transversal width.  We say that $(\sigma_0,\gamma_0)$ is a \emph{vertical splitting pair} for the surface $X_0$ if the following holds.
\begin{enumerate}
\item
The saddle connection $\gamma_0$ belongs to both the two components of the boundary 
$
\partial C_0
$ 
of the cylinder $C_0$ (with respect to its intrinsic metric). In other words $\gamma_0$ touches $C_0$ both from the left and the right side. 
\item
The closure of $C_0$ does not fill the entire surface $X_0$, that is there exists an other vertical cylinder $C_0'\not=C_0$.
\item
We have 
$
|\sigma_0|\wedge W_0=1
$, 
that is the positive integers $|\sigma|$ and $W(C_\sigma)$ are co-prime.
\end{enumerate}

Figure \ref{FigureVerticalSplittingPair} gives an example of vertical splitting pair. In \S~\ref{SectionProofExistenceSplittingPairsH(2)} we prove the Lemma \ref{LemmaExistenceSplittingPairs} below. The Lemma does not hold in all strata. For example in the stratum $\cH(1,1,1,1)$, a counterexample is given by an origami known by the German name \emph{Eierlegende Wollmilchsau}, a picture of which appears in Figure~\ref{FigureEierlegendeWollmilchsau} (see also \S~8 of \cite{ForniMatheus}).

\begin{lemma}
\label{LemmaExistenceSplittingPairs}
For any reduced origami $X$ in $\cH(2)$ there exists $X_0\in\cO(X)$ which admits a vertical splitting pair $(\sigma_0,\gamma_0)$. More precisely, we can always find a splitting pair $(\sigma_0,\gamma_0)$ with either $W_0=1$ or $W_0=2$ and $|\sigma_0|$ even.
\end{lemma}

Let $(\sigma_0,\gamma_0)$ be a vertical splitting pair for the surface $X_0$, consider the cylinder $C_0$ around $\sigma_0$ and its boundary $\partial C_0$ with respect to its intrinsic metric. Let $\gamma^{(L)}_0$ and $\gamma^{(R)}_0$ be the two representative of $\gamma_0$ in $\partial C_0$, which are two vertical segments whose horizontal distance is $W_0$. Consider the coordinates $x=x(p)$ and $y=y(p)$ given by Equation~\eqref{EquationCoordinatesCylinder}, and recall that 
$
y\in\RR/|\sigma_0|\cdot\ZZ
$. 
Then define $\Delta_0\in\NN$ as the unique integer with 
$
0\leq\Delta_0\leq|\sigma_0|-1
$ 
which gives the vertical distance between $\gamma^{(L)}_0$ and $\gamma^{(R)}_0$, that is 
\begin{equation}
\label{EquationDistancesBoundaryCylinder}
\delta_H(\gamma^{(L)}_0,\gamma^{(R)}_0)=W_0
\quad
\textrm{ and }
\quad
\delta_V(\gamma^{(L)}_0,\gamma^{(R)}_0)=\Delta_0.
\end{equation}
It is obvious that 
$
1\leq |\gamma_0|\leq |\sigma_0|-1
$ 
and 
$
0\leq \Delta_0+|\gamma_0|\leq |\sigma_0|
$. 
Given a vertical splitting pair $(\sigma_0,\gamma_0)$, the next Lemma gives a condition on the slope $\alpha$ for the existence of a subset 
$
E(\sigma_0,\gamma_0,\alpha)\subset C_0
$ 
with big measure of points whose $\phi_\alpha$-orbit remains in $C_0$ for a long time.  It is practical to observe also that if $W_0\in\NN^\ast$ and $\alpha\in(0,1)$ satisfy 
$
G(\alpha)<W^{-1}_0
$, 
since 
$
W_0\cdot\alpha^{-1}=W_0\cdot[\alpha^{-1}]+W_0\cdot\{\alpha^{-1}\}
$ 
and since of course $W_0\cdot[\alpha^{-1}]$ is an integer, we get
\begin{equation}
\label{EquationSmallG(Alpha)}
\left[\frac{W_0}{\alpha}\right]
=
W\cdot\left[\frac{1}{\alpha}\right]
\quad
\textrm{ and }
\quad
\left\{\frac{W_0}{\alpha}\right\}
=
W_0\cdot\left\{\frac{1}{\alpha}\right\}
=
W_0\cdot G(\alpha).
\end{equation}

\begin{lemma}
\label{LemmaSplittingPair}
Let $(\sigma_0,\gamma_0)$ be a vertical splitting pair on the surface $X_0$ and let $C_0$ be the corresponding cylinder. Fix a slope $\alpha\in(0,1)$ and consider the flow $\phi_\alpha$ on $X_0$. Assume that the following condition are satisfied 
\begin{eqnarray*}
&&
0<\alpha<|\sigma_0|^{-1}
\\
&&
0<G(\alpha)<W_0^{-1}
\\
&&
\left[\frac{W_0}{\alpha}\right]=
\Delta_0 \mod |\sigma_0|\cdot \ZZ.
\end{eqnarray*}
Then there is a subset 
$
E(\sigma_0,\gamma_0,\alpha)\subset C_0
$ 
with measure 
$
|E(\sigma_0,\gamma_0,\alpha)|=W_0\cdot|\gamma_0|/2
$ 
such that for any $p\in E(\sigma_0,\gamma_0,\alpha)$ we have
$$
\phi_\alpha^t(p)\in C_0
\quad
\textrm{ for any }
\quad
0\leq t\leq 
\frac{|\gamma_0|}{2}\cdot\frac{1}{\alpha\cdot G(\alpha)}.
$$
\end{lemma}

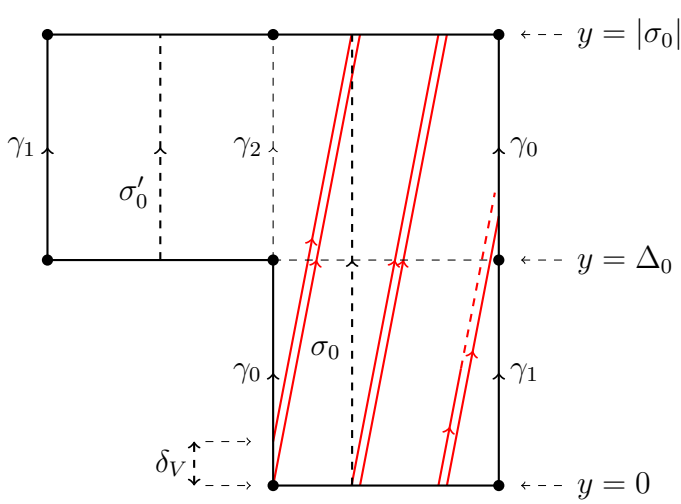
\begin{figure}[ht]
\begin{minipage}[c]{0.50\textwidth}
\begin{tikzpicture}[scale=0.3]

\tikzset
{->-/.style={decoration={markings,mark=at position .5 with {\arrow{>}}},postaction={decorate}}}


\draw[->-,thick,red] (0,0) -- (3.85,20);
\draw[->-,thick,red] (3.85,0) -- (7.70,20);
\draw[->-,thick,red] (7.70,0) -- (10,11.95);
\draw[->-,thick,red] (0,1.95) -- (3.475,20);
\draw[->-,thick,red] (3.475,0) -- (7.325,20);
\draw[->-,thick,red] (7.325,0) -- (8.325,5.195);
\draw[-,thick,red,dashed] (8.325,5.195) -- (9.825,12.99);


\node [circle,fill,inner sep=1.5pt] at (0,0) {};
\draw[->-,thick] (0,0) -- (0,10) node[pos=0.5,left] {$\gamma_0$};
\node [circle,fill,inner sep=1.5pt] at (0,10) {};
\draw[->-,thin,dashed] (0,10) -- (0,20) node[pos=0.5,left] {$\gamma_2$};
\node [circle,fill,inner sep=1.5pt] at (0,20) {};


\node [circle,fill,inner sep=1.5pt] at (10,0) {};
\draw[->-,thick] (10,0) -- (10,10) node[pos=0.5,right] {$\gamma_1$};
\node [circle,fill,inner sep=1.5pt] at (10,10) {};
\draw[->-,thick] (10,10) -- (10,20) node[pos=0.5,right] {$\gamma_0$};
\node [circle,fill,inner sep=1.5pt] at (10,20) {};


\node [circle,fill,inner sep=1.5pt] at (-10,10) {};
\draw[->-,thick] (-10,10) -- (-10,20) node[pos=0.5,left] {$\gamma_1$};
\node [circle,fill,inner sep=1.5pt] at (-10,20) {};


\draw[->-,thick,dashed] (3.5,0) -- (3.5,20) node[pos=0.3,left] {$\sigma_0$};

\draw[->-,thick,dashed] (-5,10) -- (-5,20) node[pos=0.3,left] {$\sigma_0'$};


\draw[-,thick] (0,0) -- (10,0);
\draw[-,thick] (-10,10) -- (0,10);
\draw[-,thick] (-10,20) -- (10,20);
\draw[-,thin,dashed] (0,10) -- (10,10);


\draw[<-,very thin,dashed] (11,0) -- (13,0) node[right] {$y=0$};

\draw[<-,very thin,dashed] (11,10) -- (13,10) node[right] {$y=\Delta_0$};

\draw[<-,very thin,dashed] (11,20) -- (13,20) node[right] {$y=|\sigma_0|$};

\draw[->,very thin,dashed] (-3,0) -- (-1,0);

\draw[->,very thin,dashed] (-3,1.95) -- (-1,1.95);

\draw[<->,thick,dashed] (-3.5,0) -- (-3.5,1.95);

\node at (-4.5,1) {$\delta_V$};

\end{tikzpicture}
\end{minipage}\hfill
\begin{minipage}[c]{0.50\textwidth}
\caption{A vertical splitting pair $(\sigma_0,\gamma_0)$, with two cylinders $C_0$ and $C_0'$ around $\sigma_0$ and $\sigma_0'$ respectively, where $W_0=\Delta_0=1$ and $|\sigma_0|=2$. For $\alpha\sim 0,1925$ we have $\alpha^{-1}\sim5,1948$, so that 
$
[W_0/\alpha]=5=1\mod 2\cdot\ZZ
$. 
A trajectory with slope $\alpha$ travels in $C_0$, then it crosses $\gamma_0$ and repeats the path inside $C_0$, modulo a vertical translation by 
$
\delta_V=G(\alpha)\cdot W_0\sim0,1948
$. 
As long as it re-enters inside $C_0$ such trajectory does not visit the cylinder $C_0'$.}
\label{FigureVerticalSplittingPair}
\end{minipage}
\end{figure}

\begin{proof}
Recall the notation in \S~\ref{SectionFlowSegmentsVerticalCylinder}, and in particular the time 
$
T_1:=W_0\cdot\alpha^{-1}\cdot\sqrt{1+\alpha^2}
$ 
needed by $\phi_\alpha$-orbits to travel from one component of $\partial C_0$ to the other. Let $\gamma^{(+)}_0$ and $\gamma^{(-)}_0$ be the two vertical segments such that 
$
\gamma_0=\gamma^{(-)}_0\cup \gamma^{(+)}_0
$ 
and 
$
|\gamma^{(-)}_0|=|\gamma^{(+)}_0|=|\gamma_0|/2
$, 
where $\gamma^{(+)}_0$ is above $\gamma^{(-)}_0$. Consider the set of points 
$$
E=E(\sigma_0,\gamma_0,\alpha):=
\bigcup_{0<t<T_1}\phi_\alpha^t(\gamma^{(-)}_0)
\subset C_0.
$$
We have obviously 
$
|E|=W_0\cdot|\gamma_0|/2
$. 
According to the last assumption on $\alpha$, there exists some $k\in\NN$ such that  
$$
\frac{W_0}{\alpha}=
k\cdot |\sigma_0|+\Delta_0+
\left\{\frac{W_0}{\alpha}\right\}.
$$
The last equation above, together with Equation~\eqref{EquationVerticalTranslationVerticalCylinder}, imply that for any $p_0\in\gamma^{(-)}_0$ we have 
$$
\delta_V\big(\phi_\alpha^{T_1}(p_0),p_0\big)
=
\frac{W_0}{\alpha}\mod |\sigma_0|\cdot\ZZ
=
\Delta_0
+\left\{\frac{W_0}{\alpha}\right\},
$$
and in particular 
$
\phi_\alpha^{T_1}(p_0)\in\gamma_0
$, 
according to Equation~\eqref{EquationDistancesBoundaryCylinder}. Since any $p\in E$ belongs to the orbit of some $p_0\in\gamma^{(-)}_0$ then for any $p\in E$ we have 
$
\phi_\alpha^{T_1}(p)\in E
$ 
and in particular 
$
\delta_H\big(\phi_\alpha^{T_1}(p),p\big)=0
$. 
Moreover, the equation above and Equation~\eqref{EquationSmallG(Alpha)} imply 
$$
\delta_V\big(\phi_\alpha^{T_1}(p),p\big)=
\left\{\frac{W_0}{\alpha}\right\}=
W_0\cdot G(\alpha).
$$
If $n\in\NN$ is such that  
$
n\cdot W_0\cdot G(\alpha)\leq |\gamma^{(+)}_0|
$ 
we can repeat the argument above $n$ times and we get that for any $p\in E$ we have 
$
\phi_\alpha^{nT_1}(p)\in E
$ 
and 
$$
\delta_V\big(\phi_\alpha^{nT_1}(p),p\big)=
n\cdot\left\{\frac{W_0}{\alpha}\right\}=
n\cdot W_0\cdot G(\alpha).
$$ 
Therefore any $p\in E$ stays in $C_0$ for time at least 
$$
T_1\cdot n_{max}\geq
W_0\cdot\frac{1}{\alpha}
\cdot
\frac{|\gamma^{(+)}_0|}{W_0\cdot G(\alpha)}.
$$
\end{proof}

\begin{proposition}
\label{PropositionSplittingPairs}
Let $X$ be a reduced origami and $X_0\in\cO(X)$ be an element in its orbit which admits a vertical splitting pair $(\sigma_0,\gamma_0)$. Let $\alpha=[a_1,a_2,\dots]$ be a slope on $X$ and assume that there exists $n\in\NN$ with
\begin{eqnarray*}
&&
a_{2n+1}\geq |\sigma_0|,
\\
&&
a_{2n+2}\geq W_0,
\\
&&
W_0\cdot a_{2n+1}=\Delta_0\mod |\sigma_0|,
\\
&&
g(a_1,\dots,a_{2n})\cdot X_0=X.
\end{eqnarray*} 
Then there exist two subsets $F_n,E_n\subset X$ with area $|F_n|\geq 1/3$ and $|E_n|\geq 1/2$ such that for any $p\in E_n$ and $p'\in F_n$ we have
$$
R(X,\alpha,p,p',r_n)
\geq 
\frac{a_{2n+2}\cdot a_{2n+1}\cdot q_{2n}}{\sqrt{8}}
\quad
\textrm{ where }
\quad
r_n:=\frac{1}{12\cdot q_{2n}}.
$$
\end{proposition}

\begin{proof} 
Following \S~\ref{SectionFlowSegmentsVerticalCylinder}, set 
$
A:=g(a_1,\dots,a_{2n})
$ 
and consider the corresponding affine diffeomorphism $f_A:X_0\to X$. Observe that according to the first two assumptions on $a_{2n+1}$ and $a_{2n+2}$ we have 
$
\alpha_{2n}<(a_{2n+1})^{-1}\leq|\sigma_0|^{-1}
$ 
and 
$
G(\alpha_{2n})=\alpha_{2n+1}<(a_{2n+2})^{-1}\leq W_0^{-1}
$, 
thus the first two assumptions in Lemma~\ref{LemmaSplittingPair} are satisfied by the slope $\alpha_{2n}$ on the surface $X_0$. Moreover, since $G(\alpha_{2n})\leq W_0^{-1}$, then we have 
$
[W_0\cdot\alpha_{2n}^{-1}]=W_0\cdot a_{2n+1}
$ 
according to Equation~\eqref{EquationSmallG(Alpha)}, thus the third condition on $a_{2n+1}$ in the statement implies
$$
\left[\frac{W_0}{\alpha_{2n}}\right]
=
W_0\cdot\left[\frac{1}{\alpha_{2n}}\right]
=
W_0\cdot a_{2n+1}
=\Delta_0\mod |\sigma_0|,
$$
and the third assumption in Lemma~\ref{LemmaSplittingPair} is also satisfied. Apply Lemma \ref{LemmaSplittingPair} to the slope $\alpha_{2n}$ on the surface $X_0$, with respect to the vertical splitting pair $(\sigma_0,\gamma_0)$. Let 
$
E=E(\sigma_0,\gamma_0,\alpha_{2n})\subset X_0
$ 
be the subset provided by the Lemma, whose area satisfies $|E|>1/2$, since $|\gamma_0|\geq 1$ and $W_0\geq 1$. Recalling that $(\sigma_0,\gamma_0)$ is a splitting pair, let $C_0'\subset X_0$ be a vertical cylinder in $X_0$ disjoint from $C_0$ and let $F\subset C_0'$ be the open set obtained from $C_0'$ removing the $r'$-neighborhood of its boundary $\partial C_0'$, where 
$
r':=W(C_0')/3
$. 
The set $F$ is foliated by vertical closed geodesics $\sigma_0'$ with length 
$
1\leq|\sigma_0'|\leq N-1
$ 
and has transversal with 
$
r':=W(C_0')/3
$, 
thus its area satisfies $|F|\geq 1/3$. The two required sets $E_n,F_n\subset X$ are 
$$
E_n:=f_A\big(E(\sigma_0,\gamma_0,\alpha_{2n})\big)
\quad
\textrm{ and }
\quad
F_n:=f_A\big(F\big).
$$
The lower bounds for the areas $|E_n|$ and $|F_n|$ follow trivially because the action of $\sltwoz$ preserve the areas. According to Lemma \ref{LemmaSplittingPair}, the time spent inside $C_0$ by orbits 
$
\phi^t_{\alpha_{2n}}(p)
$ 
of points 
$
p\in E(\sigma_0,\gamma_0,\theta_{2n})
$ 
is at least
$$
\Delta T_0=
\frac{1/2}{\alpha_{2n} \cdot G(\alpha_{2n})}\geq 
\frac{a_{2n+1}\cdot a_{2n+2}}{2}.
$$
The cylinders $C_0$ and $C_0'$ in direction $0$ on the surface $X_0$ correspond respectively to cylinder 
$
C_n:=f_A(C_0)
$ 
and 
$
C'_n:=f_A(C'_0)
$ 
in direction 
$
p_{2n}/q_{2n}=g(a_1,\dots,a_{2n})\cdot 0
$ 
on the surface $X$. According to Equation \eqref{EquationMaximalDilatation}, the time spent inside $C_n$ by orbits $\phi^t_{\alpha}(p)$ of points $p\in E_n$ is at least
$$
\Delta T_n\geq \frac{q_{2n}}{\sqrt{2}}\cdot \Delta T_0
\geq 
\frac{a_{2n+1}\cdot a_{2n+2}\cdot q_{2n}}{\sqrt{8}}.
$$
On the other hand the set $F\subset X_0$ has distance $w':=W(C'_0)/3$ from $\partial C_0'$. Therefore the set $F_n\subset X$ has distance $w_n$ from the boundary $\partial C'_n$ of $C'_n$ which satisfies
$$
w_n\geq
\frac{W(C_0')}{3}\cdot\frac{1}{\|A\|}
\geq
\frac{W(C_0')}{3}\cdot\frac{1}{4\cdot q_{2n}}
\geq
\frac{1}{12\cdot q_{2n}}=r_n.
$$
The Proposition follows observing that any $p\in E_n$ has distance at least $r_n$ from any point in $F_n$, thus for any pair of point $p'\in F_n$ and $p\in E_n$ we have 
$$
R(X,\alpha,p,p',r_n)
\geq
T_n\geq
(1/2)\cdot a_{2n+1}\cdot a_{2n+2}\cdot q_{2n}.
$$
\end{proof}

\section{A set of slopes with prescribed long hitting time}
\label{SectionProofTheoremOrigami}

Let $X$ be a reduced origami and let 
$
\cO(X)=\sltwoz\cdot X
$ 
be its orbit. Let $N$ be the cardinality of $\cO(X)$. Assume that the orbit $\cO(X)$ contains both an element $X_0$ with a vertical splitting pair and an element $X_0^{(\ast)}$ whose vertical is a one cylinder direction, so that in particular $N\geq2$. Denote $\sigma_0^{(\ast)}$ the vertical closed geodesic in $X_0^{(\ast)}$, and recall that since $X_0^{(\ast)}$ is reduced then the corresponding cylinder has transversal width 
$
W(C_{\sigma_0^{(\ast)}})=1
$. 
Let $(\sigma_0,\gamma_0)$ be the vertical splitting pair in $X_0$ and let $W_0$ and $\Delta_0$ be the integers defined in \S~\ref{SectionVerticalSplittingPairs}. We will use the following easy Lemmas.

\begin{lemma}
\label{LemmaSolvingCongruence}
There exists $m_0\in\{0,1,\dots,|\sigma_0|-1\}$ such that for any $a\in\NN$ we have
$$
W_0\cdot a=\Delta_0\mod |\sigma_0|
\quad
\Leftrightarrow
\quad
a=m_0\mod |\sigma_0|.
$$
\end{lemma}

\begin{proof}
Just recall that $|\sigma_0|$ and $W_0$ are co-prime.
\end{proof}

\begin{lemma}
\label{LemmaSizeEntry2p+1}
Fix a word $a_1,\dots,a_{2m}$ with $a_i\leq N$ for any $i=1,\dots,2m$ and a pair of elements $Y_1,Y_2\in\cO(X)$ such that 
$
g(a_1,a_2,\dots,a_{2m})\cdot Y_1=Y_2
$. 
Then there exists $a'_1$ with 
$$
|\sigma_0^{(\ast)}|\leq a'_1\leq |\sigma_0^{(\ast)}|+2N
$$
such that 
$
g(a'_1,a_2,\dots,a_{2m})\cdot Y_1=Y_2
$.
\end{lemma}

\begin{proof}
Recall the definition of $g(a_1,\dots,a_{2m})$ in Equation~\eqref{EquationContinuedFractionSL(2,Z)}, and that there exists $v\in\NN$ with $1\leq v\leq N$ such that $V^v\cdot Y_2=Y_2$. Write $a_1=k\cdot v+r$ and 
$
|\sigma_0^{(\ast)}|=l\cdot v+p
$ 
for $k,l,r,p$ in $\NN$, then set
$
a'_1:=(l+\max\{1,k\})\cdot v+r
$ 
and observe that for $j:=(l-k)+\max\{1,k\}$ we have 
$$
g(a'_1,a_2,\dots,a_{2m})\cdot Y_1
=
V^{j\cdot v}\cdot g(a_1,a_2,\dots,a_{2m})\cdot Y_1
=
V^{j\cdot v}\cdot Y_1=Y_2.
$$
\end{proof}

\subsection{Construction of a set of slopes}
\label{SectionConstructionCantorSlopes} 

A \emph{Cantor-like set} $\EE\subset(0,1)$ is a closed set of the form  
$
\EE:=\bigcap_{k\in\NN}\EE_k
$, 
where $\EE_0:=(0,1)$ and where for any $k\geq1$ the set $\EE_k\subset(0,1)$ is a finite union of mutually disjoint closed intervals with $\EE_{k}\subset\EE_{k-1}$. Fix $\eta\geq1$ and $s$ with $1\leq s\leq2$. 
Denote $$
\nu:=\eta^{s-1}.
$$
Proposition~\ref{PropositionConstructionCantorSlopes} below defines a Cantor-like set $\EE$ by specifying conditions on the entries $a_1,a_2,\dots$ of the continued fraction of elements 
$
\alpha\in\EE
$. 

\begin{proposition}
\label{PropositionConstructionCantorSlopes}
There exists a Cantor-like set $\EE=\EE(X,\eta,s)$ such that for any $\alpha\in\EE$ we have
$$
w(\alpha)=\eta.
$$
Moreover for any $\alpha\in\EE$ there exist integers 
$
\big(p(k)\big)_{k\in\NN}
$ 
and 
$
\big(n(k)\big)_{k\in\NN}
$ 
with $p(0)=n(0)=0$ such that for any $k\geq1$ we have $n(k-1)<p(k)<n(k)$ and 
\begin{align}
&
\label{EquationOneCylinderEntry2p}
g(a_1,\dots,a_{2p(k)})\cdot X_0^{(\ast)}=X
\\
&
\label{EquationOneCylinderEntry2p+1}
|\sigma_0^{(\ast)}|\leq a_{2p(k)+1} \leq |\sigma_0^{(\ast)}|+2N
\\
&
\label{EquationSplittingPairEntry2n}
g(a_1,\dots,a_{2n(k)})\cdot X_0=X
\\
&
\label{EquationSizeEntry2n+1}
q_{2n(k)}^{\eta-1}\leq 
a_{2n(k)+1}\leq 
(|\sigma_0|+1)\cdot q_{2n(k)}^{\eta-1}-1
\\
&
\label{EquationCongruenceEntry2n+1}
W_0\cdot a_{2n(k)+1}=\Delta_0\mod |\sigma_0|
\\
&
\label{EquationSizeEntry2n+2}
q_{2n(k)+1}^{\nu-1}\leq 
a_{2n(k)+2}\leq 
2\cdot q_{2n(k)+1}^{\nu-1}-1.
\end{align}
Finally, there exists a constant $C>0$, depending only on $N$, $|\sigma_0|$, $\eta$ and $\nu$ such that  
\begin{equation}
\label{EquationRelationEntries2p(k)And2p(k-1)}
q_{2p(k)}
\leq
C \cdot q_{2p(k-1)}^{\nu\cdot\eta}.
\end{equation}
\end{proposition}

The \emph{levels} $\EE_k$ for $k\geq1$ of $\EE$ are defined inductively according to the procedure described in \S~\ref{SectionFirstLevelCantor}, \S~\ref{SectionSecondLevelCantor}, \S~\ref{SectionInductiveAssumptionCantor}, \S~\ref{SectionOddLevelCantor}, and \S~\ref{SectionEvenLevelCantor} below. Then the proof of Proposition~\ref{PropositionConstructionCantorSlopes} is completed in \S~\ref{SectionProofPropositionCantor}. In particular, we will give two different definitions for \emph{odd levels} $\EE_{2k-1}$ and \emph{even levels} $\EE_{2k}$. We will use the following notion: a finite family $\cE$ of finite words $(a_1,\dots,a_{2m})$ of even length $2m\geq2$ is said \emph{disjoint} if there are not two words $(a_1,\dots,a_{2m})$ and $(a'_1,\dots,a'_{2m'})$ in $\cE$ such that $m'\geq m$ and $a'_i=a_i$ for $1\leq i\leq 2m$. Moreover, in virtue of Lemma~\ref{LemmaConnectionCusps}, for $k\geq1$, we will consider blocks 
$
(a_{2n(k-1)+3},\dots,a_{2p(k)})
$ 
with
\begin{equation}
\label{EquationConnectionToOneCylinder}
4\leq 2p(k)-2n(k-1)\leq 2N
\quad
\textrm{ ; } 
\quad
a_i\leq N
\quad
\forall
\quad
i=2n(k-1)+3,\dots,2p(k)
\end{equation}
and blocks 
$
(a_{2p(k)+1},\dots,a_{2n(k)})
$ 
with
$$
2\leq 2n(k)-2p(k)\leq 2N-2
\quad
\textrm{ ; } 
\quad
a_i\leq N
\quad
\forall
\quad
i=2p(k)+1,\dots,2n(k),
$$
which, according to Lemma~\ref{LemmaSizeEntry2p+1}, are then modified into blocks 
$
(a_{2p(k)+1},\dots,a_{2n(k)})
$ 
with
\begin{equation}
\label{EquationConnectionToSplittingPair}
\left\{
\begin{array}{l}
2\leq 2n(k)-2p(k)\leq 2N-2
\\
|\sigma_0^{(\ast)}|\leq a_{2p(k)+1}\leq |\sigma_0^{(\ast)}|+2N
\quad
\textrm{ ; } 
\quad
a_i\leq N
\quad
\forall
\quad
i=2p(k)+2,\dots,2n(k)
\end{array}
\right.
\end{equation}

\subsubsection{Level $(1)$.} 
\label{SectionFirstLevelCantor}

Let $\cP_1$ be the maximal disjoint family of finite words 
$
(a_1,\dots,a_{2p(1)})
$ 
which satisfy Equation~\eqref{EquationOneCylinderEntry2p}, where the first two entries are defined by $a_1:=1$ and $a_2:=1$ and where the block 
$
(a_3,\dots,a_{2p(1)})
$ 
satisfies Equation~\eqref{EquationConnectionToOneCylinder}. This is possible applying Lemma~\ref{LemmaConnectionCusps} with $Y_1:=X_0^{(\ast)}$ and $Y_2:=g(a_1,a_2)^{-1}\cdot X$, indeed we have
$$
g(a_1,\dots,a_{2p(1)})\cdot X_0^\ast
=
g(a_1,a_2)\cdot g(a_3,\dots,a_{2p(1)})\cdot X_0^\ast
=
g(a_1,a_2)\cdot g(a_1,a_2)^{-1}\cdot X=X.
$$
For any $(a_1,\dots,a_{2p(1)})\in\cP_1$, apply again Lemma~\ref{LemmaConnectionCusps} with 
$
Y_2=X_0^{(\ast)}
$ 
and $Y_1=X_0$, and then Lemma~\ref{LemmaSizeEntry2p+1}, and let 
$
\cN_1(a_1,\dots,a_{2p(1)})
$ 
be the maximal disjoint family of words $(a_{2p(1)+1},\dots,a_{2n(1)})$ such that 
$$
g(a_{2p(1)+1},\dots,a_{2n(1)})\cdot X_0=X_0^{(\ast)}
$$
and which satisfy Equation~\eqref{EquationConnectionToSplittingPair}. Let $\cE_1$ be the family of concatenated words 
\begin{eqnarray*}
&&
(a_1,\dots,a_{2n(1)}):=
(a_1,\dots,a_{2p(1)},a_{2p(1)+1},\dots,a_{2n(1)})
\\
&&
\textrm{ with }\quad
(a_1,\dots,a_{2p(1)})\in\cP_1
\\
&&
\textrm{ and }\quad 
(a_{2p(1)+1},\dots,a_{2n(1)})\in\cN_1(a_1,\dots,a_{2p(1)}).
\end{eqnarray*}
Any such word satisfies Equation~\eqref{EquationSplittingPairEntry2n}, that is 
\begin{eqnarray*}
&&
g(a_1,\dots,a_{2n(1)})\cdot X_0=
\\
&&
g(a_1,\dots,a_{2p(1)})\cdot g(a_{2p(1)+1},\dots,a_{2n(1)})\cdot X_0=
g(a_1,\dots,a_{2p(1)})\cdot X_0^{(\ast)}=X.
\end{eqnarray*}
For any $(a_1,\dots,a_{2n(1)})\in\cE_1$ let $\cF_1(a_1,\dots,a_{2n(1)})$ be the set of integers $a_{2n(1)+1}\in\NN$ which satisfy Equation~\eqref{EquationSizeEntry2n+1} and Equation~\eqref{EquationCongruenceEntry2n+1}. Let $\EE_1$ be the union, over $(a_1,\dots,a_{2n(1)})\in\cE_1$, of the intervals $E_1$ defined by
$$
E_1=
E(a_1,\dots,a_{2n(1)}):=
\bigcup_{a\in\cF_1(a_1,\dots,a_{2n(1)})}I(a_1,\dots,a_{2n(1)},a).
$$

\subsubsection{Level $(2)$.} 
\label{SectionSecondLevelCantor}

Let $\cE_2$ be the family of words of the form 
$
(a_1,\dots,a_{2n(1)},a_{2n(1)+1})
$, 
where 
$
(a_1,\dots,a_{2n(1)})\in\cE_1
$ 
and where the last entry satisfies  
$
a_{2n(1)+1}\in\cF_1(a_1,\dots,a_{2n(1)})
$. 
For any 
$
(a_1,\dots,a_{2n(1)+1})\in\cE_2
$ 
let
$
\cG_2(a_1,\dots,a_{2n(1)+1})\subset\NN^\ast
$ 
be the set of those $a_{2n(1)+2}\geq1$ which satisfy Equation~\eqref{EquationSizeEntry2n+2}. Finally let $\EE_2$ be the union, over
$
(a_1,\dots,a_{2n(1)+1})\in\cE_2
$, 
of the intervals  
$$
E(a_1,\dots,a_{2n(1)+1}):=
\bigcup_{a\in\cG_2(a_1,\dots,a_{2n(1)+1})}
I(a_1,\dots,a_{2n(1)+1},a).
$$

\subsubsection{Inductive assumption.} 
\label{SectionInductiveAssumptionCantor}

Fix $k\geq2$ and assume that the families $\cE_1,\dots,\cE_{2k-2}$ are defined, where in particular $\cE_{2k-2}$ is a family of words $(a_1,\dots,a_{2n(k-1)+1})$ of odd length, such that removing the last entry Equation~\eqref{EquationSplittingPairEntry2n} is satisfied, that is
$$
X=g(a_1,\dots,a_{2n(k-1)})\cdot X_0.
$$
Assume also that the closed sets 
$
\EE_1\supset\dots\supset\EE_{2k-2}
$ 
are defined, where any $\EE_i$ is a finite union of mutually disjoint closed intervals. Assume in particular that the intervals $E_{2k-2}$ in $\EE_{2k-2}$ are labeled by words 
$
(a_1,\dots,a_{2n(k-1)})\in\cE_{2k-2}
$. 
Finally, assume that for any interval 
$
E_{2k-2}=E(a_1,\dots,a_{2n(k-1)+1})
$ 
in the level $\EE_{2k-2}$, where
$
(a_1,\dots,a_{2n(k-1)+1})\in\cE_{2k-2}
$ 
is the corresponding word, there is a finite subset 
$
\cG_{2k-2}(a_1,\dots,a_{2n(k-1)+1})\subset\NN^\ast
$ 
such that 
$$
E(a_1,\dots,a_{2n(k-1)+1})=
\bigcup_{a\in\cG_{2k-2}(a_1,\dots,a_{2n(k-1)+1})}
I(a_1,\dots,a_{2n(k-1)+1},a).
$$

\subsubsection{Level $(2k-1)$.} 
\label{SectionOddLevelCantor}

Fix 
$$
(a_1,\dots,a_{2n(k-1)+1})\in\cE_{2k-2}
\quad
\textrm{ and }
\quad
a_{2n(k-1)+2}\in\cG_{2k-2}(a_1,\dots,a_{2n(k-1)+1})
$$
then let  
$
\cP_{2k-1}(a_1,\dots,a_{2n(k-1)+1}|a_{2n(k-1)+2})
$ 
be the maximal disjoint family of those words 
$
(a_{2n(k-1)+3},\dots,a_{2p(k)})
$ 
which satisfy Equation~\eqref{EquationConnectionToOneCylinder} and the condition
$$
g(a_{2n(k-1)+3},\dots,a_{2p(k)})\cdot X_0^{(\ast)}=
g(a_{2n(k-1)+1},a)^{-1}\cdot X_0.
$$
This is possible applying Lemma~\ref{LemmaConnectionCusps} with 
$
Y_2:=g(a_{2n(k-1)+1},a_{2n(k-1)+2})^{-1}\cdot X_0
$ 
and $Y_1:=X_0^{(\ast)}$. Then let 
$
\cP_{2k-1}(a_1,\dots,a_{2n(k-1)+1})
$ 
be the family of concatenated words
$$
(a_1,\dots,a_{2p(k)}):=
(a_1,\dots,a_{2n(k-1)+1},a_{2n(k-1)+2},a_{2n(k-1)+3},\dots,a_{2p(k)}),
$$ 
where 
$
a_{2n(k-1)+2}\in\cG_{2k-2}(a_1,\dots,a_{2n(k-1)+1})
$ 
and where the last block satisfies 
$$
(a_{2n(k-1)+3}\dots,a_{2p(k)})\in\cP_{2k-1}(a_1,\dots,a_{2n(k-1)+1}|a_{2n(k-1)+2}).
$$
Any such concatenated word satisfies Equation~\eqref{EquationOneCylinderEntry2p}, indeed, according to the inductive assumption, we have 
\begin{eqnarray*}
&&
g(a_1,\dots,a_{2p(k)})\cdot X_0^{(\ast)}=
\\
&&
g(a_1,\dots,a_{2n(k-1)+2})\cdot g(a_{2n(k-1)+3},\dots,a_{2p(k)})\cdot X_0^{(\ast)}=
\\
&&
g(a_1,\dots,a_{2n(k-1)+2})\cdot g(a_{2n(k-1)+1},a_{2n(k-1)+2})^{-1}\cdot X_0=
g(a_1,\dots,a_{2n(k-1)})\cdot X_0=X.
\end{eqnarray*}

Fix 
$
(a_1,\dots,a_{2p(k)})\in\cP_{2k-1}(a_1,\dots,a_{2n(k-1)+1})
$ 
and apply first Lemma~\ref{LemmaConnectionCusps} with 
$
Y_2=X_0^{(\ast)}
$ 
and $Y_1=X_0$, and then Lemma~\ref{LemmaSizeEntry2p+1}, and define
$
\cN_{2k-1}(a_1,\dots,a_{2p(k)})
$ 
as the maximal disjoint family of words  
$
(a_{2p(k)+1},\dots,a_{2n(k)})
$ 
which satisfy Equation~\eqref{EquationConnectionToSplittingPair} and are such that 
$$
g(a_{2p(k)+1},\dots,a_{2n(k)})\cdot X_0=X_0^{(\ast)}.
$$
Let $\cE_{2k-1}(a_1,\dots,a_{2n(k-1)+1})$ be the family of concatenated words 
\begin{eqnarray*}
&&
(a_1,\dots,a_{2n(k)}):=
(a_1,\dots,a_{2p(k)},a_{2p(k)+1},\dots,a_{2n(k)})
\\
&&
\textrm{ where }
\quad
(a_1,\dots,a_{2p(k)})\in\cP_{2k-1}(a_1,\dots,a_{2n(k-1)+1})
\\
&&
\textrm{ and }
\quad
(a_{2p(k)+1},\dots,a_{2n(k)})\in\cN_{2k-1}(a_1,\dots,a_{2p(k)}).
\end{eqnarray*}
Any such concatenation satisfies Equation~\eqref{EquationSplittingPairEntry2n}, indeed we have
\begin{eqnarray*}
&&
g(a_1,\dots,a_{2n(k)})\cdot X_0=
\\
&&
g(a_1,\dots,a_{2p(k)})\cdot g(a_{2p(k)+1},\dots,a_{2n(k)})\cdot X_0=
g(a_1,\dots,a_{p(k)})\cdot X_0^{(\ast)}=X.
\end{eqnarray*}

Define the family $\cE_{2k-1}$ as the union, over 
$
(a_1,\dots,a_{2n(k-1)+1})\in\cE_{2k-2}
$, 
of the families 
$
\cE_{2k-1}(a_1,\dots,a_{2n(k-1)+1})
$. 
For any $(a_1,\dots,a_{2n(k)})\in\cE_{2k-1}$ let $\cF_{2k-1}(a_1,\dots,a_{2n(k)})$ be the set of integers $a_{2n(k)+1}\in\NN$ which satisfy Equation~\eqref{EquationSizeEntry2n+1} and Equation~\eqref{EquationCongruenceEntry2n+1}. Finally, let $\EE_{2k-1}$ be the union, over $(a_1,\dots,a_{2n(k)})\in\cE_{2k-1}$, of the intervals $E_{2k-1}$ defined by
$$
E_{2k-1}=
E(a_1,\dots,a_{2n(k)}):=
\bigcup_{a\in\cF_{2k-1}(a_1,\dots,a_{2n(k)})}I(a_1,\dots,a_{2n(k)},a).
$$

\subsubsection{Level $(2k)$.} 
\label{SectionEvenLevelCantor}

Let $\cE_{2k}$ be the family of words of the form 
$
(a_1,\dots,a_{2n(k)},a_{2n(k)+1})
$, 
where 
$
(a_1,\dots,a_{2n(k)})\in\cE_{2k-1}
$ 
and where the last entry satisfies  
$
a_{2n(k)+1}\in\cF_{2k-1}(a_1,\dots,a_{2n(k)})
$. 
For any such word 
$
(a_1,\dots,a_{2n(k)+1})\in\cE_{2k}
$ 
define 
$
\cG_{2k}(a_1,\dots,a_{2n(k)+1})\subset\NN^\ast
$ 
as the set of those $a_{2n(k)+2}\geq1$ which satisfy Equation~\eqref{EquationSizeEntry2n+2}. Finally let $\EE_{2k}$ be the union, over
$
(a_1,\dots,a_{2n(k)+1})\in\cE_{2k}
$, 
of the intervals  
$$
E_{2k}=
E(a_1,\dots,a_{2n(k)+1}):=
\bigcup_{a\in\cG_{2k}(a_1,\dots,a_{2n(k)+1})}
I(a_1,\dots,a_{2n(k)+1},a).
$$

\subsection{Growth of denominators}
\label{SectionGrowthDenominators}

Fix $k\geq2$ and recall the notation in \S~\ref{SectionConstructionCantorSlopes}. Consider any
$$
(a_1,\dots,a_{2n(k-1)+1})\in\cE_{2k-2}
\quad
\textrm{ and }
\quad
(a_1,\dots,a_{2p(k)})\in\cP_{2k-1}(a_1,\dots,a_{2n(k-1)+1}).
$$ 
Since  
$
a_{2n(k-1)+2}\in\cG_{2k-2}(a_1,\dots,a_{2n(k-1)+1})
$ 
satisfies Equation \eqref{EquationSizeEntry2n+2}, we have  
\begin{eqnarray*}
&&
q_{2n(k-1)+1}^\nu
\leq
a_{2n(k-1)+2}\cdot q_{2n(k-1)+1}
\leq
q_{2n(k-1)+2}=
\\
&&
a_{2n(k-1)+2}\cdot q_{2n(k-1)+1}+ q_{2n(k-1)}
\leq
(a_{2n(k-1)+2}+1)\cdot q_{2n(k-1)+1}
\leq
2q_{2n(k-1)+1}^\nu,
\end{eqnarray*}
so that in particular
\begin{equation}
\label{EquationDenominatorsCantor(1)}
q_{2n(k-1)+1}^\nu\leq q_{2n(k-1)+2}\leq 2q_{2n(k-1)+1}^\nu.
\end{equation}

Observe that the partial quotients relating $q_{2n(k-1)+2}$ and $q_{2p(k)}$ are the entries of the block  
$
(a_{2n(k-1)+3},\dots,a_{2p(k)})
$ 
in the family 
$
\cP_{2k-1}(a_1,\dots,a_{2n(k-1)+1}|a_{2n(k-1)+2})
$, 
so that 
$$
q_{2n(k-1)+2}
\leq
q_{2p(k)}
\leq
(N+1)^{2N-2}\cdot q_{2n(k-1)+2}.
$$
Therefore, setting 
$
C_1:=2(N+1)^{2N-2}
$, 
Equation~\eqref{EquationDenominatorsCantor(1)} implies
\begin{equation}
\label{EquationDenominatorsCantor(2)}
q_{2n(k-1)+1}^\nu\leq 
q_{2p(k)}\leq 
C_1\cdot q_{2n(k-1)+1}^\nu.
\end{equation}

Moreover, consider any  
$
(a_1,\dots,a_{2n(k)+1})\in\cE_{2k}(a_1,\dots,a_{2p(k)})
$. 
The partial quotients relating the denominator $q_{2n(k)}$ and $q_{2p(k)}$ are the entries of the block 
$
(a_{2p(k)+1},\dots,a_{2n(k)})
$ 
in the family 
$
\cN_{2k-1}(a_1,\dots,a_{2p(k)})
$, 
so that, setting 
$
C_2:=(2N+|\sigma_0^{(\ast)}|+1)\cdot(N+1)^{2N-3}
$ 
we have 
\begin{equation}
\label{EquationDenominatorsCantor(3)}
q_{2p(k)}\leq q_{2n(k)}\leq C_2\cdot q_{2p(k)}.
\end{equation}

Therefore, since 
$
a_{2n(k)+1}\in\cF_{2k-1}(a_1,\dots,a_{2n(k)})
$, 
observing that Equation~\eqref{EquationSizeEntry2n+1} implies
\begin{eqnarray*}
&&
q_{2n(k)}^\eta
\leq
a_{2n(k)+1}\cdot q_{2n(k)}
\leq
q_{2n(k)+1}=
\\
&&
a_{2n(k)+1}\cdot q_{2n(k)}+q_{2n(k)-1}
\leq
(a_{2n(k)+1}+1)\cdot q_{2n(k)}
\leq
(|\sigma_0|+1)q_{2n(k)}^\eta,
\end{eqnarray*}
recalling Equation~\eqref{EquationDenominatorsCantor(3)} we get
\begin{equation}
\label{EquationDenominatorsCantor(4)}
q_{2p(k)}^\eta\leq 
q_{2n(k)+1}\leq 
(|\sigma_0|+1)\cdot C_2^\eta\cdot q_{2p(k)}^\eta.
\end{equation}

\subsection{Proof of Proposition~\ref{PropositionConstructionCantorSlopes}} 
\label{SectionProofPropositionCantor}

The inductive definition of the levels $(\EE_k)_{k\in\NN}$ of the set $\EE$ is given in \S~\ref{SectionConstructionCantorSlopes}. The diophantine type $w(\alpha)$ of $\alpha=[a_1,a_2,\dots]$, which is defined in Equation~\eqref{EquationDiophantineTypeClassical}, is also the supremum of those $w\geq 1$ such that 
$
a_n\geq (q_{n-1})^{w-1}
$ 
for infinitely many $n\in\NN$. Therefore $w(\alpha)=\eta$ for any $\alpha\in\EE$, because for any $k\in\NN$ the entries $a_{2n(k)+1}$ and $a_{2n(k)+2}$ of $\alpha$ satisfy Equation~\eqref{EquationSizeEntry2n+1} and Equation~\eqref{EquationSizeEntry2n+2} respectively, while all other $a_i$ are uniformly bounded, according to Equation~\eqref{EquationConnectionToOneCylinder} and Equation~\eqref{EquationConnectionToSplittingPair}. It only remains to prove Equation~\eqref{EquationRelationEntries2p(k)And2p(k-1)}. To do so, set 
$
C_3:=C_1\cdot C_2^{\nu\cdot\eta}\cdot(|\sigma_0|+1)^\nu
$, 
in terms of the constants $C_1$ and $C_2$ defined in \S~\ref{SectionGrowthDenominators}, and observe that 
$$
q_{2p(k)}
\leq
C_1\cdot q_{2n(k-1)+1}^\nu
\leq
C_1\cdot C_3\cdot q_{2p(k-1)}^{\nu\cdot\eta},
$$
where the first equality follows from Equation~\eqref{EquationDenominatorsCantor(2)} and the second from  Equation~\eqref{EquationDenominatorsCantor(4)}. Proposition~\ref{PropositionConstructionCantorSlopes} is proved. $\qed$

\subsection{Dimension estimate}
\label{SectionDimensionEstimate}

Consider a Cantor-like set  
$
\EE:=\bigcap_{k\in\NN}\EE_k
$, 
where $\EE_0:=(0,1)$ and for any $k\geq1$ the set $\EE_k\subset(0,1)$ is a finite union of mutually disjoint closed intervals with $\EE_{k}\subset\EE_{k-1}$. For any $k\geq1$ and any interval $E_{k-1}$ in the level $\EE_{k-1}$ the \emph{spacing} is the quantity defined by
$$
\epsilon_k(E_{k-1}):=\min\distance(E_k,E_k')
$$ 
where the minimum is taken over all pair of distinct $E_k,E'_k$ in $E_{k-1}\cap\EE_k$ and where for any pair of closed intervals $E,F\subset(0,1)$ with $E\cap F$ we set 
$
\distance(E,F):=\min\{|e-f|,e\in E,f\in F\}
$. 
Set also 
$$
m_k(E_{k-1}):=\sharp(E_{k-1}\cap\EE_k),
$$
that is the number of intervals of level $\EE_k$ which are contained inside the interval $E_{k-1}$. We define a probability measure $\mu$ with support in $\EE$ setting $\mu(E_0):=1$, then assuming that $\mu(E_i)$ is defined for all $i=0,\dots,k-1$ and all intervals $E_i$ in level $\EE_i$, for any $E_k$ in the level $\EE_k$ we define inductively  
$$
\mu(E_k):=\frac{\mu(E_{k-1})}{m(E_{k-1})},
$$
where $E_{k-1}$ is the unique interval of the level $\EE_{k-1}$ such that $E_k\subset E_{k-1}$. One can see that this defines a probability measure on Borel subsets of $(0,1)$ (see Proposition 1.7 in \cite{Falconer}). Moreover, for such a measure $\mu$ supported on $\EE$, if there exists a constant $\delta$ with $0\leq \delta\leq 1$ and constants $c>0$ and $\rho>0$ such that 
$$
\mu(U)\leq c\cdot |U|^\delta
$$
for any interval $U$ with $|U|<\rho$ and with endpoints in $\EE$, then we have $\dim_H(\EE)\geq \delta$ (see \S~4.2 in \cite{Falconer}). We recall Lemma \ref{LemmaFalconer} below, which is a little variation of Example 4.4 in \cite{Falconer}. We provide a proof for sake of completeness.

\begin{lemma}
\label{LemmaFalconer}
Assume that there exists a constant $\delta$ with $0\leq \delta\leq 1$ and a constant $c>0$ such that for any $k\geq1$ the following holds.
\begin{enumerate}
\item
For any $E_{k-1}$ in $\EE_{k-1}$ we have
$
\displaystyle{\epsilon_k(E_{k-1})\geq c\cdot\frac{|E_{k-1}|}{m_k(E_{k-1})}}
$. 
\item
For any $E_{k-1}$ in $\EE_{k-1}$ and any $E_k$ in $E_{k-1}\cap\EE_k$ we have
$
\displaystyle{|E_k|^\delta\geq\frac{|E_{k-1}|^\delta}{m_k(E_{k-1})}}
$. 
\end{enumerate}
Then we have $\dim_H(\EE)\geq \delta$.
\end{lemma}

\begin{proof}
Set 
$
C:=\sup_{E_1\in\EE_1}\mu(E_1)/|E_1|^\delta
$. 
For any $k\geq1$ and any interval $E_k$ in the level $\EE_k$ there is a nested sequence of intervals 
$
E_k\subset E_{k-1}\subset\dots\subset E_1
$, 
where $E_i$ belongs to the level $\EE_i$ for $i=1,\dots,k$, so that the definition of $\mu$ and Assumption (2) imply
$$
\frac{\mu(E_k)}{|E_k|^\delta}
=
\frac{\mu(E_{k-1})}{|E_k|^\delta\cdot m_k(E_{k-1})}
\leq
\frac{\mu(E_{k-1})}{|E_{k-1}|^\delta}
\leq
\dots
\leq
\frac{\mu(E_1)}{|E_1|^\delta}
\leq
\sup_{E_1\in\EE_1}\frac{\mu(E_1)}{|E_1|^\delta}=C.
$$
Let $U$ be an open interval with endpoints in $\EE$ and with 
$
|U|\leq\rho:=\min_{E_1\in\EE_1}|E_1|
$. 
Then there exists $k\geq1$ maximal such that $U\subset E_{k-1}$. Let $j\geq2$ be the number of intervals in $E_{k-1}\cap\EE_k$ which intersect $U$ (where $j\geq2$ by maximality of $k$). Assumption (1) implies
$$
|U|\geq 
(j-1)\cdot\epsilon_k(E_{k-1})\geq 
(j-1)\cdot c\cdot\frac{|E_{k-1}|}{m_k(E_{k-1})}\geq
|E_{k-1}|\cdot\frac{c}{2}\cdot\frac{j}{m_k(E_{k-1})}.
$$
Recalling that any interval $E_k$ in $E_{k-1}\cap\EE_k$ has measure 
$
\mu(E_k)=\mu(E_{k-1})/m_k(E_{k-1})
$, 
then the Lemma follows because the two last estimates imply
$$
\frac{\mu(U)}{|U|^\delta}\leq
\frac{1}{|U|^\delta}\cdot j\cdot\frac{\mu(E_{k-1})}{m_k(E_{k-1})}\leq
C\cdot \frac{|E_{k-1}|^\delta}{|U|^\delta}\frac{j}{m_k(E_{k-1})}\leq
C\cdot\frac{2^\delta}{c^\delta}\left(\frac{j}{m_k(E_{k-1})}\right)^{1-\delta}\leq
C\cdot\frac{2^\delta}{c^\delta}.
$$
\end{proof}

Fix $\eta\geq1$ and $s$ with $1\leq s\leq 2$, then let $\EE(X,\eta,s)$ be the set in Proposition~\ref{PropositionConstructionCantorSlopes}.

\begin{proposition}
\label{PropositinHausdorffDimensionCantor}
We have
$
\displaystyle{
\dim_H\big(\EE(X,\eta,s)\big)
\geq
\frac{\eta^{s-1}-1}{\eta^{s}-1}}
$.
\end{proposition}

\begin{proof}
We check that Point (1) and Point (2) in Lemma~\ref{LemmaFalconer} are satisfied. Referring to the notation in \S~\ref{SectionConstructionCantorSlopes}, we consider separately intervals in even level and intervals in odd levels. Fix $k\geq2$.

\smallskip

According to the definition in \S~\ref{SectionOddLevelCantor}, for any interval $E_{2k-1}$ in the level $\EE_{2k-1}$ Equation~\eqref{EquationSizeEntry2n+1} and Equation~\eqref{EquationBoundedDistortion} imply
$$
\frac{1}{3}
\left(\frac{1}{q_{2n(k)}^{\eta-1}}-\frac{1}{(|\sigma_0|+1)q_{2n(k)}^{\eta-1}}\right)
\leq
\frac{|E_{2k-1}|}{|I(a_1,\dots,a_{2n(k)})|}
\leq
2\left(\frac{1}{q_{2n(k)}^{\eta-1}}-\frac{1}{(|\sigma_0|+1)q_{2n(k)}^{\eta-1}}\right).
$$
Therefore Equation~\eqref{EquationSizeCylinderContinuedFraction} gives
$$
\frac{1}{12\cdot q_{2n(k)}^{\eta+1}}
\leq
|E_{2k-1}|
\leq
\frac{2}{q_{2n(k)}^{\eta+1}}.
$$
Moreover, according to \S~\ref{SectionEvenLevelCantor}, we have 
$$
m_{2k}(E_{2k-1})=
\sharp\cF_{2k-1}(a_1,\dots,a_{2n(k)})
=
q_{2n(k)}^{\eta-1}.
$$
On the other hand Lemma~\ref{LemmaSolvingCongruence} implies
$
|a_{2n(k)+1}-a'_{2n(k)+1}|=|\sigma_0|
$ 
for any two consecutive elements $a_{2n(k)+1},a'_{2n(k)+1}$ in 
$
\cF_{2k-1}(a_1,\dots,a_{2n(k)})
$, and for any 
$
a_{2n(k)+1}\in\cF_{2k-1}(a_1,\dots,a_{2n(k)})
$ 
we have
$
|E(a_1,\dots,a_{2n(k)},a_{2n(k)+1})|<(1/2)|I(a_1,\dots,a_{2n(k)},a_{2n(k)+1})|
$, 
therefore
\begin{eqnarray*}
&&
\epsilon_{2k}(E_{2k-1})\geq
\frac{|\sigma_0|}{2}
\min_{a_{2n(k)+1}\in\cF_{2k-1}(a_1,\dots,a_{2n(k)})}
|I(a_1,\dots,a_{2n(k)},a_{2n(k)+1})|=
\\
&&
\frac{|\sigma_0|}{4}
\min_{a_{2n(k)+1}}
\frac{1}{q_{2n(k)+1}^2}=
\frac{|\sigma_0|}{2}
\frac{1}{\big((|\sigma_0|+1)q_{2n(k)}^{\eta-1}\cdot q_{2n(k)}\big)^2}=
\frac{|\sigma_0|}{4(|\sigma_0|+1)^2}
\frac{1}{q_{2n(k)}^{2\eta}}.
\end{eqnarray*}
Therefore we have
$$
\frac{|E_{2k-1}|}{m_{2k}(E_{2k-1})}
\leq
\frac{2}{q_{2n(k)}^{\eta+1}\cdot q_{2n(k)}^{\eta-1}}
\leq
\frac{8(|\sigma_0|+1)^2}{|\sigma_0|}\cdot\epsilon_{2k}(E_{2k-1}).
$$
Now consider any interval $E_{2k}$ in the level $\EE_{2k}$, which is defined in \S~\ref{SectionEvenLevelCantor}. Recalling Equation~\eqref{EquationSizeEntry2n+2} and reasoning as above we get
$$
\frac{1}{12\cdot q_{2n(k)+1}^{\nu+1}}
\leq
|E_{2k}|
\leq
\frac{2}{q_{2n(k)+1}^{\nu+1}},
$$
where the analogous estimate holds for any interval $E_{2k-2}$ in the level $\EE_{2k-2}$. Moreover, referring to the definitions in \S~\ref{SectionOddLevelCantor} and in \S~\ref{SectionEvenLevelCantor}, we have 
$$
m_{2k-1}(E_{2k-2})
=
\sharp\cE_{2k-1}(a_1,\dots,a_{2n(k-1)+1})
\geq
\sharp\cG_{2k-2}(a_1,\dots,a_{2n(k-1)+1})=
q_{2n(k-1)+1}^{\nu-1}
$$
where the last equality follows from Equation~\eqref{EquationSizeEntry2n+2}. On the other hand, observing that
$
|E(a_1,\dots,a_{2n(k)})|\leq(1/2)|I(a_1,\dots,a_{2n(k)})|
$, 
we have
\begin{eqnarray*}
&&
\epsilon_{2k-1}(E_{2k-2})
\geq
\frac{1}{2}
\min_{(a_1,\dots,a_{2n(k)})\in\cE_{2k-1}(a_1,\dots,a_{2n(k-1)+1})}
|I(a_1,\dots,a_{2n(k)})|\geq
\\
&&
\frac{1}{4}
\min_{(\cdots)}\frac{1}{q_{2n(k)}^2}
\geq
\frac{1}{4}
\frac{1}{\big(C_2\cdot C_1\cdot q_{2n(k-1)+1}^\nu\big)^2}
=
\frac{1}{4(C_1\cdot C_2)^{2}}
\frac{1}{q_{2n(k-1)+1}^{2\nu}},
\end{eqnarray*}
where 
$
E_{2k-2}=E(a_1,\dots,a_{2n(k-1)+1})
$, 
where the third inequality follows from Equation~\eqref{EquationDenominatorsCantor(2)} and Equation~\eqref{EquationDenominatorsCantor(3)}, and where the constants $C_1$ and $C_2$ are defined in \S~\ref{SectionGrowthDenominators}. It follows that 
$$
\frac{|E_{2k-2}|}{m_{2k-1}(E_{2k-2})}
\leq
\frac{2}{q_{2n(k-1)+1}^{\nu+1}\cdot q_{2n(k-1)+1}^{\nu-1}}
\leq
8(C_1\cdot C_2)^{2}\epsilon_{2k-1}(E_{2k-2}).
$$
Point (1) in Lemma~\ref{LemmaFalconer} is established for all intervals in all levels of $\EE$. Fix $\delta\geq0$. For any interval $E_{2k-1}$ in $\EE_{2k-1}$ we have
$$
\frac{|E_{2k-1}|^\delta}{m_{2k}(E_{2k-1})}
\leq
\frac{1}{2^\delta\cdot q_{2n(k)}^{\delta(\eta+1)}}
\cdot
\frac{1}{q_{2n(k)}^{\eta-1}}
=
\frac{1}{2^\delta}\cdot
\left(\frac{1}{q_{2n(k)}}\right)^{\delta(\eta+1)+(\eta-1)}
$$
On the other hand for any interval 
$
E_{2k}\in E_{2k-1}\cap\EE_{2k}
$, 
arguing as in \S~\ref{SectionGrowthDenominators}, we have
$$
|E_{2k}|^\delta\geq
\frac{1}{12^\delta\cdot q_{2n(k)+1}^{\delta(\nu+1)}}
\geq
\frac{1}{12^\delta\cdot ((|\sigma_0|+1)\cdot q_{2n(k)}^\eta)^{\delta(\nu+1)}}
\geq
c_1\cdot\left(\frac{1}{q_{2n(k)}}\right)^{\delta\eta(\nu+1)},
$$
where $c_1>0$ is a constant depending on $|\sigma_0|$ and on $\delta$. Therefore Condition (2) in Lemma~\ref{LemmaFalconer} is satisfied for any $k$ big enough whenever 
$$
\delta\eta(\nu+1)<\delta(\eta+1)+(\eta-1)
\quad
\Leftrightarrow
\quad
\delta<\frac{\eta-1}{\eta\nu-1}.
$$
For any interval $E_{2k-2}$ in $\EE_{2k-2}$ we have
$$
\frac{|E_{2k-2}|^\delta}{m_{2k-1}(E_{2k-2})}
\leq
\frac{2^\delta}{q_{2n(k-1)+1}^{\delta(\nu+1)}}
\cdot
\frac{1}{q_{2n(k-1)+1}^{\nu-1}}
=
2^\delta\cdot
\left(\frac{1}{q_{2n(k-1)+1}}\right)^{\delta(\nu+1)+(\nu-1)}
$$
On the other hand for any interval 
$
E_{2k-1}\in E_{2k-2}\cap\EE_{2k-1}
$ 
we have
$$
|E_{2k-1}|^\delta\geq
\frac{1}{12^\delta\cdot q_{2n(k)}^{\delta(\eta+1)}}\geq
\frac{1}{12^\delta\cdot \big((C_1\cdot C_2)^{2}\cdot q_{2n(k-1)+1}^\eta\big)^{\delta(\eta+1)}}\geq
c_2\left(\frac{1}{q_{2n(k-1)+1}}\right)^{\delta\nu(\eta+1)},
$$
where the second inequality follows from Equation~\eqref{EquationDenominatorsCantor(2)} and Equation~\eqref{EquationDenominatorsCantor(3)}, and where $c_2>0$ is a constant depending on $N$ and on $\delta$. Therefore Condition (2) in Lemma~\ref{LemmaFalconer} is satisfied for any $k$ big enough whenever 
$$
\delta\nu(\eta+1)<\delta(\nu+1)+(\nu-1)
\quad
\Leftrightarrow
\quad
\delta<\frac{\nu-1}{\eta\nu-1}.
$$
There is no loss in generality to assume that Condition (2) is verified for any $k\geq1$. The Proposition follows recalling that $\nu=\eta^{s-1}$ and that $\nu-1\leq\eta-1$.
\end{proof}

\subsection{Hitting time estimates}

Recall that we fix a reduced origami and we assume that its orbit $\cO(X)$ contains both an element $X_0$ with a vertical splitting pair and an element $X_0^{(\ast)}$ whose vertical is a one cylinder direction. Fix $\eta\geq 1$ and let $\EE=\EE(X,\eta,s)$ be the set of slopes given by Proposition~\ref{PropositionConstructionCantorSlopes}. For any $\alpha\in\EE$ consider the flow $\phi_\alpha:X\to X$. consider also the integers 
$
\big(p(k)\big)_{k\in\NN}
$ 
and 
$
\big(n(k)\big)_{k\in\NN}
$ 
given by Proposition~\ref{PropositionConstructionCantorSlopes}.

\begin{lemma}
\label{LemmaLoverBoundHittingTimeOrigami}
For any $\alpha\in\EE$ and for almost any $p,p'$ in $X$ we have
$$
\whitsup(\phi_\alpha,p,p')\geq \eta^s.
$$
\end{lemma}

\begin{proof}
According to Lemma \ref{LemmaRecurrenceHittingConstant}, since $\phi_\alpha$ is uniquely ergodic, it is enough to prove the inequality for any pair $(p,p')$ in a positive measure subset of $X\times X$. Fix an integer $k\geq1$ and apply Proposition \ref{PropositionSplittingPairs} for $n=n(k)$, where we observe that the assumption in Proposition \ref{PropositionSplittingPairs} are satisfied according to Equations~\eqref{EquationSplittingPairEntry2n},~\eqref{EquationSizeEntry2n+1},~\eqref{EquationCongruenceEntry2n+1} and~\eqref{EquationSizeEntry2n+2} in Proposition~\ref{PropositionConstructionCantorSlopes}. There exists two subsets 
$
F_{n(k)},E_{n(k)}\subset X
$ 
with area $|F_{n(k)}|\geq 1/3$ and $|E_{n(k)}|\geq 1/2$ such that for any $p\in E_{n(k)}$ and $p'\in F_{n(k)}$ and for 
$
r_{n(k)}=\big(12\cdot q_{2n(k)}\big)^{-1}
$  
we have
\begin{align*}
\sqrt{8}\cdot R(X,\alpha,p,p',r_{n(k)})
&\geq 
a_{2n(k)+2}\cdot a_{2n(k)+1}\cdot q_{2n(k)}
\geq
(q_{2n(k)+1})^{\nu-1}\cdot a_{2n(k)+1}\cdot q_{2n(k)} \\
&\geq
(a_{2n(k)+1}\cdot q_{2n(k)})^{\nu-1}\cdot a_{2n(k)+1}\cdot q_{2n(k)}\\
&\geq (q_{2n(k)})^{\eta(\nu-1)}\cdot (q_{2n(k)})^\eta \\
&=(q_{2n(k)})^{\eta\cdot\nu}=
\frac{1}{12^{\eta\cdot\nu}}
\cdot
\left(\frac{1}{r_{n(k)}}\right)^{\eta\cdot\nu},
\end{align*}
where the second inequality corresponds to Equation~\eqref{EquationSizeEntry2n+2} and the fourth holds because 
$
a_{2n(k)+1}\cdot q_{2n(k)}>q_{2n(k)}^\eta
$, 
according to Equation~\eqref{EquationSizeEntry2n+1}. Therefore 
$
\whitsup(\phi_\alpha,p,p')\geq \eta\cdot\nu
$ 
for any pair
$$
(p,p')\in
\bigcap_{K=1}^\infty\bigcup_{k\geq K}
E_{n(k)}\times F_{n(k)}.
$$
The Lemma follows observing that 
$$
\leb\times\leb
\left(
\bigcap_{K=1}^\infty\bigcup_{k\geq K}
E_{n(k)}\times F_{n(k)}
\right)
\geq
\inf_{k\geq1}\leb(E_{n(k)})\cdot\leb(F_{n(k)})\geq \frac{1}{6}.
$$
\end{proof}

\begin{lemma}
\label{LemmaUpperBoundHittingTimeOrigami}
For any $\alpha\in\EE$ and for any $p,p'$ with $p$ not on any $(X,\alpha)$-singular leaf we have
$$
\whitsup(\phi_\alpha,p,p')\leq \eta^s.
$$
\end{lemma}

\begin{proof}
Fix $k\in\NN$ and apply Proposition~\ref{PropositionOneCylinderDirection} for $n=p(k)$, where we observe that the assumption in Proposition~\ref{PropositionOneCylinderDirection} are satified according to Equation~\eqref{EquationOneCylinderEntry2p} and Equation~\eqref{EquationOneCylinderEntry2p+1} in Proposition~\ref{PropositionConstructionCantorSlopes}. It follows that for any pair of points $p,p'$ in $X$ as above and any $k\in\NN$ we have 
$$
R(X,\alpha,p,p',r_k)\leq 16\cdot|\sigma_0^{(\ast)}|\cdot q_{2p(k)}
\quad
\textrm{ where }
\quad
r_k:=\frac{2}{q_{2p(k)}}.
$$
Finally, fix any $r>0$ and let $k$ be the unique integer with $r_k\leq r< r_{k-1}$. The Lemma follows observing that for any $p,p'$ as above we have 
\begin{align*}
\frac{\log R(X,\alpha,p,p',r)}{-\log r}
&
\leq 
\frac{\log R(X,\alpha,p,p',r_k)}{-\log r_{k-1}}
\leq 
\frac
{\log (16\cdot|\sigma_0^{(\ast)}|) + \log q_{2p(k)}}
{\log q_{2p(k-1)}-\log2}
\leq 
\\
&
\leq
\frac
{\log (16\cdot|\sigma_0^{(\ast)}|) +\log C + \log q_{2p(k-1)}^{\eta\cdot\nu}}
{\log q_{2p(k-1)}-\log2}
\to \eta\cdot\nu
\quad
\textrm{ for }
k\to\infty,
\end{align*}
where the last inequality follows from Equation~\eqref{EquationRelationEntries2p(k)And2p(k-1)}. The Lemma is proved.
\end{proof}

\subsection{Reduced origamis are enough}
\label{SectionReducedOrigamiAreEnough}

Let $X$ be any origami and let $a\geq1$ and $b\geq1$ be integers such that 
$
\langle\hol(X)\rangle= a\cdot\ZZ\oplus b\cdot\ZZ
$, 
so that $X$ is reduced if and only if $a=b=1$. If $X$ is not reduced, let 
$
t:=(\log b-\log a)/2
$ 
and $\mu:=1/\sqrt{ab}$, then consider the combined action of $g_t$ and the homothety $a_\mu:=\mu\cdot\id_{\RR^2}$. The origami 
$
X':=G\cdot X
$ 
with $G:=a_\mu\cdot g_t$ is reduced, indeed we have
$$
\langle\hol(G\cdot X)\rangle=
a_\mu\cdot g_t\cdot \langle\hol(X)\rangle=
a_\mu\cdot \big(\sqrt{ab}\cdot (\ZZ\oplus\ZZ)\big)=
\ZZ\oplus\ZZ.
$$
The homographic action of $G$ on slopes is 
$
\alpha\mapsto G\cdot\alpha=(b/a)\alpha
$, 
which is a bijection from $\QQ$ to itself. Moreover we have 
$
w(G\cdot\alpha)=w(\alpha)
$ 
for any $\alpha$, indeed for any $w\geq 1$ the change of variable
$
p/q=G\cdot (p'/q')=(bp')/(aq')
$ 
gives the equivalence 
$$
\left|G\cdot\alpha-\frac{p}{q}\right|<
\frac{1}{q^{1+w}}
\Leftrightarrow
\left|\alpha-\frac{p'}{q'}\right|<
(b^w\cdot a)\cdot\frac{1}{(q')^{1+w}}.
$$
On the other hand consider the surface $X'=G\cdot X$ and let $f_G:X\to X'$ be the affine diffeomorphism with derivative $Df=G$, then for any $\alpha$ let 
$
\alpha':=G\cdot\alpha
$ 
and consider the flow $\phi'_{\alpha'}$ with slope $\theta'$ on the surface $X'$. For any pair of points $p_1,p_2$ in $X$ we have 
$$
\whitsup(\phi_\alpha,p_1,p_2)=
\whitsup\big(\phi'_{\alpha'},f_G(p_1),f_G(p_2)\big).
$$
Finally, it is obvious from the definition that $X$ admits a splitting direction $\theta_{split}$ or respectively a one cylinder direction $\theta_{one-cyl}$ if and only it $X'$ does. Resuming, Theorem~\ref{TheoremHittingSpectrumOrigami} holds for an origami $X$ if and only if it holds for the reduced origami $X':=G\cdot X$ as above.

\subsection{Proof of Theorem~\ref{TheoremHittingSpectrumOrigami}}
\label{SectionProofTheoremHittingSpectrumOrigami}

According to \S~\ref{SectionReducedOrigamiAreEnough}, there is no loss of generality assuming that $X$ is reduced. Since the map $\alpha\mapsto\theta=\arctan\alpha$ preserves Hausdorff dimension, the proof can be done in the slope variable $\alpha$. Assume that $\cO(X)$ contains both an element $X_0$ with a vertical splitting pair and an element $X_0^{(\ast)}$ whose vertical is a one cylinder direction. For $s=1$ the Theorem follows from Part (1) of Proposition~\ref{PropositionHittingSpectrumOrigami}, which will be proved later. For $1<s\leq 2$, consider the set $\EE(X,\eta,s)$ given by Proposition~\ref{PropositionConstructionCantorSlopes}, so that 
$
\dim_H\big(\EE(X,\eta,s)\big)\geq f_\eta(s)
$, 
according to Proposition~\ref{PropositinHausdorffDimensionCantor}. The Theorem follows combining Lemma~\ref{LemmaLoverBoundHittingTimeOrigami} and Lemma~\ref{LemmaUpperBoundHittingTimeOrigami}. Modulo the proof of Proposition~\ref{PropositionHittingSpectrumOrigami} (for the case $s=1$), Theorem~\ref{TheoremHittingSpectrumOrigami} is proved. $\qed$

\subsection{Proof of Proposition~\ref{PropositionHittingSpectrumOrigami}}
\label{SectionProofPropositionHittingSpectrumOrigami}

As in \S~\ref{SectionProofTheoremHittingSpectrumOrigami}, assume that $X$ is reduced and consider the slope variable $\alpha$. Consider separately the two cases.

\smallskip

\emph{Case $s=2$.} Assume only that $\cO(X)$ contains $X_0$ with a vertical splitting pair. Let $X_0^\ast$ be any element in $\cO(X)$, whose vertical is not necessarily a one cylinder direction. The set $\EE(X,\eta,s=2)$ can be defined as in Proposition~\ref{PropositionConstructionCantorSlopes} and the dimension estimate in Proposition~\ref{PropositinHausdorffDimensionCantor} still holds. Then Part (2) of Proposition~\ref{PropositionHittingSpectrumOrigami} follows from Lemma~\ref{LemmaLoverBoundHittingTimeOrigami}.

\smallskip

\emph{Case $s=1$.} Assume only that $\cO(X)$ contains $X_0^\ast$ whose vertical is a one cylinder direction. The set $\EE(X,\eta,s=1)$ can be defined as in Proposition~\ref{PropositionConstructionCantorSlopes}, replacing $X_0$ be any element in $\cO(X)$, not necessarily admitting a one cylinder direction, but for $s=1$ the dimension estimate in Proposition~\ref{PropositinHausdorffDimensionCantor} gives the trivial bound 
$
\dim_H\big(\EE(X,\eta,s=1)\big)\geq0
$. 
Thus in this case we replace the set $\EE(X,\eta,s=1)$ in Proposition~\ref{PropositionConstructionCantorSlopes} by the one provided by Lemma~\ref{LemmaConstructionCantorSlopesBis} below.

\begin{lemma}
\label{LemmaConstructionCantorSlopesBis}
Fix $\eta\geq1$. There exists a Cantor-like set $\EE'=\EE'(X,\eta,s=1)$ with dimension 
$
\dim_H(\EE')\geq f_\eta(1)=(1+\eta)^{-1}
$ 
such that for any $\alpha\in\EE'$ we have $w(\alpha)=\eta$ and moreover there exist integers 
$
\big(p(k)\big)_{k\in\NN}
$ 
with $p(0)=0$ such that for any $k\geq1$ we have $p(k-1)<p(k)$, Equation~\eqref{EquationOneCylinderEntry2p} and Equation~\eqref{EquationOneCylinderEntry2p+1} are satisfied, and we have also
\begin{equation}
\label{EquationRatioConsecutiveEntries2p+1Bis}
q_{2p(k)}\leq C_4\cdot q_{2p(k-1)}^\eta
\quad
\textrm{ where }
\quad
C_4:=4(N+1)^{2N-2}.
\end{equation}
\end{lemma}

Fix $\alpha\in\EE'$ and consider $\phi_\alpha:X\to X$. Replying the argument in Lemma~\ref{LemmaUpperBoundHittingTimeOrigami}, where the Equation~\eqref{EquationRelationEntries2p(k)And2p(k-1)} used in the proof of Lemma~\ref{LemmaUpperBoundHittingTimeOrigami} is replaced by Equation~\eqref{EquationRatioConsecutiveEntries2p+1Bis}, we get 
$$
\whitsup(\phi_\alpha,p,p')\leq \eta 
$$
for any $p,p'$ with $p$ not on any $(X,\alpha)$-singular leaf. For almost any $p,p'$ the last inequality turns into an equality according to the general lower bound established by Equation~\eqref{EquationLimSupHittingVeech}, recalling that 
$
w(X,\arctan\alpha)=w(\alpha)
$ 
if $X$ is an origami. Modulo the proof of Lemma~\ref{LemmaConstructionCantorSlopesBis}, Proposition~\ref{PropositionHittingSpectrumOrigami} is proved. $\qed$

\subsection{Proof of Lemma~\ref{LemmaConstructionCantorSlopesBis}}

The proof is a simplified version of the proof of Proposition~\ref{PropositionConstructionCantorSlopes}. In particular for $k\geq1$ we define families $\cE'_k$ whose elements are words which label the intervals $E'_k$ in the level $\EE'_k$, so that the required Cantor-like set is 
$
\EE'=\bigcap_{k\in\NN}\EE'_k
$. 
For a word $(a_1,\dots,a_{2p})$ define $\cF'(a_1,\dots,a_{2p})$ as the set of integers $a_{2p+1}$ which satisfy 
\begin{equation}
\label{EquationOneCylinderEntry2p+1Bis}
q_{2p}^{\eta-1}\leq a_{2p+1}\leq 2\cdot q_{2p}^{\eta-1}-1,
\end{equation}
then consider the interval
\begin{equation}
\label{EquationInterval2pBis}
E'(a_1,\dots,a_{2p}):=
\bigcup_{a_{2p+1}\in\cF'_1(a_1,\dots,a_{2p})}I(a_1,\dots,a_{2p},a_{2p+1}).
\end{equation}
In the construction below, just in order to refer to the notation in \S~\ref{SectionConstructionCantorSlopes}, we introduce instants $n(k)$ with $n(k)=p(k)$ for any $k\in\NN$.

Reasoning as in \S~\ref{SectionFirstLevelCantor}, let $\cE'_1$ be the maximal disjoint family of finite words 
$
(a_1,\dots,a_{2p(1)})
$ 
which satisfy Equation~\eqref{EquationOneCylinderEntry2p}, where the first two entries are defined by $a_1:=1$ and $a_2:=1$ and where the block 
$
(a_3,\dots,a_{2p(1)})
$ 
satisfies Equation~\eqref{EquationConnectionToOneCylinder} with $n(0):=0$. Fix $k\geq1$ and assume that the family $\cE'_{k-1}$ is defined. For any 
$
(a_1,\dots,a_{2p(k-1)})\in\cE'_{k-1}
$ 
let  
$
\cF'(a_1,\dots,a_{2p(k-1)})
$ 
be the set defined by Equation~\eqref{EquationOneCylinderEntry2p+1Bis}, then let 
$
\cE'_{k}(a_1,\dots,a_{2p(k-1)})
$ 
be the family of concatenated words 
\begin{align*}
&
(a_1,\dots,a_{2p(k)})=
(a_1,\dots,a_{2p(k-1)+1},a_{2p(k-1)+2},a_{2p(k-1)+3},\dots,a_{2p(k)}),
\\
&
\textrm{where }
\quad 
a_{2p(k-1)+1}\in\cF'_{k-1}(a_1,\dots,a_{2p(k-1)})
\quad
\textrm{ and }
\quad
a_{2p(k-1)+2}:=1,
\end{align*}
and where the block 
$
(a_{2p(k-1)+3},\dots,a_{2p(k)})
$ 
satisfies Equation~\eqref{EquationConnectionToOneCylinder} with $n(k-1):=p(k-1)$, and the condition
$$
g(a_{2p(k-1)+3},\dots,a_{2p(k)})\cdot X_0^{(\ast)}=
g(a_{2p(k-1)+1},a_{2p(k-1)+2})^{-1}\cdot X_0^{(\ast)},
$$
which is possible applying Lemma~\ref{LemmaConnectionCusps} with 
$
Y_2:=g(a_{2p(k-1)+1},a_{2p(k-1)+2})^{-1}\cdot X_0^{(\ast)}
$ 
and $Y_1:=X_0^{(\ast)}$. Arguing as in \S~\ref{SectionOddLevelCantor} one can see that any 
$
(a_1,\dots,a_{2p(k)})\in\cE'_k(a_1,\dots,a_{2p(k-1)})
$ 
satisfies Equation~\eqref{EquationOneCylinderEntry2p}. Moreover, for any $k\geq1$ big enough, Equation~\eqref{EquationOneCylinderEntry2p+1} is also satisfied by the entry $a_{2p(k)+1}$, according to Equation~\eqref{EquationOneCylinderEntry2p+1Bis}. Define the family $\cE'_k$ as the union, over 
$
(a_1,\dots,a_{2p(k-1)})\in\cE'_{k-1}
$ 
of the families  
$
\cE'_{k}(a_1,\dots,a_{2p(k-1)})
$, 
then define the level $\EE'_k$ as the union, over 
$
(a_1,\dots,a_{2p(k)})\in\cE'_{k}
$ 
of the intervals 
$
E'_{k}=E(a_1,\dots,a_{2p(k)})
$ 
given by Equation~\eqref{EquationInterval2pBis}. The definition of $\EE'$ is complete. It is clear that $w(\alpha)=\eta$ for any $\alpha\in\EE'$. Moreover Equation~\eqref{EquationRatioConsecutiveEntries2p+1Bis} holds, indeed arguing as in \S~\ref{SectionGrowthDenominators} we have
$$
q_{2p(k)}\leq 
(N+1)^{2N-2}\cdot q_{2p(k-1)+2}\leq
2(N+1)^{2N-2}\cdot q_{2p(k-1)+1}\leq
4(N+1)^{2N-2}\cdot q_{2p(k-1)},
$$
where the first inequality follows from Equation~\eqref{EquationConnectionToOneCylinder}, the second holds because $a_{2p(k-1)+2}=1$ and the third follows from Equation~\eqref{EquationOneCylinderEntry2p+1Bis}. Finally, in order to prove the dimension estimate on $\EE'$, recall the notation of \S~\ref{SectionDimensionEstimate} and observe that reasoning as in the proof of Proposition~\ref{PropositinHausdorffDimensionCantor}, we have 
\begin{align*}
&
\frac{1}{12\cdot q_{2p(k)}^{\eta+1}}
\leq
|E'_k|
\leq
\frac{2}{q_{2p(k)}^{\eta+1}},
\\
&
m_k(E'_{k-1})=
\sharp\cF'_k(a_1,\dots,a_{2p(k-1)})=q_{2p(k-1)}^{\eta-1},
\\
&
\epsilon_k(E'_{k-1})\geq
\frac{1}{2}
\min_{a_{2p(k-1)+1}\in\cF(a_1,\dots,a_{2p(k-1)})}
|I(a_1,\dots,a_{2p(k-1)},a_{2p(k-1)+1})|
\\
&
\geq
\min_{a_{2p(k-1)+1}}
\frac{1}{2\cdot q_{2p(k-1)+1}^2}
\geq
\frac{1}{2\cdot (2\cdot q_{2p(k-1)}^{\eta-1}\cdot q_{2p(k-1)})^2}
=
\frac{1}{8\cdot q_{2p(k-1)}^{2\eta}}.
\end{align*}
Therefore Condition (1) in Lemma~\ref{LemmaFalconer} holds, indeed we have
$$
\frac{|E'_{k-1}|}{m_k(E'_{k-1})}
\leq
\frac{2}{q_{2p(k-1)}^{\eta+1}\cdot q_{2p(k-1)}^{\eta-1}}
\leq 
16\cdot \epsilon_k(E'_{k-1}).
$$
Moreover Condition (2) in Lemma~\ref{LemmaFalconer} is also satisfied for any $\delta<(1+\eta)^{-1}$, which is equivalent to 
$
\delta\cdot\eta\cdot(\eta+1)<\delta\cdot(\eta+1)+(\eta-1)
$, 
indeed for any $k\geq1$ big enough we have
$$
|E'_k|^\delta
\geq
\frac{1}{12^\delta\cdot q_{2p(k)}^{\delta\cdot (\eta+1)}}
\geq
\frac{1}{(12\cdot C_4)^\delta\cdot q_{2p(k-1)}^{\delta\cdot\eta\cdot(\eta+1)}}
\geq
\frac{2^\delta}{q_{2p(k-1)}^{\delta\cdot(\eta+1)}\cdot q_{2p(k-1)}^{\eta-1}}
\geq
\frac{|E'_{k-1}|^\delta}{m_k(E'_{k-1})}.
$$
Lemma~\ref{LemmaConstructionCantorSlopesBis} is proved. $\qed$

\section{Proof of Proposition~\ref{PropositionEierlegendeWollmilchsau}}
\label{SectionEierlegendeWollmilchsau}

This section follows \S~\ref{SectionOneCylinderVerticalDirections} and represents an adaptation of the arguments therein to the \emph{Eierlegende Wollmilchsau} origami $X_{EW}$.

\subsection{The Eierlegende Wollmilchsau origami}

Following \S~8.4 in \cite{ForniMatheus}, recall that the \emph{quaternion group} is the group of eight elements 
$
\cQ:=\{\pm1,\pm i,\pm j,\pm k\}
$ 
with relations 
$$
i^2=j^2=k^2=-1
\quad
\textrm{ and }
\quad
i\cdot j=k.
$$
It is easy to see that the other multiplication rules are $i\cdot k=-j$, $j\cdot i=-k$, $j\cdot k=i$, $k\cdot i=j$ and $k\cdot j=-i$. Moreover, recall from \S~\ref{SectionIntroductionOrigamis} that we can describe an origami considering a finite family of labelled squares, each square being a copy of $[0,1]^2$, and then defining identifications between their sides. The \emph{Eierlegende Wollmilchsau} origami $X_{EW}$ is the origami obtained considering the quaternion group $\cQ$ as set of labels, with identifications given by the right multiplication by the two generators $i$ and $j$. More precisely, for any $g\in\cQ$ consider the square 
$
Q_g:=\{g\}\times [0,1]^2
$, 
whose sides are $l_g:=\{g\}\times l$, $r_g:=\{g\}\times r$, $b_g:=\{g\}\times b$ and $t_g:=\{g\}\times t$, where $l$, $r$, $b$ and $t$ are the sides of the standard square $[0,1]^2$ defined by 
$$
l:=\{0\}\times[0,1]
\textrm{ , }
r:=\{1\}\times[0,1]
\textrm{ , }
b:=[0,1]\times\{0\}
\textrm{ , }
t:=[0,1]\times\{1\}.
$$ 
The surface $X_{EW}$ is obtained identifying, for any $g\in\cQ$, the right side $r_g$ of the square $Q_g$ with the left side $l_{g\cdot i}$ of the square $Q_{g\cdot i}$ and the top side $t_g$ of $Q_g$ with the bottom side $b_{g\cdot j}$ of $Q_{g\cdot j}$. Turning around the vertices of the squares and following the identifications, it is easy to check that $X_{EW}$ has $4$ conical singularities, corresponding to the orbits of the right multiplication on $\cQ$ by the commutator $[j,-i]=-1$. Each singularity has conical angle $4\pi$, thus $X_{EW}$ belongs to the stratum $\cH(1,1,1,1)$ and has genus $g=3$. Very specific dynamical and geometric properties of the surface $X_{EW}$ are explained in \S~7 and \S~8 in \cite{ForniMatheus}. Two different representations of $X_{EW}$ are given in Figure~\ref{FigureEierlegendeWollmilchsau}. In particular, by direct computation of the action on $X_{EW}$ of the generators $T$ and $V$ of $\sltwoz$ (see also Remark~87 in~\cite{ForniMatheus}), one can see that the stabilizer of $X_{EW}$ is the entire group $\sltwoz$, that is 
\begin{equation}
\label{EquationOrbitEierlegendeWollmilchsau}
\cO(X_{EW})=\{X_{EW}\}.
\end{equation}

\begin{figure}[ht]
\begin{minipage}[c]{0.20\textwidth}
\begin{center}  

{\begin{tikzpicture}[scale=0.1]


\tikzset
{->-/.style={decoration={markings,mark=at position .5 with {\arrow{>}}},postaction={decorate}}}



\draw[-,thick] (0,70) -- (10,70) {};

\node [circle,fill,inner sep=1pt] at (0,70) {};
\node [circle,fill,inner sep=1pt] at (10,70) {};

\draw[-,thick] (0,60) -- (0,70) {};
\draw[-,thick] (10,60) -- (10,70) {};

\node at (5,65) {$-j$};


\draw[-,thin,dashed] (0,60) -- (10,60);

\node [circle,fill,inner sep=1pt] at (0,60) {};
\node [circle,fill,inner sep=1pt] at (10,60) {};

\draw[-,thick] (0,50) -- (0,60) {};
\draw[-,thick] (10,50) -- (10,60) {};

\node at (5,55) {$-1$};


\draw[-,thin,dashed] (0,50) -- (10,50);

\node [circle,fill,inner sep=1pt] at (0,50) {};
\node [circle,fill,inner sep=1pt] at (10,50) {};

\draw[-,thick] (0,40) -- (0,50) {};
\draw[-,thick] (10,40) -- (10,50) {};

\node at (5,45) {$j$};


\draw[-,thin,dashed] (0,40) -- (10,40);

\node [circle,fill,inner sep=1pt] at (0,40) {};

\draw[-,thick] (0,30) -- (0,40) {};
\draw[-,thin,dashed] (10,30) -- (10,40) {};

\node at (5,35) {$1$};

\draw[-,thick] (0,30) -- (10,30);

\node [circle,fill,inner sep=1pt] at (0,30) {};
\node [circle,fill,inner sep=1pt] at (10,30) {};

\node at (2,20) {$C^L_0$};
\draw[->,thin,dashed] (2,23) -- (2,29);



\node at (18,50) {$C^R_0$};
\draw[->,thin,dashed] (18,47) -- (18,41);

\draw[-,thick] (10,40) -- (20,40) {};

\node [circle,fill,inner sep=1pt] at (10,40) {};
\node [circle,fill,inner sep=1pt] at (20,40) {};

\draw[-,thin,dashed] (10,30) -- (10,40) {};
\draw[-,thick] (20,30) -- (20,40) {};

\node at (15,35) {$i$};


\draw[-,thin,dashed] (10,30) -- (20,30);

\node [circle,fill,inner sep=1pt] at (20,30) {};

\draw[-,thick] (10,20) -- (10,30) {};
\draw[-,thick] (20,20) -- (20,30) {};

\node at (15,25) {$-k$};


\draw[-,thin,dashed] (10,20) -- (20,20);

\node [circle,fill,inner sep=1pt] at (10,20) {};
\node [circle,fill,inner sep=1pt] at (20,20) {};

\draw[-,thick] (10,10) -- (10,20) {};
\draw[-,thick] (20,10) -- (20,20) {};

\node at (15,15) {$-i$};


\draw[-,thin,dashed] (10,10) -- (20,10);

\node [circle,fill,inner sep=1pt] at (10,10) {};
\node [circle,fill,inner sep=1pt] at (20,10) {};

\draw[-,thick] (10,0) -- (10,10) {};
\draw[-,thick] (20,0) -- (20,10) {};

\node at (15,5) {$k$};

\draw[-,thick] (10,0) -- (20,00);

\node [circle,fill,inner sep=1pt] at (10,0) {};
\node [circle,fill,inner sep=1pt] at (20,0) {};


\draw[->-,thick,red] (4,62) -- (6,70) {};
\draw[->-,thick,red] (6,30) -- (10,46) {};
\draw[->-,thick,red] (10,26) -- (12,32) {};

\end{tikzpicture}}

\end{center}
\end{minipage}\hfill
\begin{minipage}[c]{0.80\textwidth}

\begin{center}

{\begin{tikzpicture}[scale=0.1]


\tikzset
{->-/.style={decoration={markings,mark=at position .7 with {\arrow{>}}},postaction={decorate}}}



\draw[-,thick] (0,0) -- (0,10) {};

\node [circle,fill,inner sep=1pt] at (0,0) {};
\node [circle,fill,inner sep=1pt] at (0,10) {};

\draw[-,thick] (0,0) -- (10,0) {};
\draw[-,thick] (0,10) -- (10,10) {};

\node at (5,5) {$1$};


\draw[-,thin,dashed] (10,0) -- (10,10);

\node [circle,fill,inner sep=1pt] at (10,0) {};
\node [circle,fill,inner sep=1pt] at (10,10) {};

\draw[-,thick] (10,0) -- (20,0) {};
\draw[-,thick] (10,10) -- (20,10) {};

\node at (15,5) {$i$};


\draw[-,thin,dashed] (20,0) -- (20,10);

\node [circle,fill,inner sep=1pt] at (20,0) {};
\node [circle,fill,inner sep=1pt] at (20,10) {};

\draw[-,thick] (20,0) -- (30,0) {};
\draw[-,thick] (20,10) -- (30,10) {};

\node at (25,5) {$-1$};


\draw[-,thin,dashed] (30,0) -- (30,10);

\node [circle,fill,inner sep=1pt] at (30,0) {};

\draw[-,thick] (30,0) -- (40,0) {};
\draw[-,thin,dashed] (30,10) -- (40,10) {};

\node at (35,5) {$-i$};

\draw[-,thick] (40,0) -- (40,10);

\node [circle,fill,inner sep=1pt] at (40,0) {};

\node at (50,2) {$C^R_\infty$};
\draw[->,thin,dashed] (46,2) -- (41,2);


\node at (20,18) {$C^L_\infty$};
\draw[->,thin,dashed] (23,18) -- (29,18);


\draw[-,thick] (30,10) -- (30,20) {};

\node [circle,fill,inner sep=1pt] at (30,10) {};
\node [circle,fill,inner sep=1pt] at (30,20) {};

\draw[-,thick] (30,20) -- (40,20) {};

\node at (35,15) {$-k$};


\draw[-,thin,dashed] (40,10) -- (40,20);

\node [circle,fill,inner sep=1pt] at (40,20) {};

\draw[-,thick] (40,10) -- (50,10) {};
\draw[-,thick] (40,20) -- (50,20) {};

\node at (45,15) {$-j$};


\draw[-,thin,dashed] (50,10) -- (50,20);

\node [circle,fill,inner sep=1pt] at (50,10) {};
\node [circle,fill,inner sep=1pt] at (50,20) {};

\draw[-,thick] (40,10) -- (50,10) {};
\draw[-,thick] (50,20) -- (50,20) {};

\node at (55,15) {$k$};


\draw[-,thin,dashed] (50,10) -- (50,20);

\node [circle,fill,inner sep=1pt] at (60,10) {};
\node [circle,fill,inner sep=1pt] at (60,20) {};

\draw[-,thick] (50,10) -- (60,10) {};
\draw[-,thick] (50,20) -- (60,20) {};

\node at (65,15) {$j$};

\draw[-,thin,dashed] (60,10) -- (60,20);

\node [circle,fill,inner sep=1pt] at (70,10) {};
\node [circle,fill,inner sep=1pt] at (70,20) {};

\draw[-,thick] (60,10) -- (70,10) {};
\draw[-,thick] (60,20) -- (70,20) {};

\draw[-,thick] (70,10) -- (70,20);


\draw[->-,thick,red] (44,12) -- (46,20) {};
\draw[->-,thick,red] (6,0) -- (8.5,10) {};
\draw[-,thick,red] (68.5,10) -- (70,16) {};
\draw[-,thick,red] (30,16) -- (31,20) {};
\draw[->-,thick,red] (11,0) -- (12,4) {};

\end{tikzpicture}}

\vspace{30px}

{\begin{tikzpicture}[scale=0.1]


\tikzset
{->-/.style={decoration={markings,mark=at position .5 with {\arrow{>}}},postaction={decorate}}}


\draw[-,thick] (0,0) -- (0,10) {};

\node [circle,fill,inner sep=1pt] at (0,0) {};
\node [circle,fill,inner sep=1pt] at (0,10) {};

\draw[-,thick] (0,0) -- (10,0) {};
\draw[-,thick] (0,10) -- (10,10) {};

\node at (5,5) {$1$};

\draw[-,thin,dashed] (10,0) -- (10,10) {};


\node [circle,fill,inner sep=1pt] at (10,0) {};
\node [circle,fill,inner sep=1pt] at (10,10) {};

\draw[-,thick] (10,0) -- (20,0) {};
\draw[-,thick] (10,10) -- (20,10) {};

\node at (15,5) {$i$};

\draw[-,thin,dashed] (20,0) -- (20,10) {};

\node [circle,fill,inner sep=1pt] at (20,0) {};
\node [circle,fill,inner sep=1pt] at (20,10) {};


\draw[-,thin,dashed] (20,0) -- (25,0) {};
\draw[-,thin,dashed] (20,10) -- (25,10) {};
\draw[-,thin,dashed] (20,0) -- (20,-5) {};

\node at (25,5) {$-1$};
\node at (25,-5) {$j$};


\draw[->-,thick,green] (0,9) -- (10,3) {};
\draw[->-,thick,green] (10,3) -- (20,-3) {};

\node at (-5,7) {$S$};
\draw[->,thin,dashed] (-3,7) -- (-1,7) {};

\draw[->-,thick,red] (5,0) -- (13,10) {};

\node at (5,-5) {$I$};
\draw[->,thin,dashed] (5,-3) -- (5,-1) {};

\end{tikzpicture}}
\hspace{20px}
{\begin{tikzpicture}[scale=0.1]


\tikzset
{->-/.style={decoration={markings,mark=at position .5 with {\arrow{>}}},postaction={decorate}}}


\node [circle,fill,inner sep=1pt] at (0,0) {};
\node [circle,fill,inner sep=1pt] at (0,10) {};
\node [circle,fill,inner sep=1pt] at (10,0) {};
\node [circle,fill,inner sep=1pt] at (10,10) {};

\draw[-,thick] (0,0) -- (0,10) {};
\draw[-,thin,dashed] (0,0) -- (10,0) {};
\draw[-,thick] (0,10) -- (10,10) {};
\draw[-,thick] (10,0) -- (10,10) {};

\node at (5,5) {$1$};


\node [circle,fill,inner sep=1pt] at (0,-10) {};
\node [circle,fill,inner sep=1pt] at (10,-10) {};

\draw[-,thick] (0,-10) -- (10,-10) {};
\draw[-,thick] (0,0) -- (0,-10) {};
\draw[-,thin,dashed] (10,0) -- (10,-10) {};

\node at (5,-5) {$-j$};

\draw[-,thin,dashed] (10,-10) -- (10,0) {};


\node [circle,fill,inner sep=1pt] at (20,0) {};
\node [circle,fill,inner sep=1pt] at (20,-10) {};

\draw[-,thick] (10,0) -- (20,-0) {};
\draw[-,thick] (10,-10) -- (20,-10) {};
\draw[-,thin,dashed] (20,-10) -- (20,0) {};

\node at (15,-5) {$k$};


\draw[-,thin,dashed] (20,0) -- (25,0) {};
\draw[-,thin,dashed] (20,-10) -- (25,-10) {};

\node at (25,-5) {$j$};


\draw[->-,thick,green] (0,2) -- (10,-3) {};
\draw[->-,thick,green] (10,-3) -- (20,-8) {};

\node at (-5,2) {$S$};
\draw[->,thin,dashed] (-3,2) -- (-1,2) {};

\draw[-,thick,red,dashed] (3,-10) -- (4.5,-5) {};
\draw[-,thick,red] (4.5,-5) -- (6,0) {};
\draw[->-,thick,red] (6,0) -- (9,10) {};

\node at (3,-15) {$I$};
\draw[->,thin,dashed] (3,-13) -- (3,-11) {};

\end{tikzpicture}}
\hspace{20px}
{\begin{tikzpicture}[scale=0.1]


\tikzset
{->-/.style={decoration={markings,mark=at position .5 with {\arrow{>}}},postaction={decorate}}}


\node [circle,fill,inner sep=1pt] at (0,0) {};
\node [circle,fill,inner sep=1pt] at (0,10) {};
\node [circle,fill,inner sep=1pt] at (10,0) {};
\node [circle,fill,inner sep=1pt] at (10,10) {};

\draw[-,thick] (0,0) -- (0,10) {};
\draw[-,thin,dashed] (0,0) -- (10,0) {};
\draw[-,thick] (0,10) -- (10,10) {};
\draw[-,thick] (10,0) -- (10,10) {};

\node at (5,5) {$1$};


\node [circle,fill,inner sep=1pt] at (0,-10) {};
\node [circle,fill,inner sep=1pt] at (10,-10) {};

\draw[-,thick] (0,-10) -- (10,-10) {};
\draw[-,thick] (0,0) -- (0,-10) {};
\draw[-,thin,dashed] (10,0) -- (10,-10) {};

\node at (5,-5) {$-j$};

\draw[-,thin,dashed] (10,-10) -- (10,0) {};


\node [circle,fill,inner sep=1pt] at (20,0) {};
\node [circle,fill,inner sep=1pt] at (20,-10) {};

\draw[-,thick] (10,0) -- (20,-0) {};
\draw[-,thick] (10,-10) -- (20,-10) {};
\draw[-,thin,dashed] (20,-10) -- (20,0) {};

\node at (15,-5) {$k$};


\draw[-,thin,dashed] (20,0) -- (25,0) {};
\draw[-,thin,dashed] (20,-10) -- (25,-10) {};
\draw[-,thin,dashed] (20,-10) -- (20,-15) {};

\node at (25,-5) {$j$};
\node at (25,-15) {$-1$};


\draw[->-,thick,green] (0,2) -- (10,-6) {};
\draw[->-,thick,green] (10,-6) -- (20,-14) {};

\node at (-5,2) {$S$};
\draw[->,thin,dashed] (-3,2) -- (-1,2) {};

\draw[->-,thick,red] (5,-10) -- (14,0) {};

\node at (5,-15) {$I$};
\draw[->,thin,dashed] (5,-13) -- (5,-11) {};
\end{tikzpicture}}
\end{center}
\end{minipage}
\caption{The Eierlegende Wollmilchsau surface $X_{EW}$. On the left its vertical cylinder decomposition, while the horizontal cylinder decomposition appears on the upper part of the right side of the picture. In both figures it is represented the same path. On the lower part of the right side of the picture are represented the three intersection criteria stated in Lemma~\ref{LemmaTransversalityCriteriaWollmilchsau}, where the line segment $S$ with slope $-\infty\leq \alpha(S)<-1$ is represented in green and the line segment $I$ with slope $0<\alpha(I)<1$ is represented in red.}
\label{FigureEierlegendeWollmilchsau}
\end{figure}
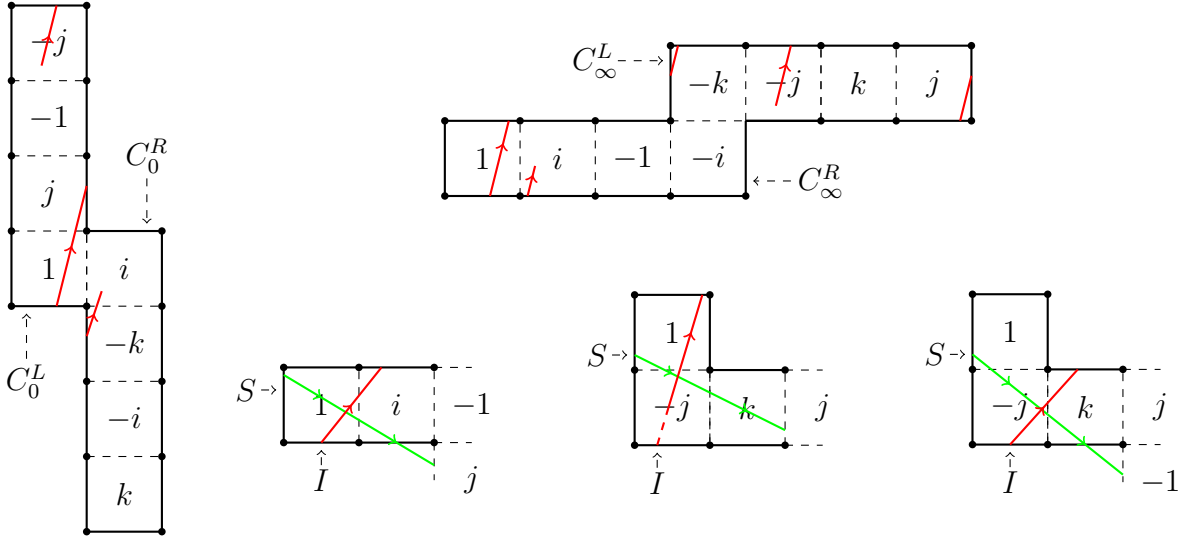

The surface $X_{EW}$ has two cylinders $C_0^L$ and $C_0^R$ in the vertical slope $p/q=0$, respectively around closed geodesics $\sigma_0^L$ and $\sigma_0^R$. These closed geodesics have length 
$
|\sigma_0^L|=|\sigma_0^R|=4
$ 
and the corresponding cylinders have transversal width $W(C_0^L)=W(C_0^R)=1$. Equation~\eqref{EquationOrbitEierlegendeWollmilchsau} implies that for any $A\in\sltwoz$, there are two cylinders $C_{p/q}^L$ and $C_{p/q}^R$ in the rational slope $p/q=A\cdot 0$, and in particular this holds for the horizontal slope $p/q=\infty$. 

Consider a line segment $S\subset X_{EW}$, that is a segment $S:(a,b)\to X_{EW}$ parametrized with constant speed $dS(t)/dt=(u_1,u_2)\in\RR^2$, so that the slope of $S$ is $\alpha(S):=u_1/u_2$. Any finite segment of trajectory of $\phi_\alpha$ is a natural example, but in the following we will consider both flow segments and segments transversal to the flow. If $\alpha(S)\not=0$, that is $S$ is not vertical, then it admits a \emph{vertical cutting sequence} 
$$
[S]^V=(g_0,\dots,g_L)
\quad
\textrm{ with }
\quad
g_r\in\cQ
\quad
\textrm{ for }
\quad
r=0,\dots,L,
$$
where we define 
$
t_0:=\min\{t\geq a,\exists g\in\cQ:S(t)\in l_g\}
$ 
and inductively for $r=0,\dots,L$ the instants $t_r\in(a,b)$ by 
$
t_r:=\min\{t>t_{r-1},\exists g\in\cQ:S(t)\in l_g\}
$ 
and the symbols $g_r\in\cQ$ by 
$$
S(t_r)\in l_{g_r}.
$$
Similarly, if $\alpha(S)\not=\infty$, that is $S$ is not horizontal, then it admits an \emph{horizontal cutting sequence} 
$$
[S]^H=(g_0,\dots,g_L)
\quad
\textrm{ with }
\quad
g_r\in\cQ
\quad
\textrm{ for }
\quad
r=0,\dots,L,
$$
where we define 
$
s_0:=\min\{s\geq a,\exists g\in\cQ:S(s)\in b_g\}
$ 
and inductively for $r=0,\dots,L$ the instants $s_r\in(a,b)$ by 
$
s_r:=\min\{s>s_{r-1},\exists g\in\cQ:S(s)\in b_g\}
$ 
and the symbols $g_r\in\cQ$ by 
$$
S(s_r)\in b_{g_r}.
$$
It is convenient to express both vertical and horizontal cutting sequences of line segments $S$ in a reduced form. If  
$
[S]^{V/H}=(g_0,\dots,g_L)
$ 
is such cutting sequence, we write
$$
[S]^{V/H}=g_0\cdot (1,\dots,g'_L)
\quad
\textrm{ where }
\quad
g'_r:=g_0^{-1}\cdot g_r
\quad
\textrm{ for }
\quad
r=0,\dots,L.
$$ 
On the other hand, if the cutting sequence 
$
[S]^{V/H}=(g_0,\dots,g_L)
$ 
is in its non-reduced form, we write its $r$-th letter as $[S]^{V/H}_r:=g_r$ for $r=0,\dots,L$.

\begin{lemma}
\label{LemmaEierlegendeWollmilchsauTransitionsHorizontal}
Fix a subset $\cE\subset\cQ$ and let $I$ be a line segment with slope $0<\alpha(I)<1$, whose horizontal cutting sequence is $[I]^H=(g_0,\dots,g_L)$. Let $r\leq L-1$ be such that  
$$
[I]^H_r \in \cE\cdot(-j)\cap \cE\cdot(-k).
$$
Then $[I]^H_{r+1}\in \cE$.
\end{lemma}

\begin{proof}
Recalling that $i\cdot j=k$, it is easy to see from the definition of $X_{EW}$ that for a line segment $I$ with slope $\alpha(I)\in (0,1)$ we have 
$$
\quad
\textrm{ either }
\quad
[I]^H_{r+1}=[I]^H_{r}\cdot j
\quad
\textrm{ or }
\quad
[I]^H_{r+1}=[I]^H_{r}\cdot k.
$$
The Lemma follows observing that in order to have $[I]^H_{r+1}\in S$ it is enough to have 
$$
\left\{
\begin{array}{c}
[I]^H_r\cdot j\in \cE\\[1ex]
[I]^H_r\cdot k\in \cE
\end{array}
\right.
\Leftrightarrow
\left\{
\begin{array}{c}
[I]^H_r \in \cE\cdot (-j)\\[1ex]
[I]^H_r \in \cE\cdot (-k)
\end{array}
\right.
\Leftrightarrow
[I]^H_r \in \cE\cdot (-j)\cap \cE\cdot (-k).
$$
\end{proof}

\subsection{Intersection Lemmas}

Lemma~\ref{LemmaTransversalityCriteriaWollmilchsau} below establishes easy intersection criteria in terms of vertical and horizontal cutting sequences, corresponding to situations which are represented in 
Figure~\ref{FigureEierlegendeWollmilchsau}. The proof is left to the reader.

\begin{lemma}
\label{LemmaTransversalityCriteriaWollmilchsau}
Let $S$ and $I$ be segments in $X_{EW}$ with slopes respectively
$
-\infty\leq \alpha(S)<-1
$ 
and  
$
0<\alpha(I)<1
$. 
Let $L=L(I)$ and $L'=L'(S)$ in $\NN$ be such that
$
[I]^H=(g_0,\dots,g_{L})
$ 
and respectively 
$
[S]^V=(g_0,\dots,g_{L'})
$. 
We have $S\cap I\not=\emptyset$ whenever there exists $g\in\cQ$ satisfying one of the conditions above.
\begin{enumerate}
\item
We have $[I]^H_r=g$ for some $r$ with $0\leq r\leq L-1$ and 
$
([S]^V_m,[S]^V_{m+1})=g\cdot (1,i)
$ 
for some $m$ with $m\leq L'-1$.
\item
We have $[I]^H_r=g$ for some $r$ with $1\leq r\leq L-1$ and 
$
([S]^V_m,[S]^V_{m+1})=g\cdot(1,k)
$ 
for some $m$ with $1\leq m\leq L'$.
\item
We have $[I]^H_r=g\cdot(-j)$ for some $r$ with $0\leq r\leq L-2$ and 
$
([S]^V_m,[S]^V_{m+1})=g\cdot(1,k)
$ 
for some $m$ with $0\leq m\leq L'-1$. 
\end{enumerate}
\end{lemma}

Let $C_0\subset X_{EW}$ be any of the two cylinders $C_0^L$ and $C_0^R$ in the vertical slope $p/q=0$. Recall from \S~\ref{SectionOneCylinderVerticalDirections} that a line segment $S:(0,1)\to C_0$ is said transversal to $C_0$ if it is contained in its interior, and moreover $S(0)\in\partial C_0$, 
$
S(1)\in\partial C_0
$ 
and $-\infty\leq\alpha(S)\leq-1$. A line segment $S:(0,4)\to X_{EW}$ is \emph{strongly transversal to the vertical} if the restriction $S_r:=S|_{(r,r+1)}$ is transversal to either $C^L_0$ or $C^R_0$ for any $r=0,1,2,3$. The vertical cutting sequence of such segment contains five letters, that is 
$
[S]^V=(g_0,g_1,g_2,g_3,g_4)
$.

\begin{lemma}
\label{LemmaSmallSlopeWollmilchsauVertical}
Let $S:(0,4)\to X_{EW}$ be a segment strongly transversal to the vertical, and consider any slope $\alpha\in(0,1)$. Then for any $p$ not on any $(X_{EW},\alpha)$-singular leaf we have
$$
I\cap S\not=\emptyset
\quad
\textrm{ where }
\quad
I:=\left\{\phi_\alpha^t(p);0\leq t\leq 7\cdot\sqrt{2}\right\}.
$$ 
\end{lemma}

\begin{proof}
Observe that, since $0<\alpha<1$ then the cutting sequence $[I]^H=(g_0,\dots,g_L)$ satisfies $L\geq 7$. Observe also that there is no loss of generality to assume $[S]^V_0=1$, so that the reduced form of the cutting sequence of $S$ is $[S]^V=(1,g_1,g_2,g_3,g_4)$. Finally, since 
$
-\infty\leq\alpha(S)<-1
$, 
then, reasoning as in the proof of Lemma~\ref{LemmaEierlegendeWollmilchsauTransitionsHorizontal}, it is easy to see that the block $([S]^H_0,[S]^H_1,[S]^H_2)$ admits only the four values $(1,i,-1)$, $(1,i,-j)$, $(1,k,j)$ and $(1,k,-1)$. We reason by absurd, considering separately the four cases. In each case, assuming that $I\cap S=\emptyset$, for any $r$ with $1\leq r\leq L-2$ we obtain a set of \emph{forbidden letters} $\cE_r\subset\cQ$ for which Lemma~\ref{LemmaTransversalityCriteriaWollmilchsau} implies $[I]^H_r\not\in\cE_r$. The proof finishes when we get $\cE_1=\cQ$, which is of course absurd.

If $([S]^H_0,[S]^H_1,[S]^H_2)=(1,i,-1)$, then $[S]^H$ contains the 2 blocks $(1,i)$ and $i\cdot (1,i)$, and also either the block $-1\cdot(1,i)$ or $-1\cdot(1,k)$, all the blocks not in last position in $[S]^V$. According to Points (1) and (2) of Lemma~\ref{LemmaTransversalityCriteriaWollmilchsau} we have 
$
[I]^H_{L-2}\not\in\cE_{L-2}:=\{1,i,-1\}
$. 
Lemma~\ref{LemmaEierlegendeWollmilchsauTransitionsHorizontal} gives the sequence of increasing sets of forbidden letters
\begin{align*}
&
\cE_{L-3}=
\cE_{L-2}\cup\big(\cE_{L-2}\cdot(-k)\cap \cE_{L-2}\cdot(-j)\big)=
\{1,i,-1\}\cup\{-k,j\}
\\
&
\cE_{L-4}=
\cE_{L-3}\cup\big(\cE_{L-3}\cdot(-k)\cap \cE_{L-3}\cdot(-j)\big)=
\{1,i,-1,-k,j\}\cup\{-i\}
\\
&
\cE_{L-5}=
\cE_{L-4}\cup\big(\cE_{L-4}\cdot(-k)\cap \cE_{L-4}\cdot(-j)\big)=
\{1,i,-1,-k,j,-i\}\cup\{-j,k\}=\cQ.
\end{align*}

If $([S]^H_0,[S]^H_1,[S]^H_2)=(1,i,-j)$, then $[S]^H$ contains the 2 blocks $(1,i)$ and $i\cdot (1,k)$, and also either the block $-j\cdot(1,i)$ or $-j\cdot(1,k)$, all the blocks not in last position in $[S]^V$. According to Points (1), (2) and (3) of Lemma~\ref{LemmaTransversalityCriteriaWollmilchsau} we have 
$
[I]^H_{L-2}\not\in\cE_{L-2}:=\{1,i,-k,-j\}
$. 
Lemma~\ref{LemmaEierlegendeWollmilchsauTransitionsHorizontal} gives the sequence of increasing sets of forbidden letters
\begin{align*}
&
\cE_{L-3}=
\cE_{L-2}\cup\big(\cE_{L-2}\cdot(-k)\cap \cE_{L-2}\cdot(-j)\big)=
\{1,i,-k,-j\}\cup\{-k,-1\}
\\
&
\cE_{L-4}=
\cE_{L-3}\cup\big(\cE_{L-3}\cdot(-k)\cap \cE_{L-3}\cdot(-j)\big)=
\{1,i,-k,-j,-1\}\cup\{j\}
\\
&
\cE_{L-5}=
\cE_{L-4}\cup\big(\cE_{L-4}\cdot(-k)\cap \cE_{L-4}\cdot(-j)\big)=
\{1,i,-k,-j,-1,j\}\cup\{-i\}
\\
&
\cE_{L-6}=
\cE_{L-5}\cup\big(\cE_{L-5}\cdot(-k)\cap \cE_{L-5}\cdot(-j)\big)=
\{1,i,-k,-j,-1,j,-i\}\cup\{k\}=\cQ.
\end{align*}

If $([S]^H_0,[S]^H_1,[S]^H_2)=(1,k,j)$, then $[S]^H$ contains the 2 blocks $(1,k)$ and $k\cdot (1,i)$, and also either the block $j\cdot(1,i)$ or $j\cdot(1,k)$, all the blocks not in last position in $[S]^V$. According to Points (1), (2) and (3) of Lemma~\ref{LemmaTransversalityCriteriaWollmilchsau} we have 
$
[I]^H_{L-2}\not\in\cE_{L-2}:=\{1,-j,k,j\}
$. 
Lemma~\ref{LemmaEierlegendeWollmilchsauTransitionsHorizontal} gives the sequence of increasing sets of forbidden letters
\begin{align*}
&
\cE_{L-3}=
\cE_{L-2}\cup\big(\cE_{L-2}\cdot(-k)\cap \cE_{L-2}\cdot(-j)\big)=
\{1,-j,k,j\}\cup\{i,1\}
\\
&
\cE_{L-4}=
\cE_{L-3}\cup\big(\cE_{L-3}\cdot(-k)\cap \cE_{L-3}\cdot(-j)\big)=
\{1,-j,k,j,i\}\cup\{-k\}
\\
&
\cE_{L-5}=
\cE_{L-4}\cup\big(\cE_{L-4}\cdot(-k)\cap \cE_{L-4}\cdot(-j)\big)=
\{1,-j,k,j,i,-k\}\cup\{-1,-i\}=\cQ.
\end{align*}

If $([S]^H_0,[S]^H_1,[S]^H_2)=(1,k,-1)$, then $[S]^H$ contains the 2 blocks $(1,k)$ and $k\cdot (1,k)$, and also either the block $-1\cdot(1,i)$ or $-1\cdot(1,k)$, all the blocks not in last position in $[S]^V$. According to Points (1), (2) and (3) of Lemma~\ref{LemmaTransversalityCriteriaWollmilchsau} we have 
$
[I]^H_{L-2}\not\in\cE_{L-2}:=\{1,-j,k,i,-1\}
$. 
Lemma~\ref{LemmaEierlegendeWollmilchsauTransitionsHorizontal} gives the sequence of increasing sets of forbidden letters
\begin{align*}
&
\cE_{L-3}=
\cE_{L-2}\cup\big(\cE_{L-2}\cdot(-k)\cap \cE_{L-2}\cdot(-j)\big)=
\{1,-j,k,i,-1\}\cup\{-k,j\}
\\
&
\cE_{L-4}=
\cE_{L-3}\cup\big(\cE_{L-3}\cdot(-k)\cap \cE_{L-3}\cdot(-j)\big)=
\{1,-j,k,i,-1,-k,j\}\cup\{-i\}=\cQ.
\end{align*}
\end{proof}

Now let $C_\infty$ be either $C^R_\infty$ or $C^L_\infty$. A line segment $S:(0,1)\to C_\infty$ is \emph{transversal to} $C_\infty$ if it is contained in its interior and moreover $S(0)\in\partial C_\infty$, $S(1)\in\partial C_\infty$ and $-1\leq \alpha(S)<0$. A line segment $S:(0,4)\to X_{EW}$ is said \emph{strongly transversal to the horizontal} if $S_r:=S|_{(r,r+1)}$ is transversal either $C^R_\infty$ of $C^L_\infty$ for $r=0,1,2,3$. Arguing as in Lemma~\ref{LemmaSmallSlopeWollmilchsauVertical} one can show the Lemma below, whose proof is left to the reader.

\begin{lemma}
\label{LemmaSmallSlopeWollmilchsauHorizontal}
Let $S:(0,4)\to X_{EW}$ be a segment strongly transversal to the horizontal, and consider any slope $\alpha\in(1,+\infty)$. Then for any $p$ not on any $(X_{EW},\alpha)$-singular leaf we have
$$
I\cap S\not=\emptyset
\quad
\textrm{ where }
\quad
I:=\left\{\phi_\alpha^t(p);0\leq t\leq 7\cdot\sqrt{2}\right\}.
$$ 
\end{lemma}

\subsection{End of the proof}

Arguing as in \S~\ref{SectionOneCylinderVerticalDirections} and replacing Lemma~\ref{LemmaVerticalDirection} by Lemma~\ref{LemmaSmallSlopeWollmilchsauVertical} and Lemma~\ref{LemmaSmallSlopeWollmilchsauHorizontal}, we get the Proposition below.

\begin{proposition}
\label{PropositionSmallSlopeWollmilchsau}
Fix any slope  
$
\alpha=[a_1,a_2,\dots]\in(0,1)
$. 
Let $p\in X_{EW}$ be a point not on any singular leaf and $p'\in X_{EW}$ be a point not on any saddle connection. Then for any $n\in\NN$ we have
$$
R(X_{EW},\alpha,p,p',r_n)\leq \sqrt{1568}\cdot q_n
\quad
\textrm{ where }
\quad
r_n:=\frac{16}{q_{n}}
$$
\end{proposition}

\begin{proof}
Consider the case of even $n$, that is $n=2k$. Set 
$
A:=g(a_1,\dots,a_{2k})
$ 
and let $\alpha_{2k}$ be the slope related to $\alpha$ by Equation~\eqref{EquationActionSL(2,Z)SlopesIrrational}, that is 
$
\alpha=A\cdot \alpha_{2k}
$. 
Following \S~\ref{SectionFlowSegmentsVerticalCylinder} and recalling Equation~\eqref{EquationOrbitEierlegendeWollmilchsau}, consider the affine diffeomorphism $f_A:X_{EW}\to X_{EW}$, which sends the cylinders $C_0^L$ and $C_0^R$ to the cylinders $C_{2k}^L$ and $C_{2k}^R$ respectively, which have slope $p_{2k}/q_{2k}=A\cdot 0$. Let $S^\perp$ be a segment passing through $p'$ with slope $-q_{2k}/p_{2k}$, that is orthogonal to $p_{2k}/q_{2k}$, and such that $S^\perp$ crosses exactly twice both cylinders $C_{2k}^L$ and $C_{2k}^R$. Such segment exists because $p'$ does not belong to any saddle connection. Moreover $S^\perp$ has length  
$
|S^\perp|=16\cdot (q_{2k}^2+p_{2k}^2)^{-1/2}
$, 
indeed we have 
$
4=\area(C^{L/R}_{p_{2k}/q_{2k}})=|S^\perp_\ast|\cdot\sqrt{q_{2k}^2+p_{2k}^2}
$ 
for any subsegment $S^\perp_\ast\subset S^\perp$ crossing perpendicularly $C^{L/R}_{p_{2k}/q_{2k}}$ exactly once. The slope $\alpha=\alpha(S)$ of the segment $S:=f_A^{-1}(S^\perp)$ is given by the same computation as in Equation~\eqref{EquationPropositionOneCylinderDirection}, thus $S$ is strongly transversal to the vertical. Therefore, according to Lemma~\ref{LemmaSmallSlopeWollmilchsauVertical} we have $t>0$ with
$$
\phi_{\alpha_{2k}}^t\big(f_A^{-1}(p)\big)\in f_A^{-1}(S^\perp)
\quad
\textrm{ and }
\quad
0\leq t\leq 2\cdot\sqrt{1+\alpha_{2k}^2}.
$$ 
Reasoning as in Proposition~\ref{PropositionOneCylinderDirection}, consider $T>0$ such that 
$
f_A\circ \phi_{\alpha_{2k}}^t=\phi_\alpha^T\circ f_A
$, 
that is 
$
\phi^T_\alpha(p)\in S^\perp
$. 
Equation~\eqref{EquationMaximalDilatation} implies 
$$
0\leq T\leq 4\cdot q_{2n}\cdot7\cdot\sqrt{2}
=
\sqrt{1568}\cdot q_{2k}.
$$
Finally, since both $p'$ and $\phi^T_\alpha$ belong to $S^\perp$, we get 
$$
\left|
\phi_{\alpha}^{T}(p)-p'
\right|
\leq
|S^\perp|
= 
\frac{16}{\sqrt{q_{2k}^2+p_{2k}^2}}
\leq
\frac{16}{q_{2k}}
=
r_k.
$$

In the case $n=2k-1$, set 
$
A:=g(a_1,\dots,a_{2k},a_{2k+1})
$ 
and let $\alpha_{2k+1}$ be the slope related to $\alpha$ by Equation~\eqref{EquationActionSL(2,Z)SlopesIrrational}, that is 
$
\alpha=A\cdot \alpha_{2k+1}^{-1}
$. 
As above, consider the affine diffeomorphism $f_A:X_{EW}\to X_{EW}$ which sends the cylinders $C_\infty^L$ and $C_\infty^R$ to the cylinders $C_{2k+1}^L$ and $C_{2k+1}^R$ respectively, which have slope $p_{2k+1}/q_{2k+1}=A\cdot \infty$. The required estimate follows by the same argument as above, replacing Lemma~\ref{LemmaSmallSlopeWollmilchsauVertical} by Lemma~\ref{LemmaSmallSlopeWollmilchsauHorizontal}. Details are left to the reader. Proposition~\ref{PropositionSmallSlopeWollmilchsau} is proved.
\end{proof}

Here we finish the proof of Proposition~\ref{PropositionEierlegendeWollmilchsau}. It is no loss of generality to consider $\alpha\in(0,1)$. Indeed, letting $a:=[\alpha]\in\ZZ$ the integer part of $\alpha$ and $\beta:=\alpha-a$, we have obviously 
$
w(\alpha)=w(\beta)
$. 
On the other hand, reasoning as in \S~\ref{SectionReducedOrigamiAreEnough} and setting $A:=T^{-a}$, it is also clear that the flow $\phi_\alpha$ on $X$ corresponds to the flow $\phi_\beta$ on $A\cdot X$ under some affine diffeomorphism $f_A:X\to A\cdot X$, thus the two flows have the same hitting time.

\medskip

Fix points $p,p'$, with $p'$ not on any saddle connection and $p$ not on any $(X_{EW},\alpha)$-singular leaf. For any $n\in\NN$ set $r_n:=16\cdot q_{n}^{-1}$. For any $r>0$ small enough consider $n$ such that $r_{n-1}\leq r<r_{n}$. Since $w(\alpha)=\eta$ then $q_n\leq C\cdot q_{n-1}^\eta$ for some $C$ and all $n$ big enough. Proposition~\ref{PropositionSmallSlopeWollmilchsau} implies
\begin{align*}
\frac{\log R(X_{EW},\alpha,p,p',r)}{|\log r|}
&
\leq
\frac{\log R(X_{EW},\alpha,p,p',r_{n})}{|\log r_{n-1}|}
\leq
\frac{\log \sqrt{1568}+\log q_n}{\log q_{n-1}-\log 16}
\\
&
\leq
\frac{\sqrt{1568}+\log C+\eta\cdot \log q_{n-1}}{\log q_{n-1}-\log 16}\to\eta
\quad
\textrm{ for }
\quad
n\to+\infty.
\end{align*}
Therefore $\whitsup(\phi_\alpha,p,p')\leq \eta$ for any $p,p'$ as above. The last inequality turns into an equality for almost any $p,p'$ in $X_{EW}$ according to Equation~\eqref{EquationLimSupHittingVeech}. Proposition~\ref{PropositionEierlegendeWollmilchsau} is proved. $\qed$

\appendix

\section{Proof of Lemma \ref{LemmaExistenceSplittingPairs}}
\label{SectionProofExistenceSplittingPairsH(2)}

This subsection follows \cite{HubertLelievre} closely. It is more convenient to represent long cylinders along the horizontal direction. Thus we prove that any orbit in $\cH(2)$ contains an origami $X_0$ with an horizontal splitting pair. The required origami is $R\cdot X_0$, where 
$$
R=
\begin{pmatrix}
0 & -1 \\
1 & 0
\end{pmatrix}
=
T^{-1}VT^{-1}.
$$

\subsection{Separatrix diagrams}

An origami $X$ in $\cH(2)$ has a cone point of angle $6\pi$, thus for any $\alpha$ there are $3$ outgoing $(X,\alpha)$-singular leaves at $p$, also called \emph{separatrices}. The horizontal direction $\alpha=\infty$ is completely periodic, so that the horizontal separatrices are saddle connections. More precisely, each of them returns to the conical point making an angle $\pi$, $3\pi$ or $5\pi$ with itself. Label the horizontal saddle connections with symbols in $\{1,2,3\}$ according to their counterclockwise order around the conical point, then define a permutation $f$ of $\{1,2,3\}$ by setting $f(i)=j$ if the saddle connection $i$ comes back between the saddle connections $j$ and $j+1\mod3$, that is with angle 
$$
\big(2\big(f(i)-i\mod 3\big)+1\big)\pi.
$$
The combinatorics of these connections is called a \emph{separatrix diagram}. The surface is obtained from such diagram by gluing cylinders along the saddle connections in the diagram, and is determined uniquely by the diagram and the metric data of the cylinders. At a conical point of angle $6\pi$, up to rotation by $2\pi$ around the conical point, there are in total four separatrix diagrams, which correspond to return angles $(\pi,\pi,\pi)$, $(\pi,3\pi,5\pi)$, $(3\pi,3\pi,3\pi)$ and $(5\pi,5\pi,5\pi)$ and are shown in the upper part of Figure~\ref{FigureSeparatrixDiagrams}. These four separatrix diagrams correspond to conjugacy classes of permutations $f\in S_3$ under cyclic permutations. Fix any such $f$ and consider an origami $X\in\cH(2)$ with separatrix diagram encoded by $f$. A cycle of $f$ corresponds to the lower boundary of an horizontal cylinder in $X$, the saddle connections composing it being those whose labels are the elements of the cycle. The upper boundaries of the horizontal cylinders in $X$ correspond to the cycles of the permutation $f'\in S_3$ defined by
$$
f'(i)=f(i)+1\mod3
$$
The separatrix diagram associated to $f\in S_3$ can be realized geometrically by a surface $X$ in $\cH(2)$ if the cycles of $f$ and $f'$ appear as lower and upper boundaries of horizontal cylinders, which is a non-trivial condition. The first diagram in the upper part of Figure~\ref{FigureSeparatrixDiagrams} corresponds to $f=(1)(2)(3)$, that is the identity of $S_3$, which gives $f'=(1,2,3)$, thus the diagram is not realizable because the cycles of $f$ correspond to $3$ horizontal cylinders while $f'$ has only one cycle, which corresponds to just one cylinder. Similarly the fourth diagram is not geometrically realizable, because $f=(1,3,2)$ gives $f'=(1)(2)(3)$. The second diagram in the upper part of Figure~\ref{FigureSeparatrixDiagrams} corresponds to $f=(1)(2,3)$, which gives $f'=(1,2)(3)$, and is realized geometrically by the first surface in the lower part of the figure, which has two horizontal cylinders. Finally, the third diagram in the upper part of Figure~\ref{FigureSeparatrixDiagrams} corresponds to $f=(1,2,3)$, which gives $f'=(1,3,2)$, and is realized geometrically by the second surface in the lower part of the figure, which has just one horizontal cylinder.

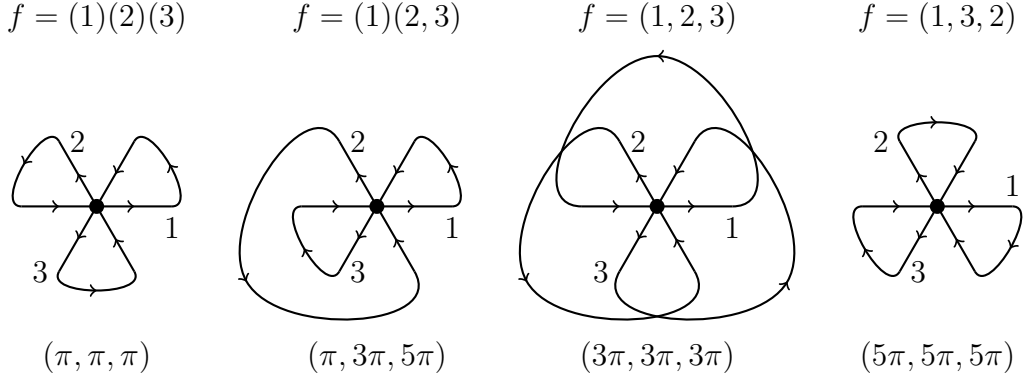
\begin{figure}
\begin{center}
{\begin{tikzpicture}[scale=0.5]
\tikzset{->-/.style={decoration={markings,mark=at position .5 with {\arrow{>}}},postaction={decorate}}}
\coordinate (A0) at ({2*cos(0)},{2*sin(0)});
\coordinate (A1) at ({2*cos(60)},{2*sin(60)});
\coordinate (A2) at ({2*cos(120)},{2*sin(120)});
\coordinate (A3) at ({2*cos(180)},{2*sin(180)});
\coordinate (A4) at ({2*cos(240)},{2*sin(240)});
\coordinate (A5) at ({2*cos(300)},{2*sin(300)});

\node [] at (A0) [below] {$1$};
\node [] at (A1) {};
\node [] at (A2) [right] {$2$};
\node [] at (A3) {};
\node [] at (A4) [left ]{$3$};
\node [] at (A5) {};
 
\node [circle,fill,inner sep=2pt] at (0,0) {};

\draw[->-,thick] (0,0) -- ({2*cos(0)},{2*sin(0)});
\draw[->-,thick] ({2*cos(60)},{2*sin(60)}) -- (0,0);
\draw[->-,thick] (0,0) -- ({2*cos(120)},{2*sin(120)});
\draw[->-,thick] ({2*cos(180)},{2*sin(180)}) -- (0,0);
\draw[->-,thick] (0,0) -- ({2*cos(240)},{2*sin(240)});
\draw[->-,thick] ({2*cos(300)},{2*sin(300)}) -- (0,0);

\draw [->-,thick] (A0) to[out=0,in=60] (A1);

\draw [->-,thick] (A2) to[out=120,in=180] (A3);

\draw [->-,thick] (A4) to[out=240,in=300] (A5);

\node [] at (0,-4) {$(\pi,\pi,\pi)$};
\node [] at (0,5) {$f=(1)(2)(3)$};

\end{tikzpicture}}
{\begin{tikzpicture}[scale=0.5]
\tikzset{->-/.style={decoration={markings,mark=at position .5 with {\arrow{>}}},postaction={decorate}}}

\coordinate (B0) at ({2*cos(0)},{2*sin(0)});
\coordinate (B1) at ({2*cos(60)},{2*sin(60)});
\coordinate (B2) at ({2*cos(120)},{2*sin(120)});
\coordinate (B3) at ({2*cos(180)},{2*sin(180)});
\coordinate (B4) at ({2*cos(240)},{2*sin(240)});
\coordinate (B5) at ({2*cos(300)},{2*sin(300)});

\node [] at (B0) [below] {$1$};
\node [] at (B1) {};
\node [] at (B2) [right] {$2$};
\node [] at (B3) {};
\node [] at (B4) [right] {$3$};
\node [] at (B5) {};
 
\node [circle,fill,inner sep=2pt] at (0,0) {};

\draw[->-,thick] (0,0) -- ({2*cos(0)},{2*sin(0)});
\draw[->-,thick] ({2*cos(60)},{2*sin(60)}) -- (0,0);
\draw[->-,thick] (0,0) -- ({2*cos(120)},{2*sin(120)});
\draw[->-,thick] ({2*cos(180)},{2*sin(180)}) -- (0,0);
\draw[->-,thick] (0,0) -- ({2*cos(240)},{2*sin(240)});
\draw[->-,thick] ({2*cos(300)},{2*sin(300)}) -- (0,0);

\draw [->-,thick] (B0) to[out=0,in=60] (B1);

\draw [->-,thick] (B4) to[out=240,in=180] (B3);

\coordinate (B2-5) at ({4*cos(210)},{4*sin(210)});
\draw [->-,thick] (B2) to[out=120,in=120] (B2-5) to[out=300,in=300] (B5);

\node [] at (0,-4) {$(\pi,3\pi,5\pi)$};
\node [] at (0,5) {$f=(1)(2,3)$};

\end{tikzpicture}}
{\begin{tikzpicture}[scale=0.5]
\tikzset{->-/.style={decoration={markings,mark=at position .5 with {\arrow{>}}},postaction={decorate}}}

\coordinate (C0) at ({2*cos(0)},{2*sin(0)});
\coordinate (C1) at ({2*cos(60)},{2*sin(60)});
\coordinate (C2) at ({2*cos(120)},{2*sin(120)});
\coordinate (C3) at ({2*cos(180)},{2*sin(180)});
\coordinate (C4) at ({2*cos(240)},{2*sin(240)});
\coordinate (C5) at ({2*cos(300)},{2*sin(300)});

\node [] at (C0) [below] {$1$};
\node [] at (C1) {};
\node [] at (C2) [right] {$2$};
\node [] at (C3) {};
\node [] at (C4) [left] {$3$};
\node [] at (C5) {};
 
\node [circle,fill,inner sep=2pt] at (0,0) {};

\draw[->-,thick] (0,0) -- ({2*cos(0)},{2*sin(0)});
\draw[->-,thick] ({2*cos(60)},{2*sin(60)}) -- (0,0);
\draw[->-,thick] (0,0) -- ({2*cos(120)},{2*sin(120)});
\draw[->-,thick] ({2*cos(180)},{2*sin(180)}) -- (0,0);
\draw[->-,thick] (0,0) -- ({2*cos(240)},{2*sin(240)});
\draw[->-,thick] ({2*cos(300)},{2*sin(300)}) -- (0,0);

\coordinate (C0-3) at ({4*cos(90)},{4*sin(90)});
\draw [->-,thick] (C0) to[out=0,in=0] (C0-3) to[out=180,in=180] (C3);

\coordinate (C2-5) at ({4*cos(210)},{4*sin(210)});
\draw [->-,thick] (C2) to[out=120,in=120] (C2-5) to[out=300,in=300] (C5);

\coordinate (C4-1) at ({4*cos(330)},{4*sin(330)});
\draw [->-,thick] (C4) to[out=240,in=240] (C4-1) to[out=60,in=60] (C1);

\node [] at (0,-4) {$(3\pi,3\pi,3\pi)$};
\node [] at (0,5) {$f=(1,2,3)$};

\end{tikzpicture}}
{\begin{tikzpicture}[scale=0.5]
\tikzset{->-/.style={decoration={markings,mark=at position .5 with {\arrow{>}}},postaction={decorate}}}

\coordinate (D0) at ({2*cos(0)},{2*sin(0)});
\coordinate (D1) at ({2*cos(60)},{2*sin(60)});
\coordinate (D2) at ({2*cos(120)},{2*sin(120)});
\coordinate (D3) at ({2*cos(180)},{2*sin(180)});
\coordinate (D4) at ({2*cos(240)},{2*sin(240)});
\coordinate (D5) at ({2*cos(300)},{2*sin(300)});


\node [] at (D0) [above] {$1$};
\node [] at (D1) {};
\node [] at (D2) [left] {$2$};
\node [] at (D3) {};
\node [] at (D4) [right] {$3$};
\node [] at (D5) {};
 
\node [circle,fill,inner sep=2pt] at (0,0) {};

\draw[->-,thick] (0,0) -- ({2*cos(0)},{2*sin(0)});
\draw[->-,thick] ({2*cos(60)},{2*sin(60)}) -- (0,0);
\draw[->-,thick] (0,0) -- ({2*cos(120)},{2*sin(120)});
\draw[->-,thick] ({2*cos(180)},{2*sin(180)}) -- (0,0);
\draw[->-,thick] (0,0) -- ({2*cos(240)},{2*sin(240)});
\draw[->-,thick] ({2*cos(300)},{2*sin(300)}) -- (0,0);

\draw [->-,thick] (D0) to[out=0,in=300] (D5);

\draw [->-,thick] (D2) to[out=120,in=60] (D1);

\draw [->-,thick] (D4) to[out=240,in=180] (D3);

\node [] at (0,-4) {$(5\pi,5\pi,5\pi)$};
\node [] at (0,5) {$f=(1,3,2)$};

\end{tikzpicture}}
\end{center}
\caption{The four separatrix diagrams at a conical point of angle $6\pi$, with the corresponding $f\in S_3$ and the return angles represented respectively above and below each of them. The second and third  diagrams can be realized by surfaces in $\cH(2)$, which are represented below.}
\label{FigureSeparatrixDiagrams}

\begin{center}
{\begin{tikzpicture}[scale=1]
\tikzset
{->-/.style={decoration={markings,mark=at position .5 with {\arrow{>}}},postaction={decorate}}}

\draw [help lines] (0,0) grid (2,1);

\node [circle,fill,inner sep=1pt] at (0,0) {};
\node [circle,fill,inner sep=1pt] at (1,0) {};
\node [circle,fill,inner sep=1pt] at (2,0) {};
\node [circle,fill,inner sep=1pt] at (0,1) {};
\node [circle,fill,inner sep=1pt] at (1,1) {};
\node [circle,fill,inner sep=1pt] at (2,1) {};
\node [circle,fill,inner sep=1pt] at (1,2) {};
\node [circle,fill,inner sep=1pt] at (2,2) {};

\draw[->-,thick] (0,0) -- (1,0) node[pos=0.7,above] {$2$};
\draw[->-,thick] (1,0) -- (2,0) node[pos=0.7,above] {$3$};
\draw[->-,thick] (1,1) -- (2,1) node[pos=0.7,above] {$1$};
\draw[->-,thick] (0,1) -- (1,1) node[pos=0.7,above] {$2$};
\draw[->-,thick] (1,1) -- (2,1) node[pos=0.7,above] {};
\draw[->-,thick] (1,2) -- (2,2) node[pos=0.7,above] {$3$};

\draw[-,thick] (0,0) -- (0,1) {};
\draw[-,thick] (2,0) -- (2,2) {};
\draw[-,thick] (1,1) -- (1,2) {};

\node [] at (3.5,0.5) {$(\pi,3\pi,5\pi)$};

\end{tikzpicture}}
\quad
\quad
{\begin{tikzpicture}[scale=1]
\tikzset
{->-/.style={decoration={markings,mark=at position .5 with {\arrow{>}}},postaction={decorate}}}

\draw [help lines] (0,0) grid (3,1);

\node [circle,fill,inner sep=1pt] at (0,0) {};
\node [circle,fill,inner sep=1pt] at (1,0) {};
\node [circle,fill,inner sep=1pt] at (2,0) {};
\node [circle,fill,inner sep=1pt] at (3,0) {};
\node [circle,fill,inner sep=1pt] at (0,1) {};
\node [circle,fill,inner sep=1pt] at (1,1) {};
\node [circle,fill,inner sep=1pt] at (2,1) {};
\node [circle,fill,inner sep=1pt] at (3,1) {};

\draw[->-,thick] (0,0) -- (1,0) node[pos=0.7,above] {$2$};
\draw[->-,thick] (1,0) -- (2,0) node[pos=0.7,above] {$1$};
\draw[->-,thick] (2,0) -- (3,0) node[pos=0.7,above] {$3$};
\draw[->-,thick] (0,1) -- (1,1) node[pos=0.7,above] {$3$};
\draw[->-,thick] (1,1) -- (2,1) node[pos=0.7,above] {$1$};
\draw[->-,thick] (2,1) -- (3,1) node[pos=0.7,above] {$2$};

\draw[-,thick] (0,0) -- (0,1) {};
\draw[-,thick] (3,0) -- (3,1) {};

\node [] at (4.5,0.5) {$(3\pi,3\pi,3\pi)$};

\end{tikzpicture}}

\end{center}
\end{figure}

\subsection{Classification of orbits in $\cH(2)$}

Any translation surface $X\in\cH(2)$ admits an unique affine diffeomorphism $\iota:X\to X$ such that $\iota^2=\id_X$ and $D\iota=-\id_{\RR^2}$, which is known as \emph{hyperelliptic involution}. Such involution has $6$ fixed points on $X$, which are known as \emph{Weierstrass points}, and the conical point is always one of them. If $X$ is an origami in $\cH(2)$, then the conical point has always integer coordinates. According to Lemma 4.2 in \cite{HubertLelievre}, the number of integer Weierstrass points, that we denote $\textrm{IWP}\in\NN^\ast$ is an invariant of $\sltwoz$-orbits of origamis $X$ in $\cH(2)$. Moreover it gives a complete invariant, according to the Theorem below, which was first proved for prime $N$ in \cite{HubertLelievre} and then by \cite{McMullen} in the general case.

\begin{theorem}
\label{TheoremClassificationHubertLelievreMcMullen}
Consider $\sltwoz$-orbits of reduced origamis in $\cH(2)$ with $N\geq 3$ squares. For any integer $N\geq 3$ the following holds.
\begin{enumerate}
\item
If $N=3$, there exist an unique orbit. We have $\textrm{IWP}=1$ for any $X$ in such orbit.
\item
If $N$ is even, $N\geq 4$, then there exists an unique orbit $\cE_N$. We have $\textrm{IWP}=2$ for any $X\in\cE_N$.
\item
If $N$ is odd and $N\geq 5$, there there exist exactly two orbits $\cA_N$ and $\cB_N$. We have $\textrm{IWP}=1$ for any $X\in\cA_N$ and $\textrm{IWP}=3$ for any $X\in\cB_N$.
\end{enumerate}
\end{theorem}

\begin{figure}
\begin{center}  
{\begin{tikzpicture}[scale=0.8]
\tikzset
{->-/.style={decoration={markings,mark=at position 0.5 with {\arrow{>}}},postaction={decorate}}}
\tikzset
{-->-/.style={decoration={markings,mark=at position 0.75 with {\arrow{>}}},postaction={decorate}}}

\draw [help lines] (0,0) grid (7,1);
\draw [help lines] (0,1) grid (5,2);

\draw[->-,thick] (0,0) -- (5,0) node[pos=0.3,below] {$\gamma_1$};
\draw[->-,thick] (5,0) -- (7,0) node[pos=0.3,below] {$\gamma_2$};
\draw[->-,thick] (0,1) -- (5,1) node[pos=0.3,above] {$\gamma_3$};
\draw[-,thick] (5,1) -- (7,1) node[pos=0.5,below] {};
\draw[-,thick] (0,2) -- (5,2) node[pos=0.5,below] {};

\draw[->-,thick,dashed] (0,0.5) -- (7,0.5) node[pos=0,left] {$\sigma_2$};
\draw[-->-,thick,dashed] (0,1.5) -- (5,1.5) node[pos=0,left] {$\sigma_1$};

\draw[-,thick] (0,0) -- (0,2) node[pos=0.5,below] {};
\draw[-,thick] (7,0) -- (7,1) node[pos=0.5,below] {};
\draw[-,thick] (5,1) -- (5,2) node[pos=0.5,below] {};

\draw[<->,thin,dashed] (0,-0.5) -- (7,-0.5) node[pos=0.5,below] {$k+1$};
\draw[<->,thin,dashed] (0,2.5) -- (5,2.5) node[pos=0.5,above] {$k-1$};

\node [circle,fill,inner sep=1pt] at (0,0) {};
\node [circle,fill,inner sep=1pt] at (5,0) {};
\node [circle,fill,inner sep=1pt] at (7,0) {};
\node [circle,fill,inner sep=1pt] at (0,1) {};
\node [circle,fill,inner sep=1pt] at (5,1) {};
\node [circle,fill,inner sep=1pt] at (7,1) {};
\node [circle,fill,inner sep=1pt] at (0,2) {};
\node [circle,fill,inner sep=1pt] at (5,2) {};

\node [circle,fill,inner sep=1.5pt] at (2.5,0.5) {};
\node [circle,fill,inner sep=1.5pt] at (6,0.5) {};
\node [circle,fill,inner sep=1.5pt] at (6,1) {};
\node [circle,fill,inner sep=1.5pt] at (2.5,1.5) {};
\node [circle,fill,inner sep=1.5pt] at (5,1.5) {};

\node [] at (9,2) [right] {Orbit $\cE_N$, $N=2k$, $\textrm{IWP}=2$};
\node [] at (9,1) [right] {$|\sigma_2|=k+1$, $W(C_{\sigma_2})=1$};

\end{tikzpicture}}

{\begin{tikzpicture}[scale=0.8]
\tikzset
{->-/.style={decoration={markings,mark=at position 0.5 with {\arrow{>}}},postaction={decorate}}}
\tikzset
{-->-/.style={decoration={markings,mark=at position 0.75 with {\arrow{>}}},postaction={decorate}}}

\draw [help lines] (0,0) grid (6,1);
\draw [help lines] (0,1) grid (5,2);

\draw[->-,thick] (0,0) -- (5,0) node[pos=0.3,below] {$\gamma_1$};
\draw[->-,thick] (5,0) -- (6,0) node[pos=0.3,below] {$\gamma_2$};
\draw[->-,thick] (0,1) -- (5,1) node[pos=0.3,above] {$\gamma_3$};
\draw[-,thick] (5,1) -- (6,1) node[pos=0.5,below] {};
\draw[-,thick] (0,2) -- (5,2) node[pos=0.5,below] {};

\draw[->-,thick,dashed] (0,0.5) -- (6,0.5) node[pos=0,left] {$\sigma_2$};
\draw[-->-,thick,dashed] (0,1.5) -- (5,1.5) node[pos=0,left] {$\sigma_1$};

\draw[-,thick] (0,0) -- (0,2) node[pos=0.5,below] {};
\draw[-,thick] (6,0) -- (6,1) node[pos=0.5,below] {};
\draw[-,thick] (5,1) -- (5,2) node[pos=0.5,below] {};

\draw[<->,thin,dashed] (0,-0.5) -- (6,-0.5) node[pos=0.5,below] {$k+1$};
\draw[<->,thin,dashed] (0,2.5) -- (5,2.5) node[pos=0.5,above] {$k$};

\node [circle,fill,inner sep=1pt] at (0,0) {};
\node [circle,fill,inner sep=1pt] at (5,0) {};
\node [circle,fill,inner sep=1pt] at (6,0) {};
\node [circle,fill,inner sep=1pt] at (0,1) {};
\node [circle,fill,inner sep=1pt] at (5,1) {};
\node [circle,fill,inner sep=1pt] at (0,2) {};
\node [circle,fill,inner sep=1pt] at (5,2) {};
\node [circle,fill,inner sep=1pt] at (6,1) {};

\node [circle,fill,inner sep=1.5pt] at (2.5,0.5) {};
\node [circle,fill,inner sep=1.5pt] at (5.5,0.5) {};
\node [circle,fill,inner sep=1.5pt] at (5.5,1) {};
\node [circle,fill,inner sep=1.5pt] at (2.5,1.5) {};
\node [circle,fill,inner sep=1.5pt] at (5,1.5) {};

\node [] at (9,2) [right] {Orbit $\cA_N$, $N=2k+1$, $\textrm{IWP}=1$};
\node [] at (9,1) [right] {$|\sigma_2|=k+1$, $W(C_{\sigma_2})=1$};

\end{tikzpicture}}

{\begin{tikzpicture}[scale=0.8]
\tikzset
{->-/.style={decoration={markings,mark=at position 0.5 with {\arrow{>}}},postaction={decorate}}}
\tikzset
{-->-/.style={decoration={markings,mark=at position 0.8 with {\arrow{>}}},postaction={decorate}}}

\draw [help lines] (0,0) grid (1,5);
\draw [help lines] (1,0) grid (5,1);

\draw[->-,thick] (0,0) -- (1,0) node[pos=0.3,below] {$\gamma_1$};
\draw[->-,thick] (1,0) -- (5,0) node[pos=0.3,below] {$\gamma_2$};
\draw[->-,thick] (0,1) -- (1,1) node[pos=0.3,above] {$\gamma_3$};
\draw[-,thick] (1,1) -- (5,1) node[pos=0.5,below] {};
\draw[-,thick] (0,5) -- (1,5) node[pos=0.5,below] {};

\draw[->-,thick,dashed] (0,0.5) -- (5,0.5) node[pos=0,left] {$\sigma_2$};
\draw[-->-,thick,dashed] (0,3) -- (1,3) node[pos=0,left] {$\sigma_1$};

\draw[-,thick] (0,0) -- (0,5) node[pos=0.5,below] {};
\draw[-,thick] (5,0) -- (5,1) node[pos=0.5,below] {};
\draw[-,thick] (1,1) -- (1,5) node[pos=0.5,below] {};

\draw[<->,thin,dashed] (1,-0.5) -- (5,-0.5) node[pos=0.5,below] {$2k$};
\draw[<->,thin,dashed] (-1,1) -- (-1,5) node[pos=0.5,left] {$2k$};

\node [circle,fill,inner sep=1pt] at (0,0) {};
\node [circle,fill,inner sep=1pt] at (1,0) {};
\node [circle,fill,inner sep=1pt] at (5,0) {};
\node [circle,fill,inner sep=1pt] at (0,1) {};
\node [circle,fill,inner sep=1pt] at (1,1) {};
\node [circle,fill,inner sep=1pt] at (5,1) {};
\node [circle,fill,inner sep=1pt] at (0,5) {};
\node [circle,fill,inner sep=1pt] at (1,5) {};

\node [circle,fill,inner sep=1.5pt] at (0.5,0.5) {};
\node [circle,fill,inner sep=1.5pt] at (3,0.5) {};
\node [circle,fill,inner sep=1.5pt] at (3,1) {};
\node [circle,fill,inner sep=1.5pt] at (0.5,3) {};
\node [circle,fill,inner sep=1.5pt] at (1,3) {};

\node [] at (9,2) [right] {Orbit $\cB_N$, $N=4k+1$, $\textrm{IWP}=3$};
\node [] at (9,1) [right] {$|\sigma_2|=2k+1$, $W(C_{\sigma_2})=1$};

\end{tikzpicture}}

{\begin{tikzpicture}[scale=0.8]
\tikzset
{->-/.style={decoration={markings,mark=at position 0.5 with {\arrow{>}}},postaction={decorate}}}
\tikzset
{-->-/.style={decoration={markings,mark=at position 0.8 with {\arrow{>}}},postaction={decorate}}}

\draw [help lines] (0,0) grid (1,3);
\draw [help lines] (1,0) grid (7,2);

\draw[->-,thick] (0,0) -- (1,0) node[pos=0.3,below] {$\gamma_1$};
\draw[->-,thick] (1,0) -- (7,0) node[pos=0.3,below] {$\gamma_2$};
\draw[->-,thick] (0,2) -- (1,2) node[pos=0.3,below] {$\gamma_3$};
\draw[-,thick] (1,2) -- (7,2) node[pos=0.5,below] {};
\draw[-,thick] (0,3) -- (1,3) node[pos=0.5,below] {};

\draw[->-,thick,dashed] (0,1) -- (7,1) node[pos=0,left] {$\sigma_2$};
\draw[-->-,thick,dashed] (0,2.5) -- (1,2.5) node[pos=0,left] {$\sigma_1$};

\draw[-,thick] (0,0) -- (0,3) node[pos=0.5,below] {};
\draw[-,thick] (7,0) -- (7,2) node[pos=0.5,below] {};
\draw[-,thick] (1,2) -- (1,3) node[pos=0.5,below] {};

\draw[<->,thin,dashed] (1,-0.5) -- (7,-0.5) node[pos=0.5,below] {$2k$};
\draw[<->,thin,dashed] (-1,0) -- (-1,2) node[pos=0.5,left] {$2$};

\node [circle,fill,inner sep=1pt] at (0,0) {};
\node [circle,fill,inner sep=1pt] at (1,0) {};
\node [circle,fill,inner sep=1pt] at (7,0) {};
\node [circle,fill,inner sep=1pt] at (0,2) {};
\node [circle,fill,inner sep=1pt] at (1,2) {};
\node [circle,fill,inner sep=1pt] at (7,2) {};
\node [circle,fill,inner sep=1pt] at (0,3) {};
\node [circle,fill,inner sep=1pt] at (1,3) {};

\node [circle,fill,inner sep=1.5pt] at (0.5,1) {};
\node [circle,fill,inner sep=1.5pt] at (4,1) {};
\node [circle,fill,inner sep=1.5pt] at (4,2) {};
\node [circle,fill,inner sep=1.5pt] at (0.5,2.5) {};
\node [circle,fill,inner sep=1.5pt] at (1,2.5) {};

\node [] at (9,2) [right] {Orbit $\cB_N$, $N=4k+3$, $\textrm{IWP}=3$};
\node [] at (9,1) [right] {$|\sigma_2|=7$, $W(C_{\sigma_2})=2$};

\end{tikzpicture}}

\end{center}
\caption{For each integer $N\geq 4$ and any value of the invariant $\textrm{IWP}$ the pair $(\sigma_2,\gamma_2)$ is an horizontal splitting pair. One Weierstrass points is always given by the conical point, which has integer coordinates. In each of the four figures, the other five Weierstrass points are represented by a black dot}
\label{FigureExistenceSplittingPair}
\end{figure}
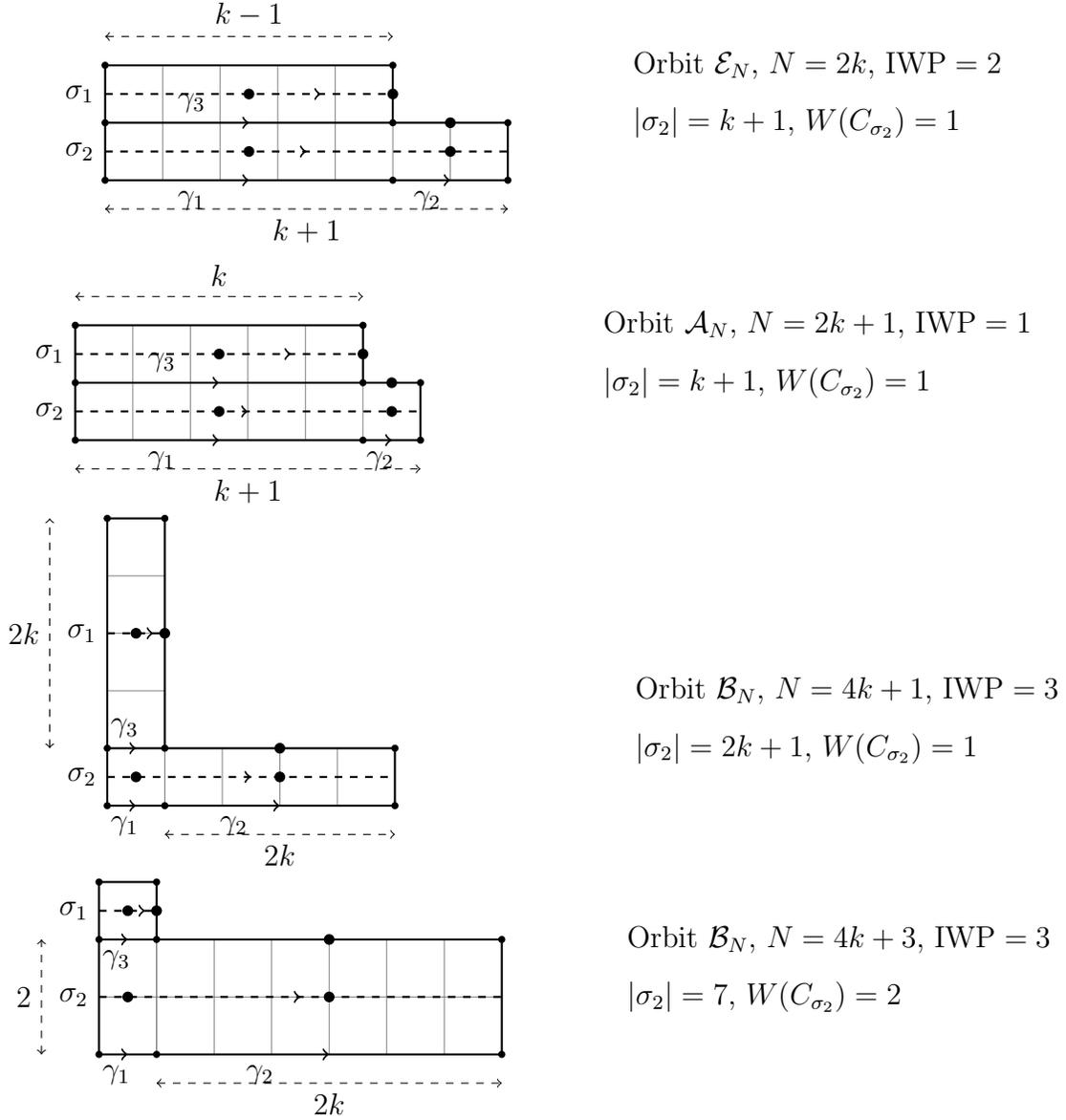

For an origami $X$ in $\cH(2)$ there exist and explicit expressions for the involution $\iota$ and for the remaining 5 Weierstrass points (other than the conical point), according to the combinatorial type of the separatrix diagram of the horizontal direction $\alpha=+\infty$ (see also \S~5.1 in \cite{HubertLelievre}).

\smallskip

\emph{If $\alpha=+\infty$ has type $(3\pi,3\pi,3\pi)$} then $X$ has only one cylinder $C$ in the horizontal direction, with core curve $\sigma$. Since $X$ is reduced, we have $W(C)=1$. In this case $\iota$ is the central symmetry around the center of $C$. The closed geodesic $\sigma$ contains 2 Weirstrass points, one of them being the center of $C$, the other being its opposite point in $\sigma$. They are never integer point, since $W(C)=1$. The remaining 3 Weierstrass points are the centers of the horizontal saddle connections $\gamma_1$, $\gamma_2$ and $\gamma_3$, which may be integer or not, according to the values of the parameters $|\gamma_i|$ for $i=1,2,3$.

\smallskip

\emph{If $\alpha=+\infty$ has type $(\pi,3\pi,5\pi)$} then $X$ has two cylinders $C_1$ and $C_2$ in the horizontal direction, with core curves $\sigma_1$ and $\sigma_2$ respectively, which satisfy 
$
|\sigma_1|=|\gamma_1|
$ 
and 
$
|\sigma_2|=|\gamma_1|+|\gamma_2|=|\gamma_2|+|\gamma_3|
$. 
In this case the longer cylinder $C_2$ can be decomposed as $C_2=P_1\sqcup P_2$, where $P_1$ is a parallelogram whose horizontal boundary is composed by $\gamma_1$ and $\gamma_3$, and where $P_2$ is a parallelogram whose horizontal boundary is given by $\gamma_2$ repeated on both sides. In this case $\iota$ acts separately on $C_1$, $P_1$ and $P_2$ as the central symmetry around their centers, then compatibility at the boundary gives a global map on $X$. The core curve $\sigma_1$ contains 2 Weierstrass points, which are the center of $C_1$ and its opposite point. The core curve $\sigma_2$ contains other 2 Weierstrass points, which are the centers of $P_1$ and of $P_2$. The last Weierstrass point is the center of $\gamma_2$. These points may have integer coordinates or not, according to the values of the parameters $|\gamma_2|$ and $W(C_i)$ and $|\sigma_i|$ for $i=1,2$.

\subsection{End of the proof: existence of splitting pairs in $\cH(2)$} 

Since Theorem \ref{TheoremClassificationHubertLelievreMcMullen} gives a complete invariant for the classification of orbits in $\cH(2)$, then it is enough to find a representative $X_0$ with an horizontal splitting pair for any value of the invariant. In particular we need at least two cylinders in the horizontal direction $\alpha=+\infty$, thus we will consider only surfaces were the horizontal has type $(\pi,3\pi,5\pi)$. In Figure \ref{FigureExistenceSplittingPair}, for each integer $N\geq 4$ and any value of the invariant $\textrm{IWP}$ the pair $(\sigma_2,\gamma_2)$ is an horizontal splitting pair. The case $N=3$ is left to the reader. Lemma \ref{LemmaExistenceSplittingPairs} is proved.

\end{document}